\documentclass{amsart}

\usepackage{amsthm,amsfonts,amssymb,euscript,verbatim,bbm}

\usepackage{graphics}
\usepackage{enumerate}

\usepackage{color}

\newtheorem{theorem}{Theorem}[section]
\newtheorem{proposition}[theorem]{Proposition}
\newtheorem{lemma}[theorem]{Lemma}
\newtheorem{remark}[theorem]{Remark}

\newcommand{\eqdef}{\overset{\mbox{\tiny{def}}}{=}}
\newcommand{\nuD}{D}
\newcommand{\DnuD}{\nuD-1}
\newcommand{\numDnu}{\nuD \ge 3}
\newcommand{\Dinit}{\mathcal{E}_{0,\nuD}}
\newcommand{\enerD}{\mathcal{E}_{T,\nuD}}
\newcommand{\iterD}{\tilde{\mathcal{E}}_{T,\nuD-1}^{(n)}}
\newcommand{\iterND}{\tilde{\mathcal{E}}_{T,\nuD-1}^{(n-1)}}
\newcommand{\NotSure}{\mathcal{G}}

\newcommand{\pel}{p}

\newcommand{\pZ}{\pel_0}
\def\vh {\hat{\pel}}

\def\ls {\lesssim}
\def\om {\omega}
\def\th {\theta}
\def\rd {\partial}
\def\Rt {\mathbb R^{d_{\pel}}}
\def\ep {\epsilon}

\def\nab {\nabla}

\newcommand{\fowC}{\mathcal{F}}
\newcommand{\bakC}{\mathcal{B}}

\newcommand{\ba}{\begin{equation}}
\newcommand{\ea}{\end{equation}}

\newcommand{\bea}{\begin{eqnarray}}
\newcommand{\eea}{\end{eqnarray}}

\def\beaa{\begin{eqnarray*}}
\def\eeaa{\end{eqnarray*}}

\title[Strichartz estimates and moment bounds I: the $2$-D and $2\frac 12$-D cases]{Strichartz estimates and moment bounds for the relativistic Vlasov-Maxwell system I. The $2$-D and $2\frac 12$-D cases.}

\author{Jonathan Luk}
\address{Department of Mathematics, MIT, Cambridge, MA 02139 and Department of Mathematics, University of Pennsylvania, Philadelphia, PA 19104}
\email{jluk@math.mit.edu}
\thanks{J.L. was partially supported by the NSF Postdoctoral Fellowship DMS-1204493.}

\author{Robert M. Strain}
\address{Department of Mathematics, University of Pennsylvania, Philadelphia, PA 19104}
\email{strain@math.upenn.edu}
\thanks{R.M.S. was partially supported by the NSF grant DMS-1200747.}

\begin{document}

\begin{abstract}
Consider the relativistic Vlasov-Maxwell system with initial data of unrestricted size.   In the two dimensional  and the two and a half dimensional cases, Glassey-Schaeffer proved \cite{GS2.5D}, \cite{GS2D1}, \cite{GS2D2} that for regular initial data with compact momentum support this system has unique global in time classical solutions. In this work we do not assume compact momentum support for the initial data and instead require only that the data have polynomial decay in momentum space. In the 2D and the $2\frac 12$D cases, we prove the global existence, uniqueness and regularity for solutions arising from this class of initial data. To this end we use Strichartz estimates and prove that suitable moments of the solution remain bounded. Moreover, we obtain a slight improvement of the temporal growth of the $L^\infty_x$ norms of the electromagnetic fields compared to \cite{GS2.5D} and \cite{GS2D2}.
\end{abstract}

\maketitle

\section{Introduction}

The relativistic Vlasov-Maxwell system describes the dynamics of a collisionless plasma. In this paper, we consider the initial value problem for the relativistic Vlasov-Maxwell system in two spatial dimensions. We first consider the case where the motion of the particles is confined within two dimensions in the spatial and the momentum variables.  We also study the situation when the particle density depends only on two spatial dimensions but particles are allowed to have momenta taking values in $\mathbb R^3$. The latter model can also be thought of as the three dimensional relativistic Vlasov-Maxwell system under a translational symmetry. We will refer to the first case as the two dimensional (2D) case  and the second case as the two-and-one-half dimensional ($2\frac 12$D) case.  More precisely, let the particle density $f:\mathbb R_t \times \mathbb R_x^2\times \mathbb R_\pel^{d_\pel} \to \mathbb R_+$ be a non-negative function of time $t\in \mathbb R$, position $x\in \mathbb R^2$ and momentum $\pel\in\mathbb R^{d_\pel}$.  We take $d_\pel=2$ in the two dimensional case and $d_\pel=3$ in the two-and-one-half dimensional case. 

The particles are subjected to the electromagnetic forces $E$ and $B$. In the $d_\pel=2$ case, we take $E:\mathbb R\times \mathbb R^2 \to \mathbb R^2$, $B:\mathbb R\times \mathbb R^2\to \mathbb R$ with
\begin{equation}\label{2D.ansatz}
E=(E^1(t,x_1,x_2),E^2(t,x_1,x_2),0),\quad B=(0,0,B(t,x_1,x_2)),
\end{equation}
while for $d_\pel=3$, we take $E,\,B:\mathbb R\times \mathbb R^2 \to \mathbb R^3$ with
\begin{equation}\label{25D.ansatz}
\begin{split}
E=(E^1(t,x_1,x_2),E^2(t,x_1,x_2),E^3(t,x_1,x_2)),
\\
B=(B^1(t,x_1,x_2),B^2(t,x_1,x_2),B^3(t,x_1,x_2)).
\end{split}
\end{equation}
Note that both reductions, either \eqref{2D.ansatz} or \eqref{25D.ansatz}, are propagated in time by regular solutions to the 
relativistic Vlasov-Maxwell system.

The relativistic Vlasov-Maxwell system can be written as
\bea
& &\rd_t f+\vh\cdot\nabla_x f+ (E+\vh\times B)\cdot \nabla_\pel f = 0,\label{vlasov}\\
& &\rd_t E= \nabla_x \times B- j,\quad \rd_t B=-\nabla_x\times E,\label{maxwell}\\
& &\nabla_x\cdot E=\rho,\quad \nabla_x \cdot B=0.\label{constraints}
\eea
where the charge is
$$\rho(t,x) \eqdef 4\pi\int_{\Rt} f(t,x,\pel) d\pel,$$
and the current is given by
$$
j_i(t,x) \eqdef  4\pi \int_{\Rt} \vh_i f(t,x,\pel) d\pel, \quad i=1,..., d_\pel.
$$
Here,
\bea
\vh=\frac{\pel}{\pel_0}, \quad \pel_0=\sqrt{1+|\pel|^2}.\label{vh.def}
\eea
In the above, we have used $\cdot$ and $\times$ to denote the $3$-dimensional dot product and cross product respectively. To make sense of them, we identity any vector field $Y=(Y^1,Y^2)\in \mathbb R^2$ with the vector field $(Y^1,Y^2,0)\in \mathbb R^3$.

We study the initial value problem for the relativistic Vlasov-Maxwell system, i.e., we prescribe initial data
$$(f,E,B)|_{t=0}=(f_0,E_0,B_0) $$
for some $(f_0,E_0,B_0)$ verifying the constraint equations \eqref{constraints}. Notice that given sufficiently regular initial data $f_0$, $E_0$, $B_0$ which satisfy \eqref{constraints}, then the equations \eqref{constraints} are propagated by the evolution equations \eqref{vlasov} and \eqref{maxwell} as long as the solution remains regular. 

According to the Vlasov equation \eqref{vlasov}, the particle density $f$ is transported along the characteristics $(X(t),V(t))$, which verify the ordinary differential equations:
\bea\label{char1}
\frac{d {X}}{ds}(s;t,x,\pel)=\hat{V}(s;t,x,\pel),
\eea
\bea\label{char2}
\frac{dV}{ds}(s;t,x,\pel)= E(s,X(s;t,x,\pel))+\hat{V}(s;t,x,\pel)\times B(s,X(s;t,x,\pel)),
\eea
together with the conditions
\bea
X(t;t,x,\pel)=x,\quad V(t;t,x,\pel)=\pel,\label{char.data}
\eea
where $\hat{V}\eqdef \frac{V}{\sqrt{1+|V|^2}}$.  We will explicitly estimate the derivatives of the characteristics in Sections \ref{2D.sec.global} and \ref{2hD.sec}.

In both the two dimensional and two-and-one-half dimensional cases, global existence and uniqueness of classical solutions was proved by Glassey-Schaeffer \cite{GS2.5D}, \cite{GS2D1}, \cite{GS2D2} for sufficiently regular data \emph{with compact momentum support}. They proved that for such solutions, the momentum support remains compact for all time and this is sufficient to guarantee that the solution remains $C^1$. We will roughly summarize their results in the following theorem.\footnote{We remark that the precise statements in \cite{GS2.5D}, \cite{GS2D1}, \cite{GS2D2} only require slightly weaker assumptions and also provide bounds for the solutions. We refer the readers to \cite{GS2.5D}, \cite{GS2D1}, \cite{GS2D2} for the precise orginial statements.}

\begin{theorem}[Glassey-Schaeffer \cite{GS2.5D}, \cite{GS2D1}, \cite{GS2D2}]\label{gs.thm}
Consider the following Cauchy initial data set $(f_0(x,\pel), E_0(x), B_0(x))$ which satisfies the constraints \eqref{constraints}; further suppose initially that \eqref{2D.ansatz} is satisfied when $d_\pel = 2$ or \eqref{25D.ansatz} is satisfied  when  $d_\pel = 3$.
Let\footnote{We denote  the space of continuously differentiable functions which are uniformly bounded with uniformly bounded first partial derivatives by $C^1_b(\mathbb R^2_x\times \mathbb R^{d_\pel}_\pel;\mathbb R)$. Similarly, $C^k_b(\mathbb R^2_x;\mathbb R^2)$ is the space of vector fields which are continuous and uniformly bounded, with continuous and uniformly bounded $k$-th partial derivatives.} $f_0\in C^1_b(\mathbb R^2_x\times \mathbb R^{d_\pel}_\pel;\mathbb R)$ with $f_0\geq 0$ when either $d_\pel = 2$ or $d_\pel = 3$  and $E_0, B_0 \in C^2_b(\mathbb R^2_x;\mathbb R^3)$. Moreover, suppose that the initial data have finite energy 
\begin{equation}\label{f.energy.2D}
\frac{1}{2}\int_{\mathbb R^2} (|E_0|^2+|B_0|^2) dx+4\pi\int_{\mathbb R^2} \int_{\mathbb R^{d_\pel}} \pel_0 f_0 d\pel dx<\infty,
\end{equation}
and the initial momentum support for $f_0$ is compact, i.e.,
\begin{equation}\label{cpt.supp.2D}
\sup\{|\pel|:f_0(x,\pel)\neq 0\}< \infty.
\end{equation}
In the case of $d_{\pel}=3$, assume in addition that 
$$\nab_x\times A_0=B_0$$
for some $A_0\in C^3_b(\mathbb R^2_x;\mathbb R^3_x)$.
Then there exists a unique global in time $C^1$ solution to the relativistic Vlasov-Maxwell system.
\end{theorem}

In this paper, we extend the global existence and uniqueness theorems of Glassey-Schaeffer to include initial data with non-compact momentum support. Our method is based on combining the moment estimates for the Vlasov equation with the Strichartz estimates for the wave equation. Using these estimates, we only require that the initial data for $f$ and $\nabla_{x,\pel} f$ decay polynomially in $\pel_0$ in an averaged sense. As a consequence of our approach, we also obtain an improved bound on the asymptotic growth of the solution. We refer the readers to Section \ref{sec.main.results} for the precise statements of the theorems.

Finally, we note that the methods developed here can also be applied to the three dimensional relativistic Vlasov-Maxwell system to obtain new continuation criteria which improves over the results in \cite{GS86}, \cite{BGP}, \cite{KS}, \cite{Pallard} and \cite{AI}. This will be carried out in Part II \cite{LS} of the present work.

\subsection{Notation}\label{sec.notation}
Let $\nab_x$ be the vector $(\frac{\partial}{\partial x^1}, \frac{\partial}{\partial x^2})$ and $\nab_\pel$ either be the vector $(\frac{\partial}{\partial \pel^1}, \frac{\partial}{\partial \pel^2})$ (in the $2$ dimensional case) or the vector $(\frac{\partial}{\partial \pel^1}, \frac{\partial}{\partial \pel^2}, \frac{\partial}{\partial \pel^3})$ (in the $2\frac 12$ dimensional case).  We will employ the notation
$$
|\nab_x g|^2\eqdef \left(\frac{\partial g}{\partial x^1}\right)^2+\left(\frac{\partial g}{\partial x^2}\right)^2.
$$
then $|\nab_{\pel} g|$ is defined similarly except for taking into account the dimension of the momentum space:
$$|\nab_\pel g|^2\eqdef \left(\frac{\partial g}{\partial \pel^1}\right)^2+\left(\frac{\partial g}{\partial \pel^2}\right)^2\quad\mbox{in }2\mbox{ dimensions},$$
$$|\nab_\pel g|^2\eqdef\left(\frac{\partial g}{\partial \pel^1}\right)^2+\left(\frac{\partial g}{\partial \pel^2}\right)^2+\left(\frac{\partial g}{\partial \pel^3}\right)^2\quad\mbox{in }2\frac 12\mbox{ dimensions}.$$
For an integer $k$,
we will use the notation $\nab_{x,\pel}^k$ schematically to denote 
$$
\nab_{x,\pel}^k g\eqdef
\left(  \partial_{x}^\alpha \partial_{\pel}^\beta g \right)_{|\alpha |+|\beta |=k},
$$
where $\partial_{x}^\alpha \eqdef \partial_{x^1}^{\alpha_1}\partial_{x^2}^{\alpha_2}$, and   $\partial_{\pel}^\beta \eqdef \partial_{\pel^1}^{\beta_1}\partial_{\pel^2}^{\beta_2}$ in the $2$D case or $\partial_{\pel}^\beta \eqdef \partial_{\pel^1}^{\beta_1}\partial_{\pel^2}^{\beta_2}\partial_{\pel^3}^{\beta_3}$ in the 
 $2\frac{1}{2}$D case.  Here  $\alpha = (\alpha_1,\alpha_2)$ and $\beta =(\beta_1,\ldots,\beta_{d_{\pel}})$ are standard multi-indicies.  The notation above denotes a vector which contains all components that satisfy the condition $|\alpha |+|\beta |=k$.    Then $\nab_{x}^k$ and $\nab_{\pel}^k$ are defined similarly with only the $x$ or $p$ derivatives respectively.  We further use $|\nab_{x,\pel}^k g|^2$ to denote the square sum of all $k$-th order derivatives:
$$
|\nab_{x,\pel}^k g|^2\eqdef \sum_{|\alpha |+|\beta |=k}
\left( \partial_{x}^\alpha \partial_{\pel}^\beta g\right)^2.
$$
Again $|\nab_{x}^k g|$ and $|\nab_{\pel}^k g|$ are defined similarly.

For a scalar function $g$, we then define the Lebesgue spaces
$$
\|g\|_{L^s([0,T);L^q_x L^r_{\pel})}\eqdef \left(\int_0^T (\int_{\mathbb R^2} (\int_{\mathbb R^{d_{\pel}}} |g|^r \,d\pel)^{\frac qr}\,dx)^{\frac sq} \,dt \right)^{\frac 1s} 
$$
with obvious modifications when $s$, $q$ or $r=\infty$.  Furthermore, for a vector valued function $G=(G_1, \ldots,G_m)$, we  define the Lebesgue spaces in exactly the same way except now
$
| G |^2  \eqdef \sum_{i=1}^m \left| G_i \right|^2
$
in the above definition.   We further define the Sobolev spaces, in either the vector or the scalar valued case, for
$H^D_x = H^D(dx)= H^D(\mathbb{R}^2_x)$ by
$$
\|g\|_{H^D_x}^2\eqdef \sum_{0\leq k\leq D}\int_{\mathbb R^2_x} |\nab_{x}^k g|^2\, dx,
$$
for an integer $D \ge 0$.  Notice that for a vector-valued function $G=(G_1, \ldots,G_m)$  we use the convention  that $\nab^k G$ is itself a vector that contains all derivatives of order $k$ of all components of the vector $G$.

We will find it convenient to use the following notation for a momentum weight
\begin{equation}\label{weight.notation}
w_{d_\pel}(p) \eqdef \pel_0^{{d_\pel}/2}\log(1+\pel_0) \quad ({d_\pel}=2,3).
\end{equation}
We define the weighted Sobolev space $H^D(w_{d_\pel}(\pel)^2 \,d\pel\,dx) = 
H^D(w_{d_\pel}(\pel)^2 \mathbb{R}^{2}_x \times \mathbb{R}^{d_\pel}_\pel)$ by
$$\|g\|_{H^D(w_{d_\pel}^2(\pel)\, d\pel\, dx)}^2\eqdef \sum_{0\leq k\leq D}\int_{\mathbb R^2_x}\int_{\mathbb R^{d_{\pel}}_{\pel}} |\nab_{x,\pel}^k g|^2 w_{d_\pel}^2(p)\, d\pel\, dx.$$
The space $L^\infty([0,T);H^D(w_{d_\pel}^2(\pel)\, d\pel\, dx))$ is then defined by making the suitable standard modifications.

For the electromagnetic fields, we will frequently use the notations $K=(E, B)$ and initially $K_0=(E_0, B_0)$.  For the Lorentz force we sometimes further use the notation $\tilde{K} \eqdef E+\vh\times B$.

We will use the notation
$
\langle w \rangle \eqdef \sqrt{1+|w|^2}
$
for a vector $w$.  For instance we use
$
\langle \pel_3 \rangle \eqdef \sqrt{1+\pel_3^2}
$
and
$
\langle (\pel_1,\pel_2) \rangle \eqdef \sqrt{1+\pel_1^2+\pel_2^2}.
$

Furthermore, we will use the notation $A \ls B$ to mean that $A\le CB$ where the implicit constant $(C\ge 0)$ may depend on any of the conserved quantities in the conservation laws in Section \ref{sec.cons.law}. In some sections below, we may slightly alter this notation in a way to be made precise in the beginning of the sections.

\subsection{Main Results}\label{sec.main.results}

Our main result in this paper for the $2$-dimensional relativistic Vlasov-Maxwell system is the following global existence and uniqueness theorem:

\begin{theorem}\label{main.theorem.2D}
Given Cauchy initial data $(f_0(x,\pel),E_0(x),B_0(x))$ to the $2$D relativistic Vlasov-Maxwell system which satisfies \eqref{2D.ansatz} and the constraints \eqref{constraints} such that $f_0\in C^1(\mathbb R^2_x\times\mathbb R^2_\pel)$ is non-negative and obeys the bound
\bea\label{ini.bd.2D.2}
\|f_0 \pel_0^N\|_{L^1_x L^1_\pel}<\infty,\quad\mbox{for some }N>13.
\eea
Additionally suppose that $\numDnu$ and
\bea\label{ini.bd.2D.2.5}
\sum_{0\leq k\leq \nuD}\| \left( \nab^k_{x,\pel} f_0 \right) w_2\|_{L^2_x L^2_\pel}<\infty.
\eea
Furthermore, for every $R>0$, we suppose that for some $C_R<\infty$,
\bea\label{ini.bd.2D.3}
\left\|
\int_{\mathbb R^2} \sup\{f_0(x+y,\pel+w) \pel_0^3:\,|y|\leq R, |w|\leq R\}\, d\pel
\right\|_{L^\infty_x} 
\leq C_R,
\eea
\bea\label{ini.bd.2D.4}
\left\|
\int_{\mathbb R^2} \sup\{|\nabla_{x,\pel} f_0|(x+y,\pel+w) \pel_0^3:\,|y|\leq R, |w|\leq R\}\, d\pel
\right\|_{L^\infty_x} 
\leq C_R,
\eea
and
\bea\label{ini.bd.2D.5}
\left\|
\sup\{|\nabla_{x,\pel} f_0|(x+y,\pel+w) :\,|y|\leq R,|w|\leq R\}
\right\|_{L^\infty_x L^\infty_{\pel}}  
\leq C_R.
\eea
Also the initial electromagnetic fields $E_0, B_0 \in H^{\nuD}(\mathbb R^2_x)$ obey the bounds
\bea\label{ini.be.2D.6}
\sum_{0\leq k\leq \nuD}(\|\nab_x^k E_0 \|_{L^2_x}+\|\nab_x^k B_0 \|_{L^2_x})<\infty.
\eea
Then there exists a unique global in time solution to the relativistic Vlasov-Maxwell system $(f,E,B)$.  The solution satisfies $f\in L^\infty([0,T); H^{\nuD}(w_2^2(\pel)\,d\pel\,dx))$ 
and the fields satisfy $E, B \in L^\infty([0,T); H^{\nuD}_x)$ for any large $T>0$.  
Moreover, there exist positive constants $C$ and $k$ which do not depend upon $T$ such that 
$$\| E(t)\|_{L^\infty_x}+\| B(t)\|_{L^\infty_x}\leq C e^{Ct^k}.$$
\end{theorem}

\begin{remark}
It is easy to see that the class of initial data considered in Theorem \ref{main.theorem.2D} has finite energy, i.e.,
\eqref{f.energy.2D} is satisfied for $d_\pel =2$.  This will allow us to use the conservation laws in Propositions \ref{cons.law.1} and \ref{cons.law.2} in the proof.
\end{remark}

\begin{remark}\label{GS.growth.rem}
The Glassey-Schaeffer theorem also gives an a priori control on the growth of the electomagnetic fields and the momentum support of the particle density. In particular, the estimates of Glassey-Schaeffer (\cite{GS2D1}, p.372-373 in \cite{GS2D2}, p.283 in \cite{GS2.5D})  imply the bounds of the form
$$\|E(t)\|_{L^\infty_x}+\|B(t)\|_{L^\infty_x}\leq C e^{e^{Ct}},$$
for some constant $C>0$. Our main theorem improves over this estimate.
\end{remark}

\begin{remark}
Notice that the assumptions \eqref{ini.bd.2D.2}, \eqref{ini.bd.2D.3} - \eqref{ini.bd.2D.5} are verified if the initial data for $f_0$ satisfy the pointwise bounds
$$|f_0(x,\pel)|\leq C\pel_0^{-(16+\ep)},$$
and 
$$|\nab_{x,\pel} f_0 (x,\pel)|\leq C\pel_0^{-6}\log^{-2}(1+\pel_0).$$
The conditions here are easier to state, however the minimal assumptions \eqref{ini.bd.2D.2} - \eqref{ini.bd.2D.5} are the ones that we used in the proof of the theorem.
\end{remark}

We also prove a similar global existence and uniqueness result for in the $2\frac 12$ dimensional case. In addition to the analogous conditions as Theorem \ref{main.theorem.2D}, we need an extra assumption\footnote{Notice that in the assumptions of Theorem \ref{main.theorem.2D.2h}, we also have \eqref{ini.bd.2D.2h.5.5} which does not seem to have a counterpart in Theorem \ref{main.theorem.2D}. However, note that a corresponding assumption 
$$\left\| \int_{\mathbb R^2} \sup\{|\nabla_{x,\pel} f_0|^2(x+y,\pel+w) w_2^2(\pel):\,|y|\leq R, |w|\leq R\}\, d\pel
\right\|_{L^\infty_x}^{\frac 12} \leq C_R$$
is automatically implied by \eqref{ini.bd.2D.4} and \eqref{ini.bd.2D.5} and does not have to be explicitly stated.}, see \eqref{ini.bd.2D.2h.4} below,  in order to exploit the additional conservation law in this case.

\begin{theorem}\label{main.theorem.2D.2h}
Given the initial data $(f_0(x,\pel),E_0(x),B_0(x))$ to the $2\frac 12$D relativistic Vlasov-Maxwell system which satisfies \eqref{25D.ansatz} and the constraints \eqref{constraints} such that $f_0\in C^1(\mathbb R^2_x\times \mathbb R^{3}_\pel)$ is non-negative and obeys the following bound:
\bea\label{ini.bd.2D.2h.2}
\|f_0 \pel_0^N\|_{L^1_x L^1_\pel}<\infty,\quad\mbox{for some }N>13.
\eea
Additionally suppose that $\numDnu$ and
\bea\label{ini.bd.2D.2h.2.5}
\sum_{0\leq k\leq \nuD}\| \left( \nab^k_{x,\pel} f_0 \right) w_3\|_{L^2_x L^2_\pel}<\infty.
\eea
Furthermore, for every $R>0$, we suppose that for some $C_R<\infty$ we have
\bea\label{ini.bd.2D.2h.3}
\left\|
 \int_{\mathbb R^3} \sup\{f_0(x+y,\pel+w) \pel_0^3:\,|y|\leq R, |w|\leq R\}\, d\pel
 \right\|_{L^\infty_x} 
 \leq C_R,
\eea
\bea\label{ini.bd.2D.2h.5}
\left\| \int_{\mathbb R^3} \sup\{|\nabla_{x,\pel} f_0|(x+y,\pel+w) \pel_0^3:\,|y|\leq R, |w|\leq R\}\, d\pel
\right\|_{L^\infty_x} \leq C_R,
\eea
\begin{equation}
\label{ini.bd.2D.2h.5.5}
\left\| \int_{\mathbb R^3} \sup\{|\nabla_{x,\pel} f_0|^2(x+y,\pel+w) w_3^2(\pel):\,|y|\leq R, |w|\leq R\}\, d\pel
\right\|_{L^\infty_x} \leq C_R,
\end{equation}
\begin{equation}\label{ini.bd.2D.2h.6}
\left\|
\sup\{|\nabla_{x,\pel} f_0|(x+y,\pel+w) :\,|y|\leq R,|w|\leq R\}
\right\|_{L^\infty_x L^\infty_{\pel}}  \leq C_R,
\end{equation}
and
\begin{multline}\label{ini.bd.2D.2h.4}
\sup_{x,\pel_1,\pel_2} 
\int_{-\infty}^{\infty} \sup\{f_0(x+y,\pel_1,\pel_2,\pel_3+w) \langle \pel_3 \rangle^{5+\delta}:\, |y|\leq R,|w|\leq R\}\, d\pel_3 
\\
\leq C_R,
\end{multline}
for some $\delta>0$.
Also the initial electromagnetic fields $E_0, B_0 \in H^{\nuD}(\mathbb R^2_x)$ obey the bounds
\bea\label{ini.be.2D.2h.7}
\sum_{0\leq k\leq \nuD} \left(\|\nab_x^k E_0 \|_{L^2_x}+\|\nab_x^k B_0 \|_{L^2_x} \right)<\infty.
\eea
Then there exists a unique global in time solution to the relativistic Vlasov-Maxwell system $(f,E,B)$.  The solution satisfies $f\in L^\infty([0,T); H^{\nuD}(w_3^2(\pel)\,d\pel\,dx))$ 
and furthermore  $E, B \in L^\infty([0,T); H^{\nuD}_x)$ for any large $T>0$.  
Moreover, there exist positive constants $C$ and $k$ which do not depend upon $T$ such that 
\begin{equation}\label{growth.bd}
\| E(t)\|_{L^\infty_x}+\| B(t)\|_{L^\infty_x}\leq C e^{Ct^k}.
\end{equation}
\end{theorem}

\subsection{Strategy of proof}
In this section we will explain the main ideas in the proof. We will focus on the $2$-dimensional case and indicate at the end the modifications needed for the $2\frac 12$-dimensional case.

\subsubsection{Moment bounds}
We consider solutions to the relativistic Vlasov-Maxwell system with initial data having noncompact momentum support. Therefore, unlike in \cite{GS2D2}, \cite{GS2.5D}, we cannot control the solution by estimating the growth of the momentum support, which is initially infinite. Instead, we bound the moments\footnote{We make a very rough comparison to the known regularity results for the non-relativistic Vlasov-Poisson system. In that context, there are two approaches used in obtaining global classical solutions for large data, namely the Pfaffelmoser theory \cite{Pfaffelmoser} (see also work of Schaeffer \cite{Schaeffer}) that is based on controlling the momentum support and the Lions-Perthame theory \cite{LP} which is based on estimating the moments. For the 2D relativistic Vlasov-Maxwell system, the Glassey-Schaeffer theorems employ the philosophy of controlling the momentum support while our approach is closer to the spirit of estimating the moments.} of $f$. It is well-known that the $N$-th moment of $f$ can be controlled by the $L^1_t L^{N+2}_x$ norm of the Lorentz force $E+\vh\times B$. In particular, for $K$ denoting $(E,B)$, we have
\begin{equation}\label{moment.intro}
\|f\pel_0^N \|_{L^\infty_t([0,T); L^1_x L^1_{\pel})}\leq \mbox{Data term} +\| K\|_{L^1_t([0,T); L^{N+2}_x)}^{N+2}.
\end{equation}
Therefore in order to control the moments of the particle density $f$, we need to obtain $L^1_t L^q_x$ bounds for the electromagnetic fields $E$ and $B$. 

\subsubsection{Strichartz estimates}
In \cite{GS2D2}, $E$ and $B$ are estimated in $L^\infty_t L^\infty_x$ by $P\log P$, where $P$ is the bound for the momentum support. In the context of this paper, $P$ is infinite and the methods of \cite{GS2D2} do not allow one to obtain a sufficiently strong $L^\infty_t L^\infty_x$ control of $E$ and $B$ in terms of the moments of $f$. Instead, we need to make use of the fact that we only need to obtain $L^1_t L^q_x$ bounds for $E$ and $B$ in order to close the moment bounds \eqref{moment.intro}, i.e., we can take advantage of both the lower integrability in $x$ and the extra integration in $t$.

Since $E$ and $B$ both obey wave equations, such $L^1_t L^q_x$ bounds can be derived using the Strichartz estimates. However, if we apply Strichartz estimates directly using the wave equation that $E$ and $B$ satisfy, i.e.,
\begin{equation*}
\begin{split}
\Box E=& \nabla_x \rho+\rd_t j,\\
\Box B=& -\nabla_x\times j,
\end{split}
\end{equation*}
then there would be a loss of derivatives. Instead, we apply the Glassey-Schaeffer representations \cite{GS2D1}, \cite{GS2D2} of the electromagnetic fields (see Section \ref{sec.2D.GS}), which were devised to avoid a similar loss of derivatives in their setting.   We then perform some preliminary estimates controlling the singularity (both in the physical space variables and the momentum space variables) in the integral kernel of the Glassey-Schaeffer representations.  Afterwards we show that the electromagnetic field $K$ can be controlled by
\begin{multline}\label{Kbound}
|K|
\ls
 \tilde{K}_0+\Box^{-1}\left(|K|\int_{\mathbb R^2} \frac{f}{\pel_0} d\pel \right)\\
+\ep^{-\frac 1{10}}\left(\Box^{-1}\left(\int_{\mathbb R^2} \pel_0^2 f d\pel \right)\right)^{\frac 25}+\ep^{\frac 3{10}}\left(\Box^{-1}\left(\int_{\mathbb R^2} \pel_0^4 f d\pel \right)\right)^{\frac 25},
\end{multline}
for every $\ep\in [0,1)$. Here, $\tilde{K}_0$ depends only on the initial data and $\Box^{-1}F$ is defined to be the solution $u$ to the wave equation $\Box u=F$ with zero initial data. 

To obtain the bound \eqref{Kbound}, we need to control the terms
$$\int_0^t \int_{|y-x|\leq t-s}\int_{\mathbb R^2} \frac{f(s,y,\pel) ~ d\pel\, dy\, ds}{\pel_0^2(1+\vh\cdot\xi)^{\frac 32}(t-s)\sqrt{(t-s)^2-|y-x|^2}}$$
and
$$\int_0^t \int_{|y-x|\leq t-s}\int_{\mathbb R^2} \frac{(K_g f)(s,y,\pel)~ d\pel\, dy\, ds}{(1+\vh\cdot\xi)\pel_0\sqrt{(t-s)^2-|y-x|^2}},$$
where $K_g$ denotes the components of $E$ and $B$ that are $L^2$ integrable on the null cones using the conservation law (see Proposition \ref{cons.law.2} in Section \ref{sec.cons.law}).  The main challenge for these terms is to control the singularity $\frac 1{1+\vh\cdot\xi}$. Noticing that we have 
$$(1+\vh\cdot\xi)^{-1}\ls \min\left\{\left(\cos^{-1}(\frac{\vh\cdot\xi}{|\vh||\xi|})\right)^{-2},\, \frac{1}{1-|\xi|^2},\, \pel_0^2\right\},$$
we can control the above integrals by splitting into the regions according to the size of $\pel_0$ and $1-|\xi|^2$. This allows us to prove \eqref{Kbound}.

Once we obtain the estimate \eqref{Kbound}, we then apply Strichartz estimates\footnote{In fact we apply the improvement obtained by Foschi \cite{Foschi} which strengthens the usual Strichartz estimates in the case of zero initial data} to the right hand side of \eqref{Kbound} to obtain $L^q_t L^r_x$ control for $K$. Combining the Strichartz estimates with the moments bounds, using the conservation laws, and choosing $\ep$ to be an appropriate inverse power of the $N$-th moment of $f$, we derive that for $N>13$, we have
\begin{multline*}
\|f\pel_0^N \|_{L^\infty_t([0,T); L^1_x L^1_{\pel})}\leq \mbox{Data term}
 + \|f\pel_0^N \|_{L^\infty_t([0,T); L^1_x L^1_{\pel})}^{\alpha}\\
+ \|f\pel_0^N \|_{L^\infty_t([0,T); L^1_x L^1_{\pel})}^{\alpha'}\|f\pel_0^N \|_{L^{q_2'}_t([0,T); L^1_x L^1_{\pel})}^{1-\alpha'},
\end{multline*}
for some $0<\alpha=\alpha(N)$, $\alpha'=\alpha'(N)<1$ and $q_2'=q_2'(N)<\infty$. Therefore, we conclude that $\|f\pel_0^N \|_{L^\infty_t([0,T); L^1_x L^1_{\pel})}$ is bounded on every time interval $[0,T)$ after using a Gronwall-type inequality.

\subsubsection{Estimates for the first derivatives}
According to the local existence theorem we prove in Section \ref{sec.2D.local.existence}, the solution to the relativistic Vlasov-Maxwell equation can be continued if one can control $\|K\|_{L^1_t L^{\infty}_x}$, $\|\nab_x K\|_{L^1_t L^{\infty}_x}$ and $\|w_2 \nab_{x,\pel} f\|_{L^1_t L^{\infty}_x L^2_{\pel}}$. We therefore have to show that the moment bound we obtained is suffucient to control these norms. 

First, we note that using the moment bounds above, $\|K\|_{L^\infty_t L^{\infty}_x}$ can be easily controlledusing \eqref{Kbound}. This can be achieved by bounding the wave kernel in $L^1_t L^{2-\ep}_x$. 

It therefore remains to control the derivatives of $K$ and $f$. Here, we follow a strategy that is similar to \cite{GS2D1} except that since the momentum support is unbounded, it is insufficient to control the $\|\nab_{x,\pel} f\|_{L^1_t L^\infty_x L^\infty_{\pel}}$ as in \cite{GS2D1}. Instead, we will bound the derivatives of the characteristics (cf. \cite{KS}).  More precisely, we show the following two estimates. First, using the Glassey-Schaeffer decomposition for the derivatives of the electromagnetic field, the bound for $\|K\|_{L^\infty_t L^{\infty}_x}$ and Gronwall's inequality, we have
\begin{multline}\label{dk.bd.intro}
\|\nab_x K\|_{L^\infty_x}(t) \\
\ls 1+\int_0^t (1+\|\pel_0^3 f\|_{L^\infty_x L^1_{\pel}}(s)+\delta(t)\|\pel_0^3\nab_{x,\pel} f\|_{L^\infty_xL^1_{\pel}}(s))\left|\log \left(\frac 1{\delta(t)}\right)\right|\, ds,
\end{multline}
for any choice of the function $0<\delta(t)<t$. Second, since we have estimated $\|K\|_{L^\infty_t L^{\infty}_x}$, we can integrate from the initial data using the initial bounds \eqref{ini.bd.2D.3}, \eqref{ini.bd.2D.4} and \eqref{ini.bd.2D.5} to show that
$$\|\pel_0^3 f \|_{L^\infty_t L^\infty_x L^1_{\pel}}\ls 1$$ 
and
\begin{equation}\label{dfbd.intro}
\|\pel_0^3 \nab_{x,\pel} f \|_{L^\infty_x L^1_{\pel}}(t)+\|\nab_{x,\pel} f \|_{L^\infty_x L^\infty_{\pel}}(t)\ls 1+ \bakC(t),
\end{equation}
where for the ``backward characteristics'', $(X,V)(0; t, x, \pel)$, we define
$$\bakC(t)=(\mbox{first derivatives of the backward characteristics})(t).$$
On the other hand, integrating along the ``forward characteristics'', $(X,V)(t; 0, x, \pel)$,  we can show that the first derivatives of forward characteristics $\fowC(t)$ satisfy
$$\fowC(t)\ls 1+\int_0^t (1+\fowC(s))\|\nab_xK(s)\|_{L^\infty_x} ds.$$ 
Combining these estimates, choosing $\delta(t)=\frac{t}{1+\bakC(t)}$ and using the fact that $\bakC(t)$ can be controlled by $\fowC(t)$ polynomially, we have the bound
\begin{multline*}
\fowC(t)\ls 1+\int_0^t (1+\fowC(s))\log (1+\bakC(s)) ds
\\
\ls 1+\int_0^t (1+\fowC(s))\log (1+\fowC(s)) ds.
\end{multline*}
This allows us to show that $\fowC$ (and hence $\bakC$) is bounded on any fixed time interval. Moreover, by \eqref{dfbd.intro}, we have
$$\|\pel_0^3 \nab_{x,\pel} f \|_{L^\infty_x L^1_{\pel}}(t)+\|\nab_{x,\pel} f \|_{L^\infty_x L^\infty_{\pel}}(t)\ls 1.$$ 
Returning to the estimate for $\nab_x K$ in \eqref{dk.bd.intro} above, we can thus show that $\|\nab_x K\|_{L^\infty_x}$ is bounded on all intervals $[0,T)$. Finally, interpolating between the bounds for $\|\pel_0^3 \nab_{x,\pel} f \|_{L^\infty_x L^1_{\pel}}(t)$ and $\|\nab_{x,\pel} f \|_{L^\infty_x L^\infty_{\pel}}(t)$, we can also get the desired estimate for $\|w_2 \nab_{x,\pel} f\|_{L^1_t L^{\infty}_x L^2_{\pel}}$. Therefore, according to the local existence theorem, the solution can be extended globally.

\subsubsection{The two-and-one-half dimensional case}\label{2.5D.intro}
Applying only the above strategy is insufficient to close the estimates in the $2\frac 12$ dimensional case. Instead, we need to take advantage of an additional conservation law from \cite{GS2.5D}. The conservation law implies that on a fixed time interval, the $\pel_3$ difference of any two points on every characteristic is bounded. In \cite{GS2.5D}, since $\pel_3$ is initially bounded over the support of $f$, this conservation law was used to obtain an a priori bound for the supremum of $\pel_3$ on the support of $f$. In our setting we do not have such bounds available; we can nevertheless use this conservation law to show that 
\bea\label{2.5D.CL.bd}
\sup_{0\le t\leq T,x\in\mathbb R^2,\pel_1,\pel_2\in\mathbb R}\int_{-\infty}^{\infty} f(t,x,\pel)\langle\pel_3\rangle^{5+\delta} \,d\pel_3\ls 1.
\eea
As a result of this estimate, we use the $2\frac 12$-D Glassey-Schaeffer representation \cite{GS2.5D} of the electromagnetic field to show that $K$ obeys the following analogue of \eqref{Kbound}:
\begin{multline*}
|K|
\ls
 \tilde{K}_0+\Box^{-1}\left(|K|\int_{\mathbb R^3} \frac{f}{\pel_0} d\pel \right)\\
+\ep^{-\frac 1{10}}\left(\Box^{-1}\left(\int_{\mathbb R^3} \pel_0^2 f d\pel \right)\right)^{\frac 25}+\ep^{\frac 3{10}}\left(\Box^{-1}\left(\int_{\mathbb R^3} \pel_0^4 f d\pel \right)\right)^{\frac 25},
\end{multline*}
This is derived using a dyadic decomposition in $|\pel_3|$, so that for each dyadic region $|\pel_3|\in [2^n,2^{n+1})$, we can divide the region of integration according to the sizes of $\pel_0$ and $1-|\xi|^2$ as in the $2$-dimensional case.

In addition, we can also show that as a result of \eqref{2.5D.CL.bd}, $f$ obeys an interpolation inequality with the same exponents as in the $2$-dimensional case. This shows that all the estimates in the proof of the $2$-dimensional case can also be derived in this case and the rest of the proof of global existence and uniqueness of solutions proceeds analogously.

\subsection{Outline of the paper}

We end the introduction with an outline of the remainder of the paper. In Section \ref{sec.cons.law}, we will recall the conservation laws that the solution obeys. Then in Section's \ref{strichartz.sec} and \ref{moment.sec}, we state the Strichartz estimates and moment bounds that will be needed in the paper. In Section \ref{2D.sec.loc}, we will prove a local existence result in $2$ dimensions and a continuation criteria using the energy method in Sobolev spaces.  Then in Section \ref{2D.sec.global}, we will prove the global existence result in $2$ dimensions (Theorem \ref{main.theorem.2D}). Finally, in Section \ref{2hD.sec}, we will prove the global existence result in $2\frac 12$ dimensions (Theorem \ref{main.theorem.2D.2h})

\section{Conservation Laws}\label{sec.cons.law}

The solution to the relativistic Vlasov-Maxwell system (in either $2$ dimensions, $2\frac 12$ dimensions or three dimensions) obeys the following pointwise identity:
\bea
\frac{\rd}{\rd t} e=\sum_{k=1}^{d_x} \frac{\rd}{\rd x_k}\left(-(B\times E)_k + 4\pi\int_{\mathbb R^{d_\pel}} \pel_k f d\pel \right), \label{cons.id}
\eea
where the energy density $e$ is given by
$$e \eqdef \frac 12 (|E|^2+|B|^2)+4\pi\int_{\mathbb R^{d_\pel}} \pel_0 f d\pel,$$
with $\pel_0$ defined as before by \eqref{vh.def}.

The identity \eqref{cons.id} will be integrated on spacetime regions and this will yield conservation laws. 
We will derive two conservation laws from this identity that we will use throughout the rest of this paper. 
In the first case we integrate in the spacetime region bounded in the past by the initial slice $\{0\}\times \mathbb R^{2}$ and in the future by a constant time slice $\{t\}\times \mathbb R^{2}$. Since the initial energy is assumed to be bounded in Theorem's \ref{main.theorem.2D} and \ref{main.theorem.2D.2h}, we obtain the following proposition.

\begin{proposition}\label{cons.law.1}
Solutions to the relativistic Vlasov-Maxwell system \eqref{vlasov}-\eqref{constraints} satisfy
\beaa
\frac 12 \int_{\{t\}\times \mathbb R^{2}_x} (|E|^2+|B|^2)dx+4\pi\int_{\{t\}\times \mathbb R_x^{2}\times\mathbb R_{\pel}^{d_\pel}} \pel_0 f d\pel\, dx =\mbox{ constant}.
\eeaa
\end{proposition}
For the second conservation law, we need to control the flux of the electromagnetic field integrated along a backward null cone. To this end, we integrate \eqref{cons.id} in the spacetime region bounded in the past by the initial slice $\{0\}\times \mathbb R^{d_x}$ and in the future by the backward null cone $C_{t,x}$ emanating from $(t,x)$, which is defined to be the set
$$
C_{t,x} \eqdef \{(s,y)\in\mathbb R\times\mathbb R^{2} | 0\le s\leq t,\, t-s=|y-x|\}.
$$
For each point $(s,y)\in C_{t,x}$, denote by $\om$ the outward normal to the 1-sphere
$C_{t,x}\cap (\{s\}\times{\mathbb R^{2}})$, i.e., 
\bea
\om \eqdef \frac{y-x}{|y-x|}.\label{normal.def}
\eea

The volume form on $C_{t,x}$ can be given in polar coordinates by
\bea
\int_{C_{t,x}} f d\sigma = 
\int_0^t ds  \int_0^{2\pi}  d\th  ~ (t-s) ~  f(s,x+(t-s)\om), \label{cone.vol.form.2D}
\eea
where $\om$ takes the form
\bea
\om = (\cos\th,\sin\th)\label{normal.coord.2D}
\eea
in this coordinate system.

We compute the flux of the electromagnetic field on the null cone, i.e., the boundary term $C_{t,x}$ arising from integrating \eqref{cons.id} by parts. We observe that it is non-negative and moreover it controls certain components of $E$ and $B$.  We first identify $\om$ with the $3$ dimensional vector $(\cos\th,\sin\th,0)$ and perform a computation in $\mathbb R^3$. Notice\footnote{Here we will use the vector identity
$
a\cdot (b\times c) = b\cdot (c\times a) = c\cdot (a \times b)
$ for $a$, $b$, $c$ three vectors.}
 that
\begin{multline} \notag
\frac 12 \left(|E|^2+|B|^2\right)+\om\cdot\left(E\times B\right)
\\
=\frac 18 \left(2|E\cdot\om|^2+2|B\cdot\om|^2+|E+\om\times B|^2+|E-\om\times B|^2
+|B+\om\times E|^2\right.
\\
\left.+|B-\om\times E|^2-4(\om\times B)\cdot E+ 4(\om\times E)\cdot B\right)
\\
=\frac 14 \left(|E\cdot\om|^2+|B\cdot\om|^2+|E-\om\times B|^2+|B+\om\times E|^2\right).
\end{multline}
We will use this to prove the next proposition.  In the $2\frac 12$ dimensional case, we will define the good component of the electro-magnetic field, $K_g$ to be the component controlled by the flux term:
\begin{equation}\label{good.comp.2hD}
K_g^2 = \left(|E\cdot\om|^2+|B\cdot\om|^2+|E-\om\times B|^2+|B+\om\times E|^2\right).
\end{equation}
In the $d_\pel=2$ case, since $B\cdot \om=0$, we have
$|B\cdot \om|^2=0$
and
\begin{multline*} 
|E-\om\times B|^2=|\om\cdot E-\om\cdot(\om\times B)|^2+|\om\times E-\om\times(\om\times B)|^2
\\
=|E\cdot \om|^2+|B+\om\times E|^2.
\end{multline*} 
Therefore, the flux can be re-expressed as
\begin{equation*} 
\frac 12 \left(|E|^2+|B|^2\right)+\om\cdot\left(E\times B\right)
=\frac 12 \left(|E\cdot\om|^2+|B+\om\times E|^2\right).
\end{equation*}
We now define the good components of the electro-magnetic field, $K_g$, in the $2$ dimensional case to be:
\begin{equation}\label{good.comp.2D}
K_g^2 = 2\left(|E\cdot\om|^2+|B+\om\times E|^2\right).
\end{equation}
Notice that in both the $2$-D and $2\frac 12$-D cases, $K_g$ does not contain all the components of $E$ and $B$.   However, as we will see in the later sections, the following conservation laws will still be useful in controlling the most singular terms.  

\begin{proposition}\label{cons.law.2}  
Solutions to the relativistic Vlasov-Maxwell system \eqref{vlasov}-\eqref{constraints} satisfy
$$
 \frac 14 \int_{C_{t,x}} K_g^2 ~ d\sigma 
+4\pi \int_{C_{t,x}} \int_{\Rt}  \pZ(1+\vh\cdot\om) f d\pel d\sigma 
 \leq \mbox{ constant},
$$
where  $K_g$ is defined by \eqref{good.comp.2hD} and \eqref{good.comp.2D} in the $2\frac12$-D and $2$-D cases respectively.
\end{proposition}

Finally, we also need the conservation law for the $L^\infty$ norm of the particle density $f$. This follows from integrating $f$ along the characteristics given by \eqref{char1} and \eqref{char2}. More precisely, we have
\begin{proposition}\label{cons.law.3}
$
 \|f\|_{L^\infty_t L^\infty_x L^\infty_\pel }\leq\mbox{ constant}.
$
\end{proposition}

\section{Strichartz estimates}\label{strichartz.sec}

As mentioned above, we will need the Strichartz estimates for the linear wave equation. These estimates have been extensively studied (see \cite{Strichartz}, \cite{Kapitanski}, \cite{KeelTao}). The particular version that we will use, which is due to Foschi \cite{Foschi}, takes advantage of an improvement for the linear inhomogeneous wave equation with $0$ data.

\begin{theorem}[Strichartz estimates, Foschi \cite{Foschi} (see also
Taggart \cite{Taggart})]\label{Strichartz}
Let $u$ be a solution to the linear inhomogeneous wave equation in
$\mathbb R^2$:
$$
\Box u= F,\quad u(0,x)=0,\quad \frac{\partial u}{\partial t}(0,x)=0.
$$
Then, the following estimates hold
$$\|u \|_{L^{q_1}_t([0,T); L^{r_1}_x)}\ls \|F\|_{L^{q'_2}_t([0,T);
L^{r'_2}_x)},$$
where 
$$
\frac 1{q_1}+\frac {2}{r_1}=\frac 1{q'_2}+\frac {2}{r'_2}-2,
$$
$$
\frac 1{q_1}< \frac 12-\frac 1{r_1},
\quad  \frac 32-\frac 1{r_2'} <\frac 1{q_2'},
$$
$$
\frac 13\leq\frac{1}{r_1}+\frac 1{r_2}< \frac 12,
$$
and\footnote{We point out that the upper bound $\frac{1}{r_1}+\frac 1{r_2}< \frac 12$ is redundant.}
$$1\leq q_1,q_2< \infty,\quad 2\leq r_1,r_2\leq\infty .$$
Here $q_2'$ denotes the usual H{\"o}lder conjugate exponent to $q_2$,
$\frac{1}{q_2} +\frac{1}{q_2'}=1$, etc.  
\end{theorem}

\begin{remark}
The above theorem is a special case of Corollary 8.7 in \cite{Taggart}. We
specialize to the case where 1) the spatial dimension is $2$; 2) the
initial data are zero; and 3) we only use $L^q$ spaces (as opposed to more
general Besov spaces). The original conditions are that there exist
$r_{1,*}$ and $r_{2,*}$ such that the following hold:
$$r_1\geq r_{1,*},\quad r_2\geq r_{2,*},$$
$$\frac{1}{q_1}+\frac{1}{q_2}=\frac 12(1-\frac{1}{r_{1,*}}-\frac
1{r_{2,*}}),$$
$$2(\frac 12-\frac 1{r_1})-\frac 1{q_1}=1-(2(\frac 12- \frac
1{r_2})-\frac{1}{q_2}),$$
$$
\frac 1{q_1}< \frac 12-\frac 1{r_{1,*}}, \quad \frac 1{q_2}< \frac
12-\frac 1{r_{2,*}},$$
$$1\leq q_1,q_2< \infty,\quad 2\leq r_{1,*},r_{2,*}\leq\infty.$$
If the conditions in Theorem \ref{Strichartz} hold, then we can define
$r_{1,*}$ and $r_{2,*}$ by
$$\frac{1}{r_{1,*}}=\frac{1}{r_1}-\frac{1-3(\frac 1{r_1}+\frac
1{r_2})}{1+\frac{\frac 12-\frac 1{q_2}-\frac 1{r_2}}{\frac 12-\frac
1{q_1}-\frac 1{r_1}}}$$ 
and 
$$\frac 1{r_{2,*}}=\frac{1}{r_1}-\frac{1-3(\frac 1{r_1}+\frac
1{r_2})}{1+\frac{\frac 12-\frac 1{q_1}-\frac 1{r_1}}{\frac 12-\frac
1{q_2}-\frac 1{r_2}}}$$ 
such that the above conditions hold.
\end{remark}

\begin{remark}\label{rem.strichartz}
Notice that the solution to the linear inhomogeneous wave equation in $\mathbb R^2$ with zero initial data is given explicitly by
$$u(t,x)=\int_0^t \int_{|y-x|\leq t-s} \frac{F(s,y)}{\sqrt{(t-s)^2-|y-x|^2}} dy\, ds.$$
Thus the above Strichartz estimates can be rephrased as
$$
\left\|\int_0^t \int_{|y-x|\leq t-s} \frac{F(s,y)}{\sqrt{(t-s)^2-|y-x|^2}} dy\, ds \right\|_{L^{q_1}_t([0,T); L^{r_1}_x)}\ls \|F\|_{L^{q'_2}_t([0,T); L^{r'_2}_x)}. 
$$
This will be the precise form of the estimate that we use in the following sections.
\end{remark}

\section{Moment estimates}\label{moment.sec}

While the rest of this paper is mostly concerned with the domains $\mathbb{R}^2_x \times \mathbb{R}^2_\pel$ or $\mathbb{R}^2_x \times \mathbb{R}^3_\pel$, in this section we prove some interpolation inequalities and moment estimates in the more general case $\mathbb{R}^{d_x}_x \times \mathbb{R}^{d_p}_\pel$ for ${d_x}, {d_p} \ge 1$.   We also work here with a general function $g=g(x,p)\in { L^\infty_x L^\infty_\pel}$ (considering Proposition \ref{cons.law.3}).

We first prove the following standard interpolation inequality

\begin{proposition}[General interpolation inequality]\label{prop.interpolation.0}
Consider $\mathbb{R}^{d_x}_x \times \mathbb{R}^{d_p}_\pel$ with $d_x$, ${d_p} \geq 1$.  Suppose that $1\leq q<\infty$ and $M\geq S>-d_\pel$.  Then we have:
$$
\| \pel_0^S g \|_{ L^q_x L^1_\pel}
\ls
\| \pel_0^M g \|_{ L^{\frac{(S+d_\pel)}{M+d_\pel}q}_x L^1_\pel}^{\frac{S+d_\pel}{M+d_\pel}}.
$$
Above the implied constant depends only on $\|  g \|_{ L^\infty_x L^\infty_\pel}$.
\end{proposition}

We remark that this proposition does not require $d_x = {d_p}$.

\begin{proof}
We divide the domain of integration in the $| \pel |$ variable into $| \pel |\leq R$ and 
$| \pel |> R$ for some $R\ge 1$.   We can assume without loss of generality that $g\ge 0$.  
Since $\|  g \|_{ L^\infty_x L^\infty_\pel}<\infty$, we have
$$
\int_{| \pel |\leq R} \pel_0^S g(x,\pel) d\pel \ls  \|  g \|_{ L^\infty_x L^\infty_\pel} R^{S+d_\pel}.
$$
For $| \pel |> R$, we have
$$
\int_{| \pel |> R} \pel_0^S g(x,\pel) d\pel 
\leq R^{-(M-S)}\int_{\mathbb{R}^{d_p}_\pel} \pel_0^M g(x,\pel) d\pel.
$$
We choose $R=\big(\int \pel_0^M g(x,\pel) d\pel\big)^{\frac{1}{M+d_\pel}}$ when this quantity is $\ge 1$ to obtain
$$
\int_{\mathbb{R}^{d_p}_\pel} \pel_0^S g(x,\pel) d\pel \ls \left(\int_{\mathbb{R}^{d_p}_\pel} \pel_0^M g(x,\pel) d\pel\right)^{\frac{S+d_\pel}{M+d_\pel}}.
$$
Notice that this inequality is further trivially satisfied when our choice of $R$ satisfies $R \le 1$.
We take the $L^q_x$ of both sides above to achieve the desired inequality.
\end{proof}

We will only use the special case $q=\frac{M+d_\pel}{S+d_\pel}$ in the following sections:

\begin{proposition}[Interpolation inequality]\label{prop.interpolation}
Consider $\mathbb{R}^{d_x}_x \times \mathbb{R}^{d_p}_\pel$ with $d_x$, ${d_p} \geq 1$.  
Suppose $M\geq S>-d_\pel$, then the following estimate holds:
$$\| \pel_0^S g \|_{ L^{\frac{M+d_\pel}{S+d_\pel}}_x L^1_\pel}\ls\| \pel_0^M g \|_{ L^1_x L^1_\pel}^{\frac{S+d_\pel}{M+d_\pel}}.$$
Again the implied constant depends only on $\|  g \|_{ L^\infty_x L^\infty_\pel}$.
\end{proposition}

We record the following standard moment estimates (see Lions-Perthame \cite{LP}).

\begin{proposition}[Moment estimate]\label{prop.moment}  For $N> 0$ we have the estimate
$$
\| \pel_0^N f \|_{L^\infty_t([0,T); L^1_x L^1_\pel)}^{\frac 1{N+d_\pel}}
\ls 
\| \pel_0^N f_0 \|_{L^1_x L^1_\pel}^{\frac 1{N+d_\pel}}+\| K\|_{L^1_t([0,T); L^{N+d_\pel}_x)}
$$
\end{proposition}

\begin{proof}
Differentiating the $N$-th moment of \eqref{vlasov} in time and integrating by parts in $x$ and $\pel$, we get
$$
\frac{d}{dt}\| \pel_0^N f(t) \|_{ L^1_x L^1_\pel}\ls \| \pel_0^{N-1} f(t) |K(t)|\|_{ L^1_x L^1_\pel}.$$
Then H\"older's inequality implies that
$$\frac{d}{dt}\| \pel_0^N f(t) \|_{ L^1_x L^1_\pel}\ls \| \pel_0^{N-1} f(t) \|_{ L^{\frac{N+d_\pel}{N+d_\pel-1}}_x L^1_\pel}\|K(t)\|_{ L^{N+d_\pel}_x}.$$
We use Proposition \ref{prop.interpolation} with $M=N$, $S=N-1$ to obtain
$$\| \pel_0^{N-1} f(t) \|_{ L^{\frac{N+d_\pel}{N+d_\pel-1}}_x L^1_\pel}\ls \| \pel_0^N f(t) \|_{ L^1_x L^1_\pel}^{\frac{N+d_\pel-1}{N+d_\pel}},$$
which implies
$$\frac{d}{dt}\| \pel_0^N f(t) \|_{ L^1_x L^1_\pel}\leq \| \pel_0^N f(t) \|_{ L^1_x L^1_\pel}^{\frac{N+d_\pel-1}{N+d_\pel}}\|K(t)\|_{ L^{N+d_\pel}_x}.$$
Dividing both sides by $\| \pel_0^N f(t) \|_{ L^1_x L^1_\pel}^{\frac{N+d_\pel-1}{N+d_\pel}}$, we get
$$\frac{d}{dt}\| \pel_0^N f(t) \|_{ L^1_x L^1_\pel}^{\frac{1}{N+d_\pel}}\ls \|K(t)\|_{ L^{N+d_\pel}_x}.$$
Integrating in time gives the desired result.
\end{proof}

\section{The two dimensional case: Local existence}\label{2D.sec.loc}
To prove our main theorem on the global existence and uniqueness of solutions to the two-dimensional relativistic Vlasov-Maxwell system, we will proceed in the following steps. First, we prove, using the energy method, the local existence and uniqueness of solutions in Section \ref{sec.2D.local.existence}. Moreover, using the local theory, in Section \ref{cc.sec} we will prove continuation criterion guaranteeing that the solution can be continued unless some norms of $f$, $E$ and $B$ and their first derivatives blow up.  Section \ref{2D.sec.global} then extends the local existence theorem globally in time by establishing that these norms remain finite on any bounded time interval.

In this Section \ref{2D.sec.loc}, we will use the notation $A \ls B$ to mean that $A\le CB$ where the implicit constant $(C\ge 0)$ may depend on any of the conserved quantities in Section \ref{sec.cons.law}, but $C$ will not depend upon $T>0$.  In this section the dependence upon $T$ will be explicitly tracked in the upper bounds.

\subsection{Local existence}\label{sec.2D.local.existence}
In this subsection, we will prove the following theorem on the local existence for the relativistic Vlasov-Maxwell system.  We also state the continuation criteria in the same theorem for convenience.

\begin{theorem}\label{theorem.local.existence}
Given the $2$D  initial data $(f_0(x,\pel),E_0(x),B_0(x))$ which satisfies \eqref{2D.ansatz} and the constraints \eqref{constraints}.  Suppose for $\numDnu$ we have
\begin{equation}\label{f.energy.est.2D}
\Dinit\eqdef \sum_{0\leq k\leq \nuD}\left(\|\nab_x^k K_0 \|_{L^2_x }^2
+
\|w_2\nab_{x,\pel}^k f_0\|_{L^2_x L^2_{\pel}}^2\right)<\infty,
\end{equation}
Then there exists a $T=T(\Dinit,\nuD)>0$ such that there exists a unique local solution to the relativistic Vlasov-Maxwell system in $[0,T]$ where the bound
\begin{multline}\label{f.energy.T.2D}
\enerD \eqdef
\sum_{0\leq k\leq \nuD}\left(\|\nab_x^k K \|_{L^\infty_t ([0,T];L^2_x)}^2
+
\|w_2\nab_{x,\pel}^k f\|_{L^\infty_t([0,T]; L^2_x L^2_{\pel})}^2\right)
\ls \Dinit 
\end{multline}
holds. Moreover, if $[0,T_*)$ is the maximal time interval of existence and $T_*< +\infty$
then 
$
\lim_{s \uparrow T_*}
\|
 \mathcal{A}
\|_{L^1_t ([0,s))}=+\infty,
$
where
\begin{equation}\label{cont.A}
 \mathcal{A}(t)
\eqdef
\|(K,\nab_x K)\|_{L^\infty_x}(t)
+\|w_2\nab_{x,\pel} f\|_{L^\infty_x L^2_{\pel}}(t).
\end{equation}
Here we recall the notation $K=(E, B)$ and $K_0=(E_0, B_0)$.
\end{theorem}

We will prove the existence and uniqueness parts of Theorem \ref{theorem.local.existence} by using an iteration scheme.  In this section we are considering case $d_\pel = 2$. Let $(f^{(n)},E^{(n)},B^{(n)})$ be defined iteratively for $n\geq 1$ as solutions to the following linear system:
\bea
& &\rd_t f^{(n)}+\vh\cdot\nabla_x f^{(n)}+ (E^{(n-1)}+\vh\times B^{(n-1)})\cdot \nabla_\pel f^{(n)} = 0,\label{vlasov.l}\\
& &\rd_t E^{(n)}= \nabla_x \times B^{(n)}-j^{(n)},\quad\rd_t B^{(n)}= -\nabla_x\times E^{(n)},\label{maxwell.l.1}\\
& &\nab_x\cdot E^{(n)}=\rho^{(n)},\quad \nab_x\cdot B^{(n)}=0.\label{maxwell.l.2}
\eea
This system is equipped with initial data
$$
(f^{(n)},E^{(n)},B^{(n)})|_{t=0} = (f_0,E_0,B_0)
$$
such that $(f_0,E_0,B_0)$ verify the constraint equations \eqref{constraints} and where $\rho^{(n)}$ and $j^{(n)}$ are defined by
$$\rho^{(n)}(t,x) \eqdef 4\pi\int_{\mathbb{R}^{d_\pel}} f^{(n)}(t,x,\pel) d\pel,$$
and
$$
j_i^{(n)}(t,x) \eqdef  4\pi \int_{\mathbb{R}^{d_\pel}} \vh_i f^{(n)}(t,x,\pel) d\pel, \quad i=1,..., d_\pel.
$$
We will also use the convention that $E^{(0)}=0$ and $B^{(0)}=0$.

Notice that by the definition of $f^{(n)}$, we have $\rd_t\rho^{(n)}+\nab_x\cdot j^{(n)}=0$. Therefore, the linear Maxwell equations \eqref{maxwell.l.1} and \eqref{maxwell.l.2} are well-posed\footnote{This can be seen, for example, by defining instead $(E^{(n)},B^{(n)})$ via the wave equations $\Box E^{(n)}= \nabla_x \rho^{(n)}+\rd_t j^{(n)}$ and $\Box B^{(n)}= -\nabla_x\times j^{(n)}$ with initial data $(f^{(n)},E^{(n)},B^{(n)})|_{t=0}=(f_0,E_0,B_0)$ and $(\partial_t E^{(n)}, \partial_t B^{(n)})|_{t=0}=(\nab_x\times B_0-j_0,-\nab_x\times E_0)$. 

One then shows that $\Box(\rd_t E^{(n)}-\nab_x\times B^{(n)}+j^{(n-1)})=0$ with zero initial data and similarly for the other equations in \eqref{maxwell.l.1} and \eqref{maxwell.l.2}. Thus the solutions to the wave equations are indeed the solutions to the Maxwell equations.} and $(f^{(n)},E^{(n)},B^{(n)})$ are defined globally in time. We will show that they converge to a solution of the relativistic Vlasov-Maxwell system using energy estimates. First we show 

\begin{proposition}\label{unif.iterate}
Given initial data $(f_0(x,\pel),E_0(x),B_0(x))$ as in the statement of Theorem \ref{theorem.local.existence}, there exists a $T=T(\Dinit, \nuD)>0$ such that for all $n\geq 1$, we have
$$
\sum_{0\leq k\leq \nuD}\left(\|\nab_x^k K^{(n)} \|_{L^\infty_t ([0,T];L^2_x )}^2
+
\|w_2\nab_{x,\pel}^k f^{(n)}\|_{L^\infty_t([0,T]; L^2_x L^2_{\pel})}^2
\right)
\ls \Dinit.
$$
Here the implicit constant is uniform in $n\geq 1$.
\end{proposition}

Above and in the following we will use the notation $K^{(n)} \eqdef (E^{(n)}, B^{(n)})$.  Additionally we use the notation $\tilde{K}^{(n)} \eqdef E^{(n)}+\vh\times B^{(n)}$.

\begin{proof}
Consider equation \eqref{vlasov.l} when it is commuted with the derivatives $\nab_{x,\pel}^k$:
\begin{equation}\label{vlasov.n.commute}
\begin{split}
\rd_t \nab_{x,\pel}^k f^{(n)}+\vh\cdot\nab_x \nab_{x,\pel}^k f^{(n)}+ (E^{(n-1)}+\vh\times B^{(n-1)})\cdot \nab_\pel \nab_{x,\pel}^k f^{(n)}=F_k^{(n)},
\end{split}
\end{equation}
where $F_k$ denotes all of the remaining terms.  It obeys the bound
\begin{equation}\label{vlasov.n.commute.2}
|F_k^{(n)}|\ls 
\sum_{\substack{i+j=k\\ 0\le j\leq k-1 }}
\pel_0^{-i} |\nab_{x}\nab_{x,\pel}^{j} f^{(n)}|
+
\sum_{\substack{i+j=k\\ 0\le j\leq k-1 }}
 |\nab_{x,\pel}^i \tilde{K}^{(n-1)}|  |\nab_{\pel}\nab_{x,\pel}^{j} f^{(n)}|.
\end{equation}
We recall \eqref{weight.notation} and then we multiply \eqref{vlasov.n.commute} by $w_2(p)^2\nab_{x,\pel}^k f$ and integrate in $[0,s]\times \mathbb R^2_x\times \mathbb R^2_{\pel}$ to obtain
\begin{multline*}
\frac 12\int_{[0,s]\times \mathbb R^2_x\times \mathbb R^2_{\pel}} w_2(p)^2\frac{\partial}{\partial t}(\nab_{x,\pel}^k f^{(n)}(t))^2 d\pel\,dx\,dt
\\
+\frac 12\int_{[0,s]\times \mathbb R^2_x\times \mathbb R^2_{\pel}} w_2(p)^2 \vh\cdot\nab_x (\nab_{x,\pel}^k f^{(n)})^2d\pel\,dx\,dt
\\
+\frac 12\int_{[0,s]\times \mathbb R^2_x\times \mathbb R^2_{\pel}} w_2(p)^2(E^{(n-1)}+\vh\times B^{(n-1)})\cdot\nab_{\pel} (\nab_{x,\pel}^k f^{(n)})^2 d\pel\,dx\,dt
\\
\ls 
\|w_2^2F_k^{(n)} \nab_{x,\pel}^k f^{(n)} \|_{L^1_t([0,T]; L^1_x L^1_{\pel})}
\\
\ls 
\|w_2\nab_{x,\pel}^k f^{(n)}\|_{L^\infty_t([0,T]; L^2_x L^2_{\pel})}
\|w_2 F_k^{(n)}\|_{L^1_t([0,T]; L^2_x L^2_{\pel})}.
\end{multline*}
The first term on the left hand side is equal to
$$\|w_2\nab_{x,\pel}^k f^{(n)}(s)\|_{L^2_x L^2_{\pel}}^2-\|w_2\nab_{x,\pel}^k f_0\|_{L^2_x L^2_{\pel}}^2.$$
The second term on the left hand side vanishes. This can be seen after integrating by parts in $x$. We integrate by parts in the third term on the left hand side and control it using H{\"o}lder's inequality (up to a constant) by
$$
\|w_2 \nab_{x,\pel}^k f^{(n)}\|_{L^\infty_t([0,T]; L^2_x L^2_{\pel})}\|\log(1+\pel_0)\tilde{K}^{(n-1)} \nab_{x,\pel}^k f^{(n)}\|_{L^1_t([0,T]; L^2_x L^2_{\pel})}.
$$
Putting these together, and taking the supremum of $\|w_2\nab_{x,\pel}^k f^{(n)}(s)\|_{L^2_x L^2_{\pel}}^2$ over all $s\in[0,T]$, we get
\begin{multline*}
\|w_2\nab_{x,\pel}^k f^{(n)}\|_{L^\infty_t([0,T];L^2_x L^2_{\pel})}^2
\ls 
\|w_2\nab_{x,\pel}^k f_0\|_{L^2_x L^2_{\pel}}^2
\\
+\|w_2\nab_{x,\pel}^k f^{(n)}\|_{L^\infty_t([0,T]; L^2_x L^2_{\pel})}
\\
\qquad\times\left(\|w_2 F_k^{(n)}\|_{L^1_t([0,T]; L^2_x L^2_{\pel})}\right.
\\
\left.\qquad\qquad+\|\log(1+\pel_0)\tilde{K}^{(n-1)} \nab_{x,\pel}^k f^{(n)}\|_{L^1_t([0,T]; L^2_x L^2_{\pel})}\right).
\end{multline*}
Absorbing $\|w_2\nab_{x,\pel}^k f^{(n)} \|_{L^1_t([0,T]; L^2_x L^2_{\pel})}$ to the left hand side, we get
\begin{multline*}
\|w_2\nab_{x,\pel}^k f^{(n)}\|_{L^\infty_t([0,T];L^2_x L^2_{\pel})}^2
\\
\ls 
\|w_2\nab_{x,\pel}^k f_0\|_{L^2_x L^2_{\pel}}^2
+
\|\log(1+\pel_0)\tilde{K}^{(n-1)} \nab_{x,\pel}^k f^{(n)}\|_{L^1_t([0,T]; L^2_x L^2_{\pel})}^2
\\
+\|w_2 F_k^{(n)}\|_{L^1_t([0,T]; L^2_x L^2_{\pel})}^2.
\end{multline*}
Taking into account the form of $F_k^{(n)}$ in \eqref{vlasov.n.commute.2}, we thus have
\begin{multline}\label{EE.f}
\|w_2\nab_{x,\pel}^k f^{(n)}\|_{L^\infty_t([0,T];L^2_x L^2_{\pel})}^2
\\
\ls 
\|w_2\nab_{x,\pel}^k f_0\|_{L^2_x L^2_{\pel}}^2
+
\|\log(1+\pel_0)\tilde{K}^{(n-1)} \nab_{x,\pel}^k f^{(n)}\|_{L^1_t([0,T]; L^2_x L^2_{\pel})}^2
\\
+
\sum_{\substack{i+j=k\\ 0\le j\leq k-1 }}
\|w_2 \nab_{x,\pel}^i \tilde{K}^{(n-1)} \nab_{\pel}\nab_{x,\pel}^{j} f^{(n)} \|_{L^1_t([0,T]; L^2_x L^2_{\pel})}^2
\\
+\sum_{\substack{0\le j\leq k-1 }}
\|\log(1+\pel_0)\nab_{x}\nab_{x,\pel}^{j} f^{(n)} \|_{L^1_t([0,T]; L^2_x L^2_{\pel})}^2.
\end{multline}
This completes our weighted $L^2_{x,\pel}$ estimate for $\nab_{x,\pel}^k f^{(n)}$.

We now derive an analogous $L^2$ estimate for $\nab_x^k K^{(n)}$. We take the first equation in \eqref{maxwell.l.1} and multiply it by $E^{(n)}$.  We similarly consider the second equation in \eqref{maxwell.l.1} multiplied by $B^{(n)}$.  This yields
\begin{multline*}
\|K^{(n)}\|_{L^\infty_t([0,T];L^2_x)}^2 
\ls 
\|K_0\|_{L^2_x}^2
+
\|f^{(n)}\|_{L^1_t([0,T];L^2_xL^1_{\pel})}^2
\\
\ls 
\|K_0\|_{L^2_x}^2
+
\|w_2 f^{(n)}\|_{L^1_t([0,T];L^2_xL^2_{\pel})}^2,
\end{multline*}
since $\int_{\mathbb R^2} \pel_0^{-2}\log^{-2}(1+\pel_0) d\pel \ls 1$. Commuting equations \eqref{maxwell.l.1} with $\nab_{x}^k$, similarly
\begin{equation}\label{EE.K}
\|\nab_x^k K^{(n)}\|_{L^\infty_t([0,T];L^2_x)}^2 
\ls  
\|\nab_x^k K_0\|_{L^2_x}^2+
\|w_2 \nab_x^k f^{(n)}\|_{L^1_t([0,T];L^2_xL^2_{\pel})}^2.
\end{equation}
Adding \eqref{EE.f} and \eqref{EE.K} for $0\leq k\leq D$ and controlling the time integral by a pointwise $L^\infty_t$ time bound, we obtain
\begin{multline}\label{iterate.bd}
\sum_{0\leq k\leq \nuD}\left(\|\nab_x^k K^{(n)} \|_{L^\infty_t([0,T];L^2_x )}^2
+
\|w_2\nab_{x,\pel}^k f^{(n)}\|_{L^\infty_t([0,T]; L^2_x L^2_{\pel})}^2\right)
\\
\ls \sum_{0\leq k\leq \nuD}\left(\|\nab_x^k K_0 \|_{L^2_x }^2
+
\|w_2\nab_{x,\pel}^k f_0\|_{L^2_x L^2_{\pel}}^2\right) + M_1,
\end{multline}
where
\begin{multline*}
M_1\eqdef  T^2\| K^{(n-1)} \|_{L^\infty_t([0,T]; L^\infty_x )}^2
\sum_{0\leq k\leq \nuD}\|\log(1+\pel_0) \nab_{x,\pel}^k f^{(n)}\|_{L^\infty_t([0,T]; L^2_x L^2_{\pel})}^2
\\
+T^2\sum_{0\leq k\leq \nuD}\|w_2\nab_{x,\pel}^k f^{(n)}\|_{L^\infty_t([0,T]; L^2_x L^2_{\pel})}^2
+T^2 M_2.
\end{multline*}  
We define $M_2$  above as in the following:
\begin{multline*}
 M_2\eqdef  
\sum_{0\leq k\leq \nuD}
\sum_{\substack{i+j =k \\  0\le j\leq k-1 }}\|w_2\nab_x^i \tilde{K}^{(n-1)}\nab_{\pel}\nab_{x,\pel}^j f^{(n)} \|_{L^\infty_t([0,T]; L^2_x L^2_{\pel})}^2
\\
\ls 
\sum_{\substack{ i+j \leq \nuD \\ 0\le i \leq \nuD-2}}
\|\nab_x^i K^{(n-1)} \|_{L^\infty_t([0,T]; L^\infty_x )}^2
\|w_2 \nab_{x,\pel}^j f^{(n)}\|_{L^\infty_t([0,T]; L^2_x L^2_{\pel})}^2
\\
+
\sum_{\substack{i+j \leq \nuD \\ 0\leq j\leq \nuD-2  }}
\|\nab_x^i K^{(n-1)} \|_{L^\infty_t([0,T]; L^2_x )}^2
\|w_2 \nab_{x,\pel}^j f^{(n)}\|_{L^\infty_t([0,T]; L^\infty_x L^2_{\pel})}^2
\\
+
\|\nab_x^{2} K^{(n-1)} \|_{L^\infty_t([0,T]; L^4_x)}^2
\|w_2 \nab_{x,\pel}^{2}  f^{(n)}\|_{L^\infty_t([0,T]; L^4_x L^2_{\pel})}^2. 
\end{multline*}
The last term in the upper bound is only needed for the case $\nuD = 3$.  Notice that
\begin{equation*}
\begin{split}
|\nab_x (\|\nab^k_{x,\pel} f \|_{L^2_{\pel}}^2 (t,x))|=&
2\left|\int \nab^k_{x,\pel} f \nab_{x}\nab^{k}_{x,\pel} f d\pel \right|\\
\leq &2(\|\nab^k_{x,\pel} f \|_{L^2_{\pel}} (t,x))(\|\nab^{k+1}_{x,\pel} f \|_{L^2_{\pel}} (t,x)).
\end{split}
\end{equation*}
Dividing through by $2\|\nab^k f \|_{L^2_{\pel}}^2 (t,x)$, we have
\begin{equation}\label{norm.DER}
|\nab_x (\|\nab^k_{x,\pel} f \|_{L^2_{\pel}} (t,x))|\leq \|\nab^{k+1}_{x,\pel} f \|_{L^2_{\pel}} (t,x). 
\end{equation}
We will also use the following Gagliardo-Nirenberg inequality in 2D.  Let $G: \mathbb R^2 \to \mathbb R$ be a real-valued function. Then
\begin{equation}\label{GN}
\|G\|_{L^4(\mathbb R^2)} \ls \|G \|_{L^2(\mathbb R^2)}^{\frac 12}\|\nab G \|_{L^2(\mathbb R^2)}^{\frac 12}.
\end{equation}
Further we record the Sobolev embedding:
\begin{equation}\label{SE.used}
\| G \|_{L^\infty(\mathbb{R}^2_x)} \ls \| G \|_{H^2(\mathbb{R}^2_x)}.
\end{equation}
Using \eqref{norm.DER}, the Sobolev embedding \eqref{SE.used},  and the Gagliardo-Nirenberg inequality \eqref{GN}, we obtain
\begin{multline*}
M_1
\ls 
T^2\sum_{0\leq k\leq \nuD}\|w_2\nab_{x,\pel}^k f^{(n)}\|_{L^\infty_t([0,T]; L^2_x L^2_{\pel})}^2
\\
+T^2\sum_{0\leq k\leq \nuD}\|\nab_x^k K^{(n-1)} \|_{L^\infty_t([0,T]; L^2_x )}^2
\sum_{0\leq k\leq \nuD}\|w_2\nab_{x,\pel}^k f^{(n)}\|_{L^\infty_t([0,T]; L^2_x L^2_{\pel})}^2.
\end{multline*}
We plug this estimate into \eqref{iterate.bd}.  Then, choosing $T$ appropriately small depending on the initial norm $\Dinit$, we can apply a simple induction argument to show that 
\begin{equation*}
\begin{split}
&\sum_{0\leq k\leq \nuD}\left(\|\nab_x^k K^{(n)} \|_{L^\infty_t([0,T); L^2_x )}^2+\|w_2\nab_{x,\pel}^k f^{(n)}\|_{L^\infty_t([0,T); L^2_x L^2_{\pel})}^2\right)\ls \Dinit,
\end{split}
\end{equation*}
where this bound will hold uniformly for all $n \ge 1$.
\end{proof}

To proceed, we define the following energy for the difference of the iterates for $n\ge 1$ and $\numDnu$:
\begin{multline*}
\iterD
\eqdef 
\sum_{0\leq k\leq \nuD -1}\|\nab_x^k (K^{(n)}-K^{(n-1)}) \|_{L^\infty_t ([0,T];L^2_x )}^2
\\
+
\sum_{0\leq k\leq \nuD -1}
\|w_2\nab_{x,\pel}^k (f^{(n)}-f^{(n-1)})\|_{L^\infty_t([0,T]; L^2_x L^2_{\pel})}^2.
\end{multline*}
We now show that in fact $(f^{(n)}, E^{(n)}, B^{(n)})$ is a Cauchy sequence.

\begin{proposition}\label{convergence}
Given initial data $(f_0(x,\pel),E_0(x),B_0(x))$ as in the statement of Theorem \ref{theorem.local.existence} and $\Dinit$ from \eqref{f.energy.est.2D}, then for $\numDnu$ there exists a positive time $T=T(\Dinit,\nuD)\ll 1$ such that for all $n\geq 1$ we have the following estimate for some constant $C>0$:
\begin{equation*}
\iterD\leq \left( C \Dinit T^2\right)^{n-1}.
\end{equation*}
In particular, by choosing $T$ smaller if necessary, $K^{(n)}$ is a Cauchy sequence in $L^\infty_t ([0,T];H^{\nuD-1}_x)$ and $f^{(n)}$ is a Cauchy sequence in 
$L^\infty_t([0,T]; H^{\nuD-1}(w_2^2(p)d\pel dx))$.  Moreover, the limits  $f \in L^\infty_t([0,T]; H^{\nuD}(w_2^2(p)d\pel dx))$ and $K \in L^\infty_t ([0,T];H^{\nuD}_x)$  give rise to a unique local solution to the 2D relativistic Vlasov-Maxwell system \eqref{vlasov}, \eqref{maxwell}, \eqref{constraints}.
\end{proposition}

\begin{proof}
We now derive the energy estimates for the difference of solutions $f^{(n)}-f^{(n-1)}$ and $K^{(n)}-K^{(n-1)}$. Similar to the previous proposition we have
$$
\left( \rd_t +\vh\cdot\nab_x + \tilde{K}^{(n-1)}\cdot \nab_\pel  \right)\nab_{x,\pel}^k \left( f^{(n)} - f^{(n-1)} \right)
=
H^{(n)}_k,
$$
where $H^{(n)}_k$ satisfies the upper bound
\begin{multline*}
\left| H^{(n)}_k\right| 
\lesssim
\sum_{i+j= k}  \left| \nab_{x,\pel}^i \left(\tilde{K}^{(n-1)}-\tilde{K}^{(n-2)}\right) \right| \left| \nab_\pel  \nab_{x,\pel}^j f^{(n-1)} \right|
\\
+\sum_{\substack{i+j=k\\ i \ne 0 }}  \left|  \nab_{x,\pel}^i \tilde{K}^{(n-1)}  \right|   \left|  \nab_\pel  \nab_{x,\pel}^j \left(f^{(n)}-f^{(n-1)}\right)  \right|
\\
+\sum_{\substack{i+j=k\\ i \ne 0 }} \pel_0^{-i}  \left| \nab_x  \nab_{x,\pel}^j \left(f^{(n)}-f^{(n-1)}\right)  \right|.
\end{multline*}
Since the initial data terms coincide, similar to the estimates \eqref{EE.f} and \eqref{EE.K} we get
\begin{multline*}
\iterD
\ls 
\sum_{0\leq k\leq \DnuD}\|w_2\nab_{x,\pel}^k (f^{(n)}-f^{(n-1)})\|_{L^1_t([0,T]; L^2_x L^2_{\pel})}^2
\\
+\sum_{0\leq k\leq \DnuD}
\sum_{i+j= k}
\|w_2\nab_{x,\pel}^i (\tilde{K}^{(n-1)}-\tilde{K}^{(n-2)}) \nab_{\pel}\nab_{x,\pel}^{j} f^{(n-1)}\|_{L^1_t([0,T]; L^2_x L^{2}_{\pel})}^2
\\
+\sum_{0\leq k\leq \DnuD}
\sum_{\substack{ i+j= k \\  i \ne 0}}
\|w_2\nab_{x,\pel}^i \tilde{K}^{(n-1)} \nab_{\pel} \nab_{x,\pel}^{j} (f^{(n)}-f^{(n-1)})\|_{L^1_t([0,T]; L^2_x L^{2}_{\pel})}^2
\\
+\sum_{0\leq k\leq \DnuD}
\|\log(1+\pel_0) \tilde{K}^{(n-1)} \nab_{x,\pel}^k (f^{(n)}-f^{(n-1)})\|_{L^1_t([0,T]; L^2_x L^2_{\pel})}^2.
\end{multline*}
Similar to the proof of the previous proposition, using \eqref{norm.DER}, \eqref{GN} and \eqref{SE.used} we obtain
$$
\iterD
\ls  T^2 \iterD+ T^2 \Dinit \iterND + T^2 \Dinit \iterD,
$$
where the implicit constant does not depend upon $n\ge 1$.   Then a simple induction argument gives the desired estimate. 

In particular the sequence $f^{(n)}$ is Cauchy in $L^\infty_t([0,T]; H^{\nuD-1}(w_2^2(p)d\pel dx))$ and the sequences $E^{(n)}$ and $B^{(n)}$ are Cauchy in $L^\infty_t ([0,T];H^{\nuD-1}_x)$ with $\numDnu$. Using also the uniform bounds in Proposition \ref{unif.iterate}, sending $n\to\infty$, we can therefore straightforwardly define the limits $f \in L^\infty_t([0,T]; H^{\nuD}(w_2(p)^2 d\pel dx))$, and $(E, B) \in L^\infty_t ([0,T];H^{\nuD}_x)$.

The same estimates as above will also control the difference of two solutions and this then shows that the solution constructed is unique.
\end{proof}

\subsection{Continuation Criteria}\label{cc.sec}
We now prove the continuation criterion. To this end, it suffices to work with the actual solution (instead of an approximating sequence) and show that as long as 
$\|K\|_{L^1_t ([0,T_*);L^\infty_x)}$, $\|\nab_xK\|_{L^1_t( [0,T_*);L^\infty_x)}$ and also $\|w_2\nab_{x,\pel} f\|_{L^1_t ([0,T_*);L^\infty_x L^2_{\pel})}$ are bounded, then $\enerD$ from \eqref{f.energy.T.2D} is additionally bounded. This will allow us to invoke the local existence theorem to contradict the maximality of $T_*$.   We now prove the following continuation criterion:

\begin{proposition}\label{prop.cont}
To continue a solution, we recall the quantity $\mathcal{A}(t)$ from \eqref{cont.A} and we suppose that $\|  \mathcal{A} \|_{L^1_t( [0,T_*))} < \infty$.  
Then, recalling \eqref{f.energy.T.2D}, for $\nuD \ge 0$ we have
$$
\sqrt{\mathcal E_{T_*,D}}
\le
C^*<\infty,
$$
where $C^*=C^*(\Dinit,\|\mathcal{A}\|_{L^1_t([0,T_*))},T_*,\nuD)$ is a positive constant depending on $\Dinit$, $\|\mathcal{A}\|_{L^1_t([0,T_*))}$, $\nuD$ and $T_*$ only.
\end{proposition}

\begin{proof}
We will prove Proposition \ref{prop.cont} via induction on the number of derivatives $\nuD \ge 0$.  
The energy inequality in \eqref{EE.f} and \eqref{EE.K} with $f$ and $K$ (instead of $f^{(n)}$ and $K^{(n)}$) yields
$$
\sqrt{\mathcal E_{T_*,D}}
\ls 
\sqrt{\Dinit}
+
\tilde{M}_\nuD,
$$
where we remark that we removed the squares in \eqref{EE.f} and \eqref{EE.K} and 
\begin{multline}\label{M.est.D}
\tilde{M}_\nuD \eqdef
\sum_{0\leq k\leq \nuD}
\sum_{\substack{i+j=k\\ 0\le j\leq k-1 }}
\|w_2\nab_{x,\pel}^i \tilde{K} \nabla_{\pel} \nab_{x,\pel}^{j} f\|_{L^1_t([0,T_*); L^2_x L^2_{\pel})}
\\
+\sum_{0\leq k\leq \nuD}\|\log(1+\pel_0) |K| \nab_{x,\pel}^k f\|_{L^1_t([0,T_*); L^2_x L^2_{\pel})}
\\
+\sum_{0\leq k\leq \nuD}\|w_2\nab_{x,\pel}^k f\|_{L^1_t([0,T_*); L^2_x L^2_{\pel})}.
\end{multline}
Considering the first term above, we have the estimate
\begin{multline}\label{M.est.2D}
\sum_{0\leq k\leq \nuD}
\sum_{\substack{i+j=k\\ 0\le j\leq k-1 }}
\|w_2\nab_{x,\pel}^i \tilde{K} \nabla_{\pel} \nab_{x,\pel}^{j} f\|_{L^1_t([0,T_*); L^2_x L^2_{\pel})}
\\
\ls
\left\| 
\mathcal{A}(t) \sqrt{\mathcal{E}_{t,\nuD}} \right\|_{L^1_t([0,T_*))}
+
\sum_{\substack{i+j\leq \nuD\\ 1\le j\leq \nuD-2\\ 2\le i \le \nuD-1 }}
\int_0^{T_*} \|\nab_x^i K \|_{L^4_x }\|w_2 \nabla_{\pel} \nab_{x,\pel}^{j} f\|_{L^4_x L^2_{\pel}} \, dt
\end{multline}
Note that the sum $\sum_{\substack{i+j\leq \nuD\\ 1\le j\leq \nuD-2\\ 2\le i \le \nuD-1 }}$ is empty when $0\leq \nuD \leq 2$, i.e., when $0\leq \nuD \leq 2$ the second term is not present at all.  Now we can start the induction; when $0\leq \nuD \leq 2$ we have 
$$
\tilde{M}_\nuD \ls
\int_0^{T_*} dt \left( \mathcal{A}(t)+1\right)
\sqrt{\mathcal{E}_{t,\nuD}} \quad (0\leq \nuD \leq 2)
$$
Using Gronwall's inequality, we then have
\begin{equation}\notag
\sqrt{\mathcal E_{T_*,D}}
\le 
\tilde{C}^* \sqrt{\Dinit} \quad (0\leq \nuD \leq 2),
\end{equation}
where $\tilde{C}^*<\infty$ is a positive constant depending only on $\|\mathcal{A}\|_{L^1_t([0,T_*))}$ and $T_*$. 
Further suppose that for some integer $J\ge 2$ we have
\begin{equation}\label{cc.2}
\sqrt{\mathcal E_{T_*,D}}
\ls 
\tilde{C}^*_J  \quad (0\le \nuD \le J),
\end{equation}
where $\tilde{C}^*_J<\infty$ is a positive constant depending only on $\sqrt{\mathcal E_{0,J}}$, $\|\mathcal A\|_{L^1_t([0,T_*))}$, $T_*$ and $J$.
We will prove the same inequality holds for $J+1$.

To this end we estimate last term in \eqref{M.est.2D} when $\nuD = J+1$.   We apply \eqref{GN} and \eqref{norm.DER}, when 
$1\le j\leq \nuD-2$ and $2\le i \le \nuD-1$, and we use \eqref{cc.2} 
 to get
\begin{multline*}
\int_0^{T_*}  \|\nab_x^i K \|_{L^4_x }\|w_2 \nabla_{\pel} \nab_{x,\pel}^{j} f\|_{L^4_x L^2_{\pel}} \, dt
\\
\ls \int_0^{T_*}  \|\nab_x^i K \|_{L^2_x }^{\frac 12}\|\nab_x^{i+1} K \|_{L^2_x }^{\frac 12}
\|w_2 \nab_{x,\pel}^{j+1} f\|_{L^2_x L^2_{\pel}}^{\frac 12}\|w_2 \nab_{x,\pel}^{j+2} f\|_{L^2_x L^2_{\pel}}^{\frac 12}\, dt
\\
\ls 
\tilde{C}^*_J 
\int_0^{T_*}  \|\nab_x^{i+1} K \|_{L^2_x }^{\frac 12}\|w(p) \nab_{x,\pel}^{j+2} f\|_{L^2_x L^2_{\pel}}^{\frac 12}\, dt\\
\ls \tilde{C}^*_J  \int_0^{T_*}  \sqrt{\mathcal E_{t,J+1}}\, dt.
\end{multline*}
Substituting this into \eqref{M.est.D} and \eqref{M.est.2D} using similar estimates and applying Gronwall's inequality, we obtain
the desired result.
\end{proof}

\section{The two dimensional case: Global existence}\label{2D.sec.global}

In this section we prove that the local existence theorem can be extended globally in time.  Specifically in Theorem \ref{theorem.local.existence} we proved that the local classical solution can be continued as long as certain norms of the solution remain finite.  In particular Section's \ref{sec.2D.GS} and \ref{sec.2D.hr} are devoted to proving that these norms are in fact bounded and the unique solution is therefore global in time.  To achieve this, we fix $T>0$ and will bound $\|\mathcal A\|_{L^1_t([0,T);L^\infty_x)}$, where $\mathcal A$ is defined as in \eqref{cont.A}. 

In the following, we first prove moment bounds for $f$ and estimates for $E$ and $B$ in some Strichartz norms. These bounds will be obtained in Section \ref{sec.2D.moment} after recalling the Glassey-Schaeffer decomposition of the electromagnetic fields in Section \ref{sec.2D.GS}. Finally, in Section \ref{sec.2D.hr}, we will show that these bounds in fact imply stronger control for $f$, $E$, $B$ and their first derivatives, thus guaranteeing (by the local existence theorem) that the solution is global.

In the proofs of the main theorems below, we will use the notation $A \ls B$ to mean that $A\le CB$ where the implicit constant $(C\ge 0)$ may depend on any of the conserved quantities in Section \ref{sec.cons.law} and it can also depend upon the time $T>0$. Since we will also prove the growth bound \eqref{growth.bd} of $\|E\|_{L^\infty_t([0,T);L^\infty_x)}$ and $\|B\|_{L^\infty_t([0,T);L^\infty_x)}$, we will need to track this dependence on $T$. In Sections \ref{sec.2D.GS}-\ref{sec.K.Linfty.bd}, we will allow the implicit constant in $\ls$ to depend on $T$ at most \emph{polynomially}. After deriving estimates for $E$ and $B$ in $L^\infty_x$, we will proceed in Sections \ref{sec.2D.hr} and \ref{sec.2D.con} to obtain bounds for the first derivatives for $E$, $B$ and $f$. In these two sections, we will not optimize the dependence on $T$ and will allow the implicit constant to have \emph{arbitrary dependence} on $T$. Nevertheless, since $T$ is arbitrary, we obtain a global solution by Theorem \ref{theorem.local.existence}.

\subsection{Glassey-Schaeffer decomposion of the electromagnetic fields}\label{sec.2D.GS}

In order to obtain the necessary estimates to apply Theorem \ref{theorem.local.existence} to obtain the global existence of solutions, we need to control the electromagnetic field by the particle density and the electromagnetic field itself (and not their derivatives). To this end, we use the Glassey-Schaeffer representation of the electromagnetic fields. More precisely, in \cite{GS2D1}, \cite{GS2D2}, Glassey-Schaeffer showed that the electromagnetic fields can be decomposed into
$$E^i(t,x)=E^i=\tilde{E}_0^i+E^i_S+E^i_T, \quad (i=1,2),$$
and
$$B(t,x)=B=\tilde{B}_0+B_S+B_T,$$
where $\tilde{E}^i_0$ and $\tilde{B}_0$ depend only on the initial data.\footnote{Moreover, these initial data terms have $C^1$  and $H^3$ norms depending only on the initial data norms \eqref{ini.bd.2D.2} - \eqref{ini.bd.2D.5} for $f_0$ in Theorem \ref{main.theorem.2D} and the time interval $[0,T]$.}

Further define $\xi\eqdef\frac{y-x}{t-s}$; then the other terms in $E^i$ are given by
$$
E^i_T= -2\int_0^t \int_{|y-x|\leq t-s}\int_{\mathbb R^2} \frac{(1-|\vh|^2)(\xi_i+\vh_i)}{(1+\vh\cdot\xi)^2}\frac{ f(s,y,\pel) d\pel\, dy\, ds,}{(t-s)\sqrt{(t-s)^2-|y-x|^2}} 
$$
and
$$
E^i_S
=
\int_0^t \int_{|y-x|\leq t-s}\int_{\mathbb R^2} 
es^i\cdot\frac{((E_1+\vh_2 B,E_2-\vh_1 B) f)(s,y,\pel)}{\sqrt{(t-s)^2-|y-x|^2}} d\pel\, dy\, ds,
$$
where 
$$
es^{ij}
\eqdef
-2 \frac{\delta_{ij}-\vh_i\vh_j}{1+\vh\cdot\xi}\frac{1}{\pel_0}
+2\frac{(\xi_i+\vh_i)(\xi_j-(\xi\cdot\vh)\vh_j)}{(1+\vh\cdot\xi)^2}\frac{1}{\pel_0}
\quad (j=1,2).
$$
Further the terms in $B$ can be expressed as
$$
B_T= 2\int_0^t \int_{|y-x|\leq t-s}\int_{\mathbb R^2} \frac{(1-|\vh|^2)(\xi \wedge \vh)}{(1+\vh\cdot\xi)^2}\frac{ f(s,y,\pel)}{(t-s)\sqrt{(t-s)^2-|y-x|^2}} d\pel\, dy\, ds
$$
and
$$
B_S
=
\int_0^t \int_{|y-x|\leq t-s}\int_{\mathbb R^2} 
bs\cdot\frac{((E_1+\vh_2 B,E_2-\vh_1 B) f)(s,y,\pel)}{\sqrt{(t-s)^2-|y-x|^2}} d\pel\, dy\, ds,
$$
where 
$$
bs^{j}
\eqdef
2 \frac{\xi_1\delta_{j2}-\xi_2\delta_{j1}}{(1+\vh\cdot\xi)\pel_0}
-2\frac{(\xi \wedge\vh)\vh_j }{(1+\vh\cdot\xi)\pel_0}
-
\frac{(\xi \wedge\vh)(\xi_j-(\xi\cdot\vh)\vh_j)}{(1+\vh\cdot\xi)^2\pel_0}
\quad (j=1,2).
$$
In these expressions the two dimensional cross product is $a\wedge b \eqdef a_1 b_2 - a_2 b_1 $ for two vectors $a=(a_1,a_2)$ and $b=(b_1,b_2)$.

We will use the following abbreviated notations
$$
|K| \eqdef |(E,B)|, \quad | K_S|  \eqdef |E_S|+|B_S|, \quad | K_T | \eqdef |E_T|+|B_T|.
$$ 
We further divide $K_S$ into $K_{S,1}, K_{S,2} \ge 0$ where
$$
| K_S | \le  K_{S,1} +K_{S,2},
$$ 
and $K_{S,2}$ contains only the good components $K_g$ from \eqref{good.comp.2D} but $K_{S,1}$ is less singular. Then the following bounds are proved in \cite{GS2D2}:

\begin{proposition}[Glassey-Schaeffer \cite{GS2D2}]\label{GSbd.2D}  The following estimates hold:
\begin{gather}\label{KT.gs.est}
|K_T(t,x)|\ls \int_0^t \int_{|y-x|\leq t-s}\int_{\mathbb R^2} \frac{f(s,y,\pel) ~ d\pel\, dy\, ds}{\pel_0^2(1+\vh\cdot\xi)^{\frac 32}(t-s)\sqrt{(t-s)^2-|y-x|^2}},
\\
\label{KS1.gs.est}
|K_{S,1}(t,x)|\ls \int_0^t \int_{|y-x|\leq t-s}\int_{\mathbb R^2} \frac{(|K|  f)(s,y,\pel)~ d\pel\, dy\, ds}{\pel_0\sqrt{(t-s)^2-|y-x|^2}},
\\
\label{KS2.gs.est}
|K_{S,2}(t,x)|\ls \int_0^t \int_{|y-x|\leq t-s}\int_{\mathbb R^2} \frac{(K_g f)(s,y,\pel)~ d\pel\, dy\, ds}{(1+\vh\cdot\xi)\pel_0\sqrt{(t-s)^2-|y-x|^2}}.
\end{gather}
\end{proposition}

\begin{proof}
To obtain the first bound \eqref{KT.gs.est}, it suffices to note that
$$\frac{(1-|\vh|^2)(\xi_i+\vh_i)}{(1+\vh\cdot\xi)^2}\ls \frac{1}{\pel_0^2(1+\vh\cdot\xi)^{\frac 32}},
\quad\frac{(1-|\vh|^2)(\xi \wedge \vh)}{(1+\vh\cdot\xi)^2}\ls \frac{1}{\pel_0^2(1+\vh\cdot\xi)^{\frac 32}}.$$
These easily follow from the following facts:
$$1-|\vh|^2=\frac{1}{\pel_0^2}$$
and
$$|\xi_i+\vh_i|+|\xi \wedge \vh|\ls (1+\vh\cdot\xi)^{\frac 12}.$$
The remaining bounds \eqref{KS1.gs.est} and \eqref{KS2.gs.est} can be found in Lemma 2 of \cite{GS2D2}.
\end{proof}

We will derive some further estimates for $K_T$ and $K_{S,2}$ in this decomposition so that we can apply Strichartz estimates. To this end, we need the following lemma:

\begin{lemma}\label{sing.lemma}
Suppose $|y-x|\le (t-s)$ and $\xi=\frac{y-x}{t-s}$, then we have the estimates:
\begin{equation}\label{sing.1}
\int_{\mathbb R^2} \frac{f(s,y,\pel)d\pel}{\pel_0(1+\vh\cdot\xi)}\ls  \frac{(\int_{\mathbb R^2} \pel_0^2 f(s,y,\pel)d\pel)^{\frac 25}}{(1-|\xi|^2)^{\frac 25}},
\end{equation}
and 
\begin{equation}\label{sing.2}
\int_{\mathbb R^2} \frac{f(s,y,\pel)d\pel}{\pel_0(1+\vh\cdot\xi)}\ls  (\int_{\mathbb R^2} \pel_0^4 f(s,y,\pel)d\pel)^{\frac 25}.
\end{equation}
\end{lemma}
\begin{proof}
We first prove the estimate \eqref{sing.1}. We will control the integral separately for small and large momenta. Let $R$ be a constant to be chosen later. Define the angle $\theta\in (-\pi,\pi]$ by
$$-\vh\cdot\xi=|\vh| |\xi| \cos\theta.$$
Notice that 
\begin{equation}\label{sing.2D.est}
(1+\vh\cdot\xi)^{-1}\ls \min\left\{\frac{1}{\th^{2}},\, \frac{1}{1-|\xi|^2},\, \pel_0^2\right\}.
\end{equation}
For $|\pel|\leq R$, we write the $d\pel$ integral in polar coordinates to get
\begin{multline}\label{sing.lemma.1}
\int_{|\pel|\leq R} \frac{f(s,y,\pel)d\pel}{\pel_0(1+\vh\cdot\xi)}
\ls \int_0^R \int_{-\pi}^{\pi} \frac{1}{\pel_0(1+\vh\cdot\xi)} |\pel|\,d\theta\,d|\pel|\\
\ls \int_0^R \int_{[-\pi,\pi]\setminus [-R^{-1},R^{-1}]} \frac{1}{\th^2} \,d\theta\,d|\pel|+\int_0^R \int_{-R^{-1}}^{R^{-1}} \pZ^2 \,d\theta\,d|\pel|
\ls 1+R^2.
\end{multline}
For $|\pel|\geq R$, we use $\frac{1}{1+\vh\cdot\xi}\ls \frac 1{1-|\xi|^2}$ to get
\begin{equation}\label{sing.lemma.2}
\begin{split}
\int_{|\pel|\geq R} \frac{f(s,y,\pel)d\pel}{\pel_0(1+\vh\cdot\xi)}\ls &\frac{1}{\langle R\rangle^3(1-|\xi|^2)}\int_{|\pel|\geq R} \pel_0^2 f(s,y,\pel)d\pel.
\end{split}
\end{equation}
Taking $R=\frac{(\int_{\mathbb R^2} f\pel_0^2\,d\pel)^{\frac 15}}{(1-|\xi|^2)^{\frac 15}}$, \eqref{sing.lemma.1} and \eqref{sing.lemma.2} imply the first estimate \eqref{sing.1} when $R \ge 1$.

We now turn to the second estimate \eqref{sing.2}. As before, we will consider large and small momenta separately. For $|\pel|\leq R$, we use the estimate \eqref{sing.lemma.1}.  For $|\pel|\geq R$, we use $\frac{1}{1+\vh\cdot\xi}\ls \pel_0^2$ to get
\begin{equation}\label{sing.lemma.4}
\begin{split}
\int_{|\pel|\geq R} \frac{f(s,y,\pel)d\pel}{\pel_0(1+\vh\cdot\xi)}\ls &\frac{1}{\langle R\rangle^3}\int_{|\pel|\geq R} \pel_0^4 f(s,y,\pel)d\pel.
\end{split}
\end{equation}
Let $R=(\int_{\mathbb R^2} \pel_0^4 f(s,y,\pel)d\pel)^{\frac 15}$. Then \eqref{sing.lemma.1} and \eqref{sing.lemma.4} imply the second estimate \eqref{sing.2} when $R \ge 1$.  
If the previously chosen $R$'s are $\le 1$ then the desired estimates are proven more easily using instead \eqref{sing.lemma.2} and \eqref{sing.lemma.4} with $R=0$.
\end{proof}

Using this lemma, we can then derive the following estimates for $K_T$. We allow a parameter $\ep>0$ in our bound, which we will later choose to be some inverse power of some moments of $f$.
\begin{proposition}\label{KT.2D}
For every $\ep\in (0,1]$, $K_T$ obeys the following estimate:
\begin{multline}
|K_T(t,x)|
\ls \ep^{-\frac 1{10}}\left(\int_0^t \int_{|y-x|\leq t-s} \frac{\int_{\mathbb R^2} f(s,y,\pel)\pel_0^2 d\pel}{\sqrt{(t-s)^2-|y-x|^2}}  dy ds\right)^{\frac 25}\\
+\ep^{\frac 3{10}}\left(\int_0^t \int_{|y-x|\leq t-s} \frac{\int_{\mathbb R^2} f(s,y,\pel)\pel_0^4 d\pel}{\sqrt{(t-s)^2-|y-x|^2}}  dy ds\right)^{\frac{2}{5}}.
\end{multline}
\end{proposition}

\begin{proof}
We will begin with the upper bound from \eqref{KT.gs.est}. Notice that by \eqref{sing.2D.est}, we have
$$\frac{1}{\pel_0^2(1+\vh\cdot\xi)^{\frac 32}}\ls \frac{1}{\pel_0(1+\vh\cdot\xi)}.$$
It therefore suffices to control
$$\int_0^t \int_{|y-x|\leq t-s}\int_{\mathbb R^2} \frac{f(s,y,\pel) ~ d\pel\, dy\, ds}{\pel_0(1+\vh\cdot\xi)(t-s)\sqrt{(t-s)^2-|y-x|^2}}.$$
We will divide the region of integration into $|y-x|\leq (1-\ep)(t-s)$ and $(1-\ep)(t-s)<|y-x|\leq (t-s)$.

We first consider the integration region $|y-x|\leq (1-\ep)(t-s)$. Using \eqref{sing.1} in Lemma \ref{sing.lemma} and the H{\"o}lder's inequality (with $q=\frac 53$ and $q'=\frac 52$), we obtain
\begin{multline}\label{KT.2D.1}
\int_0^t \int_{\frac{|y-x|}{t-s}\leq (1-\ep)}\int_{\mathbb R^2} \frac{f(s,y,\pel) ~ d\pel\, dy\, ds}{\pel_0(1+\vh\cdot\xi)(t-s)\sqrt{(t-s)^2-|y-x|^2}}\\
\ls \int_0^t \int_{\frac{|y-x|}{t-s}\leq (1-\ep)}\frac{(\int_{\mathbb R^2} \pel_0^2 f(s,y,\pel) ~ d\pel)^{\frac 25}\, dy\, ds}{(1-|\xi|^2)^{\frac 25}(t-s)\sqrt{(t-s)^2-|y-x|^2}}\\
\ls 
\int_0^t \int_{\frac{|y-x|}{t-s}\leq (1-\ep)} \frac{1}{(t-s)^{\frac 85}\big(1-|\xi|^2\big)^{\frac 7{10}}}
\times\frac{(\int_{\mathbb R^2} \pel_0^2 f(s,y,\pel) ~ d\pel)^{\frac 25}}{\big((t-s)^2-|y-x|^2\big)^{\frac{2}{10}}} \, dy\, ds
\\
\ls 
\left(\int_0^t \int_{\frac{|y-x|}{t-s}\leq (1-\ep)} \frac{ dy\, ds}{(t-s)^{\frac 83}(1-|\xi|^2)^{\frac 76}} \right)^{\frac 35}\\
\times\left(\int_0^t \int_{|y-x|\leq t-s} \frac{\int_{\mathbb R^2} f(s,y,\pel)\pel_0^2 d\pel}{\sqrt{(t-s)^2-|y-x|^2}}  dy ds\right)^{\frac{2}{5}}.
\end{multline}
We now control the first factor. Recalling $\xi=\frac{y-x}{t-s}$, we get
\begin{multline}\label{KT.1stfactor.1}
\int_0^t \int_{\frac{|y-x|}{t-s}\leq (1-\ep)} \frac{dy\, ds}{(t-s)^{\frac 83}(1-|\xi|^2)^{\frac 76}}
\\
\ls \int_0^t \int_{|\xi|\leq 1-\ep} \frac{d\xi\, ds}{(t-s)^{\frac 23}(1-|\xi|^2)^{\frac 76}}
\\
\ls \int_0^t  \frac{ds}{(t-s)^{\frac{2}{3}}} 
\int_{|\xi|\leq 1-\ep} \frac{d\xi }{(1-|\xi|^2)^{\frac 76}}\ls \ep^{-\frac 16}.
\end{multline}
Therefore, returning to \eqref{KT.2D.1}, we have
\begin{multline}\label{KT.2D.2}
\int_0^t \int_{\frac{|y-x|}{t-s}\leq (1-\ep)}\int_{\mathbb R^2} \frac{f(s,y,\pel) ~ d\pel\, dy\, ds}{\pel_0(1+\vh\cdot\xi)(t-s)\sqrt{(t-s)^2-|y-x|^2}}\\
\ls  \ep^{-\frac 1{10}}\left(\int_0^t \int_{|y-x|\leq t-s} \frac{\int_{\mathbb R^2} f(s,y,\pel)\pel_0^2 d\pel}{\sqrt{(t-s)^2-|y-x|^2}}  dy ds\right)^{\frac{2}{5}}.
\end{multline}
We now turn to the remaining integration region $(1-\ep)(t-s)<|y-x|\leq (t-s)$. For this region, we use \eqref{sing.2} instead to obtain
\begin{multline}\notag 
\int_0^t \int_{1-\ep<\frac{|y-x|}{t-s}\leq 1}\int_{\mathbb R^2} \frac{f(s,y,\pel) ~ d\pel\, dy\, ds}{\pel_0(1+\vh\cdot\xi)(t-s)\sqrt{(t-s)^2-|y-x|^2}}\\
\ls \int_0^t \int_{1-\ep<\frac{|y-x|}{t-s}\leq 1}\frac{(\int_{\mathbb R^2} \pel_0^4 f(s,y,\pel) ~ d\pel)^{\frac 25}\, dy\, ds}{(t-s)\sqrt{(t-s)^2-|y-x|^2}}\\
\ls 
\int_0^t \int_{1-\ep<\frac{|y-x|}{t-s}\leq 1} \frac{1}{(t-s)^{\frac 85}\big(1-|\xi|^2\big)^{\frac 3{10}}}
\times\frac{(\int_{\mathbb R^2} \pel_0^4 f(s,y,\pel) ~ d\pel)^{\frac 25}}{\big((t-s)^2-|y-x|^2\big)^{\frac{2}{10}}} \, dy\, ds
\eqdef A_2.
\end{multline}
where
\begin{multline}\label{KT.2D.3}
A_2
\ls 
\left(\int_0^t \int_{1-\ep<\frac{|y-x|}{t-s}\leq 1} \frac{ dy\, ds}{(t-s)^{\frac 83}(1-|\xi|^2)^{\frac 12}} \right)^{\frac 35}\\
\times\left(\int_0^t \int_{|y-x|\leq t-s} \frac{\int_{\mathbb R^2} f(s,y,\pel)\pel_0^4 d\pel}{\sqrt{(t-s)^2-|y-x|^2}}  dy ds\right)^{\frac{2}{5}}.
\end{multline}
We now control the first factor.
\begin{multline}\label{KT.1stfactor.2}
\int_0^t \int_{(1-\ep)<\frac{|y-x|}{t-s}\leq 1} \frac{dy\, ds}{(t-s)^{\frac 83}(1-|\xi|^2)^{\frac 12}}
\\
\ls \int_0^t \int_{(1-\ep)<|\xi|\leq 1} \frac{d\xi\, ds}{(t-s)^{\frac 23}(1-|\xi|^2)^{\frac 12}}
\\
\ls \int_0^t  \frac{ds}{(t-s)^{\frac{2}{3}}} 
\int_{(1-\ep)<|\xi|\leq 1} \frac{d\xi }{(1-|\xi|^2)^{\frac 12}}\ls \ep^{\frac 12}.
\end{multline}
Therefore, returning to \eqref{KT.2D.3}, we get
\begin{multline}\label{KT.2D.4}
\int_0^t \int_{(1-\ep)<\frac{|y-x|}{t-s}\leq 1}\int_{\mathbb R^2} \frac{f(s,y,\pel) ~ d\pel\, dy\, ds}{\pel_0(1+\vh\cdot\xi)(t-s)\sqrt{(t-s)^2-|y-x|^2}}\\
\ls  \ep^{\frac 3{10}}\left(\int_0^t \int_{|y-x|\leq t-s} \frac{\int_{\mathbb R^2} f(s,y,\pel)\pel_0^4 d\pel}{\sqrt{(t-s)^2-|y-x|^2}}  dy ds\right)^{\frac{2}{5}}.
\end{multline}
We obtain the conclusion of the proposition after combining \eqref{KT.2D.2} and \eqref{KT.2D.4}.
\end{proof}
\begin{remark}\label{KT.2D.rmk}
Notice that we have in fact proved a slightly stronger result. The proof of Proposition \ref{KT.2D} shows that for all non-negative functions $F(t,s,x,y)$, $G(s,y)$ and $H(s,y)$ satisfying 
\begin{equation*}
F(t,s,x,y)\ls  \frac{G(s,y)^{\frac 25}}{(1-|\xi|^2)^{\frac 25}},
\end{equation*}
and 
\begin{equation*}
F(t,s,x,y)\ls  H(s,y)^{\frac 25},
\end{equation*}
in the region $|y-x|\leq (t-s)$,
we have
\begin{equation*}
\begin{split}
&\int_0^t \int_{|y-x|\leq (t-s)} \frac{F(t,s,x,y)\, dy\,ds}{(t-s)\sqrt{(t-s)^2-|y-x|^2}}\\
\ls &\ep^{-\frac 1{10}}\left(\int_0^t \int_{|y-x|\leq t-s} \frac{G(s,y)}{\sqrt{(t-s)^2-|y-x|^2}}  dy ds\right)^{\frac 25}\\
&+\ep^{\frac 3{10}}\left(\int_0^t \int_{|y-x|\leq t-s} \frac{H(s,y)}{\sqrt{(t-s)^2-|y-x|^2}}  dy ds\right)^{\frac{2}{5}}
\end{split}
\end{equation*}
for every $\ep\in (0,1]$.
\end{remark}
We then show that $K_{S,2}$ obeys the same bound after applying the conservation law in Proposition \ref{cons.law.2}.

\begin{proposition}\label{KS2.2D}  For every $\ep\in (0,1]$, $K_{S,2}$ obeys the following estimate:
\begin{multline*}
|K_{S,2}(t,x)|
\ls \ep^{-\frac 1{10}}\left(\int_0^t \int_{|y-x|\leq t-s} \frac{\int_{\mathbb R^2} f(s,y,\pel)\pel_0^2 d\pel}{\sqrt{(t-s)^2-|y-x|^2}}  dy ds\right)^{\frac 25}\\
+\ep^{\frac 3{10}}\left(\int_0^t \int_{|y-x|\leq t-s} \frac{\int_{\mathbb R^2} f(s,y,\pel)\pel_0^4 d\pel}{\sqrt{(t-s)^2-|y-x|^2}}  dy ds\right)^{\frac{2}{5}}.
\end{multline*}
\end{proposition}

\begin{proof}
We recall the estimate \eqref{KS2.gs.est}.  As in the proof of Proposition \ref{KT.2D}, we first derive the estimate in the region $|y-x|\leq (1-\ep)(t-s)$ and then obtain the bound in the region $|y-x|> (1-\ep)(t-s)$. In the first region, we apply \eqref{sing.1} and H\"older's inequality to get
\begin{multline}\label{KS2.2D.main.1}
\int_0^t \int_{\frac{|y-x|}{t-s}\leq (1-\ep)}\int_{\mathbb R^2} \frac{K_g(s,y)f(s,y,\pel)}{\pel_0(1+\vh\cdot\xi)\big((t-s)^2-|y-x|^2\big)^{\frac 12}}\,d\pel\,dy\,ds\\
\ls \int_0^t \int_{\frac{|y-x|}{t-s}\leq (1-\ep)}\frac{K_g(s,y)(\int_{\mathbb R^2} \pel_0^2 f(s,y,\pel) ~ d\pel)^{\frac 25}\, dy\, ds}{(1-|\xi|^2)^{\frac 25}\sqrt{(t-s)^2-|y-x|^2}}\\
\ls 
\int_0^t \int_{\frac{|y-x|}{t-s}\leq (1-\ep)} \frac{K_g(s,y)}{(t-s)^{\frac 35}\big(1-|\xi|^2\big)^{\frac 7{10}}}
\times\frac{(\int_{\mathbb R^2} \pel_0^2 f(s,y,\pel) ~ d\pel)^{\frac 25}}{\big((t-s)^2-|y-x|^2\big)^{\frac{2}{10}}} \, dy\, ds
\\
\ls 
\left(\int_0^t \int_{\frac{|y-x|}{t-s}\leq (1-\ep)} \frac{K_g(s,y)^{\frac 53} dy\, ds}{(t-s)(1-|\xi|^2)^{\frac 76}} \right)^{\frac 35}\\
\times\left(\int_0^t \int_{|y-x|\leq t-s} \frac{\int_{\mathbb R^2} f(s,y,\pel)\pel_0^2 d\pel}{\sqrt{(t-s)^2-|y-x|^2}}  dy ds\right)^{\frac{2}{5}}.
\end{multline}
We will now bound the integral 
\bea\label{KS2.2D.1}
\int_0^t \int_{\frac{|y-x|}{(t-s)}\leq (1-\ep)} \frac{K_g^{\frac 53} \, dy\, ds}{(t-s)(1-|\xi|^2)^{\frac 76}}.
\eea
We use the conservation law in Proposition \ref{cons.law.2}. To proceed, we will use a series of changes of variable from \cite{GS2.5D,GS2D1,GS2D2}. In particular, one of the coordinate functions $\psi$ is a null variable which measures the distance to the cone $\{|y-x|=(t-s)\}$.

We first introduce this change of variables in the whole region $\{|y-x|\leq (t-s)\}$. We expand out the integrals as
$$
\int_0^t \, ds  \int_{|y-x|\leq t-s} ~ dy 
=
\int_0^t \, ds \int_0^{t-s} \, dr  \int_{|y-x|=r} ~ dS_y  
$$
Then we will do the change of variables $\psi = \frac{1}{2}\left( t-s-r\right)$ in the $r$ integration.  After that we apply Fubini's theorem to interchange the order of integration.  Then 
\begin{multline*}
\int_0^t \, ds \int_0^{t-s} \, dr  \int_{|y-x|=r} ~ dS_y 
=
\frac{1}{2}\int_0^t \, ds \int_0^{\frac{1}{2}\left(t-s\right)} \, d\psi  \int_{|y-x|=t-s-2\psi} ~ dS_y 
\\
=
\frac{1}{2}\int_0^{t/2} \, d\psi \int_0^{t-2\psi} \, ds  \int_{|y-x|=t-s-2\psi} ~ dS_y. 
\end{multline*}
We now return to the region $r\leq (1-\ep)(t-s)$ which appears in \eqref{KS2.2D.1}. We have
\begin{multline*}
\int_0^t \, ds \int_0^{(1-\ep)(t-s)} \, dr  \int_{|y-x|=r} ~ dS_y 
=
\int_0^t \, ds \int_{\ep \frac{1}{2}(t-s)}^{\frac{1}{2}(t-s)} \, d\psi  \int_{|y-x|=t-s-2\psi} ~ dS_y 
\\
=
\int_{\ep\frac{t}{2}}^{\frac t2} \, d\psi \int_0^{t-2\psi} \, ds  \int_{|y-x|=t-s-2\psi} ~ dS_y\\
+\int_{0}^{\ep\frac{t}{2}} \, d\psi \int_{t-\frac{2}{\ep}\psi}^{t-2\psi} \, ds  \int_{|y-x|=t-s-2\psi} ~ dS_y. 
\end{multline*}
In order to bound \eqref{KS2.2D.1}, we first note that in this coordinate system, we have
\begin{equation}\label{equality.coord}
1-|\xi|^2=\frac{4\psi(t-s-\psi)}{(t-s)^2}. 
\end{equation}
Therefore, writing the integral \eqref{KS2.2D.1} in this coordinate system, we have
\begin{multline}\label{KS2.2D.2}
\int_0^t \int_{|y-x|\leq (1-\ep)(t-s)} \frac{K_g^{\frac 53} \, dy\, ds}{(t-s)(1-|\xi|^2)^{\frac 76}}\\
\ls \int_{\ep\frac{t}{2}}^{\frac t2} \, d\psi \int_0^{t-2\psi} \, ds  \int_{|y-x|=t-s-2\psi} ~ dS_y \frac{K_g^{\frac 53}(t-s)^{\frac 43}}{\psi^{\frac 76}(t-s-\psi)^{\frac 76}}\\
+\int_{0}^{\ep\frac{t}{2}} \, d\psi \int_{t-\frac{2}{\ep}\psi}^{t-2\psi} \, ds  \int_{|y-x|=t-s-2\psi} ~ dS_y \frac{K_g^{\frac 53}(t-s)^{\frac 43}}{\psi^{\frac 76}(t-s-\psi)^{\frac 76}}.
\end{multline}
Notice that by the conservation law in Proposition \ref{cons.law.2}, the integrals of $|K_g|^2$  can be bounded after taking supremum in $\psi$, i.e.,
\begin{equation}\label{notSure.def}
{\NotSure}(t-2\psi)
\eqdef
\int_0^{t-2\psi}\int_{|y-x|=t-s-2\psi} |K_g|^2 dS_y\, ds
\le 
\sup_{\psi} {\NotSure}(t-2\psi)\ls 1.
\end{equation}
Thus we will use H\"older's inequality; 
in this way the first line on the right hand side of \eqref{KS2.2D.2} can be controlled by 
\begin{equation}\label{KS2.2D.3}
\int_{\ep\frac{t}{2}}^{\frac t2} \frac{d\psi }{\psi^{\frac 76}}
({\NotSure}(t-2\psi))^{\frac 56} (\int_0^{t-{2\psi}} \frac{(t-s)^8(t-s-2\psi)}{(t-s-\psi)^7}\, ds)^{\frac 16}  
\end{equation}
and the second line of the right hand side of \eqref{KS2.2D.2} can be bounded by
\begin{equation}\label{KS2.2D.4}
\int_0^{\ep\frac{t}{2}} \frac{d\psi }{\psi^{\frac 76}}({\NotSure}(t-2\psi))^{\frac 56} (\int_{t-\frac{2\psi}{\ep}}^{t-2\psi} \frac{(t-s)^8(t-s-2\psi)}{(t-s-\psi)^7}\, ds)^{\frac 16}. 
\end{equation}
Then in order to estimate \eqref{KS2.2D.3} and \eqref{KS2.2D.4}, we now need to bound
\begin{equation}\label{KS2.2D.5}
\int_{\ep\frac{t}{2}}^{\frac t2} \frac{1}{\psi^{\frac 76}} (\int_0^{t-{2\psi}} \frac{(t-s)^8(t-s-2\psi)}{(t-s-\psi)^7}\, ds)^{\frac 16}  d\psi 
\end{equation}
and
\begin{equation}\label{KS2.2D.6}
\int_0^{\ep\frac{t}{2}} \frac{1}{\psi^{\frac 76}} (\int_{t-\frac{2\psi}{\ep}}^{t-2\psi} \frac{(t-s)^8(t-s-2\psi)}{(t-s-\psi)^7}\, ds)^{\frac 16}  d\psi. 
\end{equation}
We make the simple observation that 
\begin{equation}\label{KS2.2D.7}
\begin{split}
\frac{(t-s)^8(t-s-2\psi)}{(t-s-\psi)^7}=&(1+\frac{\psi}{t-s-\psi})^7((t-s-\psi)^2-\psi^2)\\
\ls & (t-s-\psi)^2+\psi^2+\frac{\psi^7}{(t-s-\psi)^5}+\frac{\psi^9}{(t-s-\psi)^7}.
\end{split}
\end{equation}
To control the term \eqref{KS2.2D.5},
we apply \eqref{KS2.2D.7} to get
\begin{equation}\label{KS2.2D.8}
\begin{split}
&\int_{\ep\frac{t}{2}}^{\frac t2} \frac{1}{\psi^{\frac 76}} (\int_0^{t-2\psi} \frac{(t-s)^8(t-s-2\psi)}{(t-s-\psi)^7}\, ds)^{\frac 16}  d\psi\\
\ls &\int_{\ep\frac{t}{2}}^{\frac t2} \frac{1}{\psi^{\frac 76}} (1+\psi^3+\frac{\psi^7}{(t-\psi)^4}+\frac{\psi^9}{(t-\psi)^6})^{\frac 16}  d\psi\\
\ls &\int_{\ep\frac{t}{2}}^{\frac t2} \big(\frac{1}{\psi^{\frac 76}}+\frac{1}{(t-\psi)^{\frac 23}}+\frac{\psi^{\frac 13}}{(t-\psi)} \big) d\psi
\ls \ep^{-\frac 16}.
\end{split}
\end{equation}
Notice that in the above, 
we have absorbed the dependence on $t$ in the implicit constant. Similarly, we apply \eqref{KS2.2D.7} to control \eqref{KS2.2D.6}:
\begin{equation}\label{KS2.2D.9}
\begin{split}
&\int_0^{\ep\frac{t}{2}} \frac{1}{\psi^{\frac 76}} (\int_{t-\frac{2\psi}{\ep}}^{t-2\psi} \frac{(t-s)^8(t-s-2\psi)}{(t-s-\psi)^7}\, ds)^{\frac 16}  d\psi\\
\ls &\int_0^{\ep\frac{t}{2}} \frac{1}{\psi^{\frac 76}} (\big(\frac{\psi}{\ep}\big)^3+\psi^3)^{\frac 16}  d\psi\\
\ls &\int_0^{\ep\frac{t}{2}} \big(\frac{1}{\ep^{\frac 12}\psi^{\frac 23}}+\frac{1}{\psi^{\frac 23}} \big) d\psi
\ls \ep^{-\frac 16}.
\end{split}
\end{equation}
Now the estimates \eqref{KS2.2D.2}, \eqref{KS2.2D.3}, \eqref{KS2.2D.4}, \eqref{KS2.2D.8} and \eqref{KS2.2D.9} together imply that we have the bound
$$\int_0^t \int_{|y-x|\leq (1-\ep)(t-s)} \frac{K_g^{\frac 53} \, dy\, ds}{(t-s)(1-|\xi|^2)^{\frac 76}}\ls \ep^{-\frac 16}. $$
Therefore, returning to \eqref{KS2.2D.main.1}, we have
\begin{multline}\label{KS2.2D.main.1.1}
\int_0^t \int_{\frac{|y-x|}{t-s}\leq (1-\ep)}\int_{\mathbb R^2} \frac{K_g(s,y)f(s,y,\pel)}{\pel_0(1+\vh\cdot\xi)\big((t-s)^2-|y-x|^2\big)^{\frac 12}}\,d\pel\,dy\,ds \\
\ls \ep^{-\frac 1{10}}\left(\int_0^t \int_{|y-x|\leq t-s} \frac{\int_{\mathbb R^2} f(s,y,\pel)\pel_0^2 d\pel}{\sqrt{(t-s)^2-|y-x|^2}}  dy ds\right)^{\frac 25}.
\end{multline}

We now turn to the integral over the region $(1-\ep)<\frac{|y-x|}{t-s}\leq 1$. We use \eqref{sing.2} in Lemma \ref{sing.lemma} to obtain
\begin{multline}\label{KS2.2D.main.2}
\int_0^t \int_{1-\ep\leq \frac{|y-x|}{t-s}\leq 1}\int_{\mathbb R^2} \frac{K_g(s,y)f(s,y,\pel)}{\pel_0(1+\vh\cdot\xi)\big((t-s)^2-|y-x|^2\big)^{\frac 12}}\,d\pel\,dy\,ds\\
\ls \int_0^t \int_{1-\ep\leq \frac{|y-x|}{t-s}\leq 1}\frac{K_g(s,y)(\int_{\mathbb R^2} \pel_0^4 f(s,y,\pel) ~ d\pel)^{\frac 25}\, dy\, ds}{\sqrt{(t-s)^2-|y-x|^2}}\\
\ls 
\int_0^t \int_{1-\ep\leq \frac{|y-x|}{t-s}\leq 1} \frac{K_g(s,y)}{(t-s)^{\frac 35}\big(1-|\xi|^2\big)^{\frac 3{10}}}
\times\frac{(\int_{\mathbb R^2} \pel_0^4 f(s,y,\pel) ~ d\pel)^{\frac 25}}{\big((t-s)^2-|y-x|^2\big)^{\frac{2}{10}}} \, dy\, ds
\\
\ls 
\left(\int_0^t \int_{1-\ep\leq \frac{|y-x|}{t-s}\leq 1} \frac{K_g(s,y)^{\frac 53} dy\, ds}{(t-s)(1-|\xi|^2)^{\frac 12}} \right)^{\frac 35}\\
\times\left(\int_0^t \int_{|y-x|\leq t-s} \frac{\int_{\mathbb R^2} f(s,y,\pel)\pel_0^4 d\pel}{\sqrt{(t-s)^2-|y-x|^2}}  dy ds\right)^{\frac{2}{5}}.
\end{multline}
Similar to the previous case, we now estimate the integral
\begin{equation}\label{KS2.2D.10}
\int_0^t \int_{1-\ep\leq \frac{|y-x|}{t-s}\leq 1} \frac{K_g^{\frac 53} dy\, ds}{(t-s)(1-|\xi|^2)^{\frac 12}}
\end{equation}
by changing into the null coordinates. Recalling that $\psi=\frac 12(t-s-r)$, we have in this region
\begin{multline*}
\int_0^t \, ds \int_{(1-\ep)(t-s)}^{t-s} \, dr  \int_{|y-x|=r} ~ dS_y
=
\int_0^{\ep\frac{t}{2}} \, d\psi \int_0^{t-\frac{2}{\ep}\psi} \, ds  \int_{|y-x|=t-s-2\psi} ~ dS_y.
\end{multline*}
Moreover, recall that \eqref{equality.coord} holds. 
Then \eqref{KS2.2D.10} can be rewritten as
\begin{multline}\label{KS2.2D.11}
\int_0^t \int_{1-\ep\leq \frac{|y-x|}{t-s}\leq 1} \frac{K_g^{\frac 53} dy\, ds}{(t-s)(1-|\xi|^2)^{\frac 12}}\\
=C \int_0^{\ep\frac{t}{2}} \int_0^{t-\frac{2}{\ep}\psi}  \int_{|y-x|=t-s-2\psi}\frac{K_g^{\frac 53}}{\psi^{\frac 12}(t-s-\psi)^{\frac 12}} \, d\psi\, ds\, dS_y.
\end{multline}
Applying H\"older's inequality, and recalling \eqref{notSure.def}, then \eqref{KS2.2D.11} can be controlled by
\begin{equation}\label{KS2.2D.12}
\int_0^{\ep\frac{t}{2}} \frac{1}{\psi^{\frac 12}}
({\NotSure}(t-2\psi))^{\frac 56} (\int_0^{t-\frac{2}{\ep}\psi} \frac{(t-s-2\psi)}{(t-s-\psi)^3}\, ds)^{\frac 16}  d\psi.
\end{equation}
As in the previous case, we use Proposition \ref{cons.law.2} to obtain the upper bound in \eqref{notSure.def} that ${\NotSure}(t-\frac{2}{\ep}\psi) \ls 1$.  The remaining contribution can be estimated by
\begin{multline}\notag
\int_0^{\ep\frac{t}{2}} \frac{1}{\psi^{\frac 12}} (\int_0^{t-{\frac{2}{\ep}\psi}} \frac{(t-s-2\psi)}{(t-s-\psi)^3}\, ds)^{\frac 16}  d\psi\\
\ls \int_0^{\ep\frac{t}{2}} \frac{1}{\psi^{\frac 12}} (\int_0^{t-{\frac{2}{\ep}\psi}} \frac{1}{(t-s-\psi)^2}+\frac{\psi}{(t-s-\psi)^3}\, ds)^{\frac 16}  d\psi\\
\ls \int_0^{\ep\frac{t}{2}}  \frac{1}{\psi^{\frac 12}}(1+\frac{\ep}{\psi}+\frac{\ep^2}{\psi})^{\frac 16} d\psi\\
\ls \int_0^{\ep\frac{t}{2}}  (\frac{1}{\psi^{\frac 12}}+\frac{\ep^{\frac 16}}{\psi^{\frac 23}}+\frac{\ep^{\frac 13}}{\psi^{\frac 23}})\,d\psi
\ls \ep^{\frac 12}.
\end{multline}
Therefore, returning to \eqref{KS2.2D.main.2}, we obtain
\begin{multline}\label{KS2.2D.main.2.1}
\int_0^t \int_{1-\ep\leq \frac{|y-x|}{t-s}\leq 1}\int_{\mathbb R^2} \frac{K_g(s,y)f(s,y,\pel)}{\pel_0(1+\vh\cdot\xi)\big((t-s)^2-|y-x|^2\big)^{\frac 12}}\,d\pel\,dy\,ds\\
\ls 
\ep^{\frac 3{10}}\left(\int_0^t \int_{|y-x|\leq t-s} \frac{\int_{\mathbb R^2} f(s,y,\pel)\pel_0^4 d\pel}{\sqrt{(t-s)^2-|y-x|^2}}  dy ds\right)^{\frac{2}{5}}.
\end{multline}
Combining \eqref{KS2.2D.main.1.1} and \eqref{KS2.2D.main.2.1}, we obtain the desired conclusion.
\end{proof}
\begin{remark}\label{KS2.2D.rmk}
Similar to Proposition \ref{KT.2D}, we have in fact proved a slightly stronger result. The proof of Proposition \ref{KS2.2D} shows that if $F(t,s,x,y)$, $G(s,y)$ and $H(s,y)$ are non-negative functions satisfying 
\begin{equation*}
F(t,s,x,y)\ls  \frac{G(s,y)^{\frac 25}}{(1-|\xi|^2)^{\frac 25}},
\end{equation*}
and 
\begin{equation*}
F(t,s,x,y)\ls  H(s,y)^{\frac 25},
\end{equation*}
in the region $|y-x|\leq (t-s)$ and $K_g$ obeys the bound as in Proposition \ref{cons.law.2},
then
\begin{equation*}
\begin{split}
&\int_0^t \int_{|y-x|\leq (t-s)} \frac{|K_g(s,y)| F(t,s,x,y)\, dy\,ds}{\sqrt{(t-s)^2-|y-x|^2}}\\
\ls &\ep^{-\frac 1{10}}\left(\int_0^t \int_{|y-x|\leq t-s} \frac{G(s,y)}{\sqrt{(t-s)^2-|y-x|^2}}  dy ds\right)^{\frac 25}\\
&+\ep^{\frac 3{10}}\left(\int_0^t \int_{|y-x|\leq t-s} \frac{H(s,y)}{\sqrt{(t-s)^2-|y-x|^2}}  dy ds\right)^{\frac{2}{5}}
\end{split}
\end{equation*}
for every $\ep\in (0,1]$.
\end{remark}

\subsection{Propagation of moments}\label{sec.2D.moment}
To show the global existence, uniqueness and regularity in two dimensions, we first prove that for $N>13$ as in the assumption \eqref{ini.bd.2D.2} in Theorem \ref{main.theorem.2D}, $\| \pel_0^N f\|_{L^\infty_t([0,T);L^1_x L^1_\pel)}$ remains finite on bounded time intervals. To this end, we will first show that appropriate $L^1_tL^q_x$ norms of the electromagnetic field can be controlled by an appropriate power of $\| \pel_0^N f\|_{L^{\infty}_t([0,T);L^1_x L^1_\pel)}$ and $\| \pel_0^N f\|_{L^{q'}_t([0,T);L^1_x L^1_\pel)}$. We first consider the $K_T$ term:
\begin{proposition}\label{KT.2}
For $N>13$ as in the assumption \eqref{ini.bd.2D.2} in Theorem \ref{main.theorem.2D}, there exists $q_2'<\infty$ sufficiently large such that we have the estimate
$$\| K_T \|_{L^{1}_t([0,T);L^{N+2}_x)}^{N+2}
\ls  \| f\pel_0^N\|_{L^\infty_t([0,T);L^1_xL^1_{\pel})}^{\frac{N+2}{5(N-1)}}\|f \pel_0^N\|_{L^{q_2'}_t([0,T);L^1_x L^1_\pel)}^{\frac{4N-7}{5(N-1)}},$$
where the implicit constant depends at most polynomially on $T$.
\end{proposition}

\begin{proof}
Using the estimate for $K_T$ in Proposition \ref{KT.2D} and Remark \ref{rem.strichartz} after Theorem \ref{Strichartz}, we can apply the Strichartz estimates to $|K_T|^{\frac 52}$ with $\ep^{-\frac 14}\int_{\mathbb R^2} f \pel_0^{2} d\pel$ and $\ep^{\frac 34}\int_{\mathbb R^2} f \pel_0^{4} d\pel$ as the inhomogeneous terms. More precisely, noticing that the exponents
$$
q_1=\frac{3k(N+2)}{(k-3)N+(2k-51)}, 
\quad 
r_1=\frac{2(N+2)}{5},$$
$$
 q'_2=\frac{k(N+2)}{(k-1)N+2k-17}, 
 \quad 
 r'_2=\frac{6(N+2)}{4N+23}.
$$ 
satisfy the assumptions of Theorem \ref{Strichartz} for $k>6$ and $N>13$, we can use the Strichartz estimates (Theorem \ref{Strichartz}) with these exponents to get 
\begin{equation*}
\begin{split}
&\| |K_T|^{\frac 52} \|_{L^{q_1}_t([0,T);L^{\frac{2(N+2)}{5}}_x)}^{N+2}\\
\ls  &\ep^{-\frac{N+2}{4}}\|f \pel_0^2\|_{L^{q_2'}_t([0,T);L_x^{\frac{6(N+2)}{4N+23}}L^1_\pel)}^{N+2}+\ep^{\frac{3(N+2)}{4}}\|f \pel_0^4\|_{L^{q_2'}_t([0,T);L_x^{\frac{6(N+2)}{4N+23}}L^1_\pel)}^{N+2}.
\end{split}
\end{equation*}
For the remainder of the argument, we do not need the precise values of $q_1$ and $q_2'$. It suffices to notice that $q_2'<\infty$. Notice further that
$$
\| K_T \|_{L^{{\frac 52}q_1}_t([0,T);L^{N+2}_x)}^{(N+2)\frac{5}{2}}
\approx
\| |K_T|^{\frac 52} \|_{L^{q_1}_t([0,T);L^{\frac{2(N+2)}{5}}_x)}^{N+2}.
$$
Then by H\"older's inequality, and allowing the implicit constant to depend at most polynomially on $T$, we have
\begin{equation}\label{KT.int.0}
\begin{split}
&\| K_T \|_{L^{1}_t([0,T);L^{N+2}_x)}^{N+2}\\
\ls & \ep^{-\frac{N+2}{10}}\|f \pel_0^2\|_{L^{q_2'}_t([0,T);L_x^{\frac{6(N+2)}{4N+23}}L^1_\pel)}^{\frac{2(N+2)}{5}}+\ep^{\frac{3(N+2)}{10}}\|f \pel_0^4\|_{L^{q_2'}_t([0,T);L_x^{\frac{6(N+2)}{4N+23}}L^1_\pel)}^{\frac{2(N+2)}{5}}.
\end{split}
\end{equation}
Our goal is now to control the right hand side of the above inequality. In particular we would like to replace $L^{q_2'}_t([0,T);L_x^{\frac{6(N+2)}{4N+23}}L^1_\pel)$ with $L^{q_2''}_t([0,T);L_x^{1}L^1_\pel)$ (for some potentially larger $q_2''$) and to reduce the powers on the right hand side.  The price to be paid in these upper bounds is to replace $\pel_0^2$ and $\pel_0^4$ by $\pel_0^N$.

We will consider fixed time estimates for $\|f\pel_0^2\|_{L_x^{\frac{6(N+2)}{4N+23}}L^1_\pel}$ and $\|f\pel_0^4\|_{L_x^{\frac{6(N+2)}{4N+23}}L^1_\pel}$ separately. We first consider $\|f\pel_0^2\|_{L_x^{\frac{6(N+2)}{4N+23}}L^1_\pel}$. We will use Proposition \ref{prop.interpolation} with 
$$
S= 2
\quad
\& 
\quad
M = \frac{16N+2}{4N+23}.
$$ 
Notice that $\frac{M+2}{S+2}=\frac{6(N+2)}{4N+23}$ since $M+2 = \frac{24N+48}{4N+23}$. 

Therefore, by the interpolation inequality Proposition \ref{prop.interpolation}, we have for every $t\in [0, T)$ that
\begin{equation}\label{KT.int.1}
\|f \pel_0^2\|_{L_x^{\frac{6(N+2)}{4N+23}}L^1_\pel}^{\frac{2(N+2)}{5}}
\ls 
\|f \pel_0^M\|_{L^1_x L^1_\pel}^{\frac{4N+23}{15}}.
\end{equation}
In the next steps we will control this final term by $\| f\pel_0^N\|_{L^\infty_t([0,T); L^1_x L^1_\pel)}$ and the energy $\| f\pel_0\|_{L^\infty_t([0,T); L^1_x L^1_\pel)}$ via interpolation. 

We will use H{\"o}lder's inequality such that 
$$
\frac{1}{q}+\frac{1}{q'} = 1, \quad 
M=\frac{1}{q}+\frac{N}{q'}
$$
or
$
q=\frac{N-1}{N-M}
$
and 
$
q'=\frac{N-1}{M-1}=\frac{(4N+23)(N-1)}{12N-21}.
$
(Notice that the assumptions of this proposition imply that $q,q'\in [1,\infty)$.) Then we use an iterated 
H{\"o}lder inequality to achieve
\begin{multline}\label{Holder.T}
\|f \pel_0^M\|_{ L^1_x L^1_\pel}
=\|f^{\frac{1}{q}} \pel_0^{\frac{1}{q}} f^{\frac{1}{q'}} \pel_0^{\frac{N}{q'}} \|_{L^1_x L^1_\pel}
\\
\le  \int_{\mathbb{R}^{2}} dx \|f \pel_0 \|_{L^1_\pel}^{\frac{1}{q}} \|f \pel_0^N \|_{L^1_\pel}^{\frac{1}{q'}}
\le   \|f \pel_0 \|_{L_x^{1}L^1_\pel}^{\frac{1}{q}} \|f \pel_0^N \|_{L_x^{1}L^1_\pel}^{\frac{1}{q'}}
\ls  \|f \pel_0^N \|_{ L^1_x L^1_\pel}^{\frac{1}{q'}}.
\end{multline}
Here also
$
 \|f \pel_0^N \|_{ L^1_x L^1_\pel}^{\frac{1}{q'}}
= \|f \pel_0^N \|_{ L^1_x L^1_\pel}^{\frac{12N-21}{(4N+23)(N-1)}}.
$
Notice that we have used the conservation law in Proposition \ref{cons.law.1}. Combining this with \eqref{KT.int.1}, we obtain the following bound for every $t\in [0,T)$:
\begin{equation}\label{KT.int.2}
\|f \pel_0^2\|_{L_x^{\frac{6(N+2)}{4N+23}}L^1_\pel}^{\frac{2(N+2)}{5}}
\ls 
\|f \pel_0^M\|_{L^1_x L^1_\pel}^{\frac{12N-21}{15(N-1)}}=\|f \pel_0^M\|_{L^1_x L^1_\pel}^{\frac{4N-7}{5(N-1)}}.
\end{equation}

We now turn to the estimate for $\|f\pel_0^4\|_{L_x^{\frac{6(N+2)}{4N+23}}L^1_\pel}$. We use Proposition \ref{prop.interpolation} with 
$$
S=4
\quad
\& 
\quad
M =  \frac{28N+26}{4N+23}.
$$ 
Notice that $\frac{M+2}{S+2}={\frac{6(N+2)}{4N+23}}$ since $M+2 = \frac{36(N+2)}{4N+23}$. 

Therefore, by the interpolation inequality Proposition \ref{prop.interpolation}, we have, for every $t\in [0,T)$,
\begin{equation}\label{KT.int.3}
\|f \pel_0^4\|_{L_x^{\frac{6(N+2)}{4N+23}}L^1_\pel}^{\frac{2(N+2)}{5}}
\ls 
\|f \pel_0^{M}\|_{L^1_x L^1_\pel}^{\frac{4N+23}{15}}.
\end{equation}
As before, we will them control this final term by $\| f\pel_0^N\|_{L^\infty_t([0,T); L^1_x L^1_\pel)}$ and the energy $\| f\pel_0\|_{L^\infty_t([0,T); L^1_x L^1_\pel)}$ via interpolation. 

We will use H{\"o}lder's inequality such that 
$$
\frac{1}{q}+\frac{1}{q'} = 1, \quad 
M=\frac{1}{q}+\frac{N}{q'}
$$
or
$
q=\frac{N-1}{N-M}
$
and 
$
q'=\frac{N-1}{M-1}=\frac{(4N+23)(N-1)}{24N+3}.
$
(As before, notice that the assumptions of the proposition imply that $q,q'\in [1,\infty)$.) Then we use an iterated 
H{\"o}lder inequality, as in \eqref{Holder.T}, to achieve
$$
\|f \pel_0^{M}\|_{ L^1_x L^1_\pel}
\ls  \|f \pel_0^N \|_{ L^1_x L^1_\pel}^{\frac{1}{q'}}=\|f \pel_0^N \|_{ L^1_x L^1_\pel}^{\frac{24N+3}{(4N+23)(N-1)}}.
$$
Combining this with \eqref{KT.int.3}, we obtain the following bound for every $t\in [0,T)$:
\begin{equation}\label{KT.int.4}
\|f \pel_0^4\|_{L_x^{\frac{6(N+2)}{4N+23}}L^1_\pel}^{\frac{2(N+2)}{5}}
\ls 
\|f \pel_0^N\|_{L^1_x L^1_\pel}^{\frac{24N+3}{15(N-1)}}=\|f \pel_0^N\|_{L^1_x L^1_\pel}^{\frac{8N+1}{5(N-1)}}.
\end{equation}
Taking $q_2'$ larger if necessary,  \eqref{KT.int.2} and \eqref{KT.int.4} imply via H\"older's inequality
$$
\|f \pel_0^2\|_{L^1_t([0,T);L_x^{\frac{6(N+2)}{4N+23}}L^1_\pel)}^{\frac{2(N+2)}{5}} 
\ls 
\|f \pel_0^4\|_{L^{q_2'}_t([0,T);L_x^{\frac{6(N+2)}{4N+23}}L^1_\pel)}^{\frac{2(N+2)}{5}}
\ls \|f \pel_0^N\|_{L^{q_2'}_t([0,T);L^1_x L^1_\pel)}^{\frac{4N-7}{5(N-1)}}
$$
and
$$
\|f \pel_0^4\|_{L^1_t([0,T);L_x^{\frac{6(N+2)}{4N+23}}L^1_\pel)}^{\frac{2(N+2)}{5}} 
\ls 
\|f \pel_0^4\|_{L^{q_2'}_t([0,T);L_x^{\frac{6(N+2)}{4N+23}}L^1_\pel)}^{\frac{2(N+2)}{5}}
\ls \|f \pel_0^N\|_{L^{q_2'}_t([0,T);L^1_x L^1_\pel)}^{\frac{8N+1}{5(N-1)}},
$$
where as before, the implicit constants depend on $T$.

Substituting these bounds into \eqref{KT.int.0}, we derive the estimate
\begin{equation}\label{KT.int.5}
\begin{split}
&\| K_T \|_{L^{1}_t([0,T);L^{N+2}_x)}^{N+2}\\
\ls & \ep^{-\frac{N+2}{10}}\|f \pel_0^N\|_{L^{q_2'}_t([0,T);L^1_x L^1_\pel)}^{\frac{4N-7}{5(N-1)}}+\ep^{\frac{3(N+2)}{10}}\|f \pel_0^N\|_{L^{q_2'}_t([0,T);L^1_x L^1_\pel)}^{\frac{8N+1}{5(N-1)}}.
\end{split}
\end{equation}
We now let
$$\ep=\| f\pel_0^N\|_{L^\infty_t([0,T);L^1_xL^1_{\pel})}^{-\frac 2{N-1}}.$$
Hence, \eqref{KT.int.5} implies
\begin{equation*}
\begin{split}
\| K_T \|_{L^{1}_t([0,T);L^{N+2}_x)}^{N+2} \ls & \| f\pel_0^N\|_{L^\infty_t([0,T);L^1_xL^1_{\pel})}^{\frac{N+2}{5(N-1)}}\|f \pel_0^N\|_{L^{q_2'}_t([0,T);L^1_x L^1_\pel)}^{\frac{4N-7}{5(N-1)}}\\
&+\| f\pel_0^N\|_{L^\infty_t([0,T);L^1_xL^1_{\pel})}^{-\frac{3(N+2)}{5(N-1)}}\|f \pel_0^N\|_{L^{q_2'}_t([0,T);L^1_x L^1_\pel)}^{\frac{8N+1}{5(N-1)}}\\
\ls & \| f\pel_0^N\|_{L^\infty_t([0,T);L^1_xL^1_{\pel})}^{\frac{N+2}{5(N-1)}}\|f \pel_0^N\|_{L^{q_2'}_t([0,T);L^1_x L^1_\pel)}^{\frac{4N-7}{5(N-1)}}.
\end{split}
\end{equation*}
This concludes the proof of the proposition.
\end{proof}

For the $K_{S,1}$ term, we have the following bound:

\begin{proposition}\label{KS1.2}
For $N>13$ as in the assumption \eqref{ini.bd.2D.2} in Theorem \ref{main.theorem.2D}, we have the estimate
$$\| K_{S,1}\|_{L^1_t([0,T); L^{N+2}_x)}^{N+2} \ls \| f\pel_0^N\|_{L^\infty_t([0,T); L^1_x L^1_\pel)}^{\alpha},$$
where $\alpha=\frac{N-4}{2(N-1)}<1$ and the implicit constant depends at most polynomially on $T$.
\end{proposition}

\begin{proof}
We use the estimate for $K_{S,1}$ in Proposition \ref{GSbd.2D} and Remark \ref{rem.strichartz} after Theorem \ref{Strichartz}; then we can apply the Strichartz estimates to $|K_{S,1}|$ with 
$\|\frac{K f}{\pel_0} \|_{L^1_\pel}$ as the inhomogeneous term. First, we note that for $k>8$ and $N>4$, the following exponents satisfy the conditions in Theorem \ref{Strichartz}  
\begin{equation}\label{good.exp2}
\begin{split}
q_1=&\frac{3k(N+2)}{(k-3)N+(2k-24)}, 
\quad 
r_1=N+2,\\
 q'_2=&\frac{k(N+2)}{(k-1)N+(2k-8)}, 
 \quad 
 r'_2=\frac{3(N+2)}{2N+7},
\end{split}
\end{equation}
Therefore, using Theorem \ref{Strichartz} with these exponents,
we have  
\begin{equation}
\| K_{S,1} \|_{L^{q_1}_t([0,T);L^{N+2}_x)}^{N+2}
\ls 
 \left\|\frac{K f}{\pel_0} \right\|_{L^{q'_2}_t([0,T);L_x^{\frac{3(N+2)}{2N+7}}L^1_\pel)}^{N+2}.\label{KS2.2.1}
\end{equation}
Further using H{\"o}lder's inequality and allowing a constant depending at most polynomially on $T$, we obtain
\begin{multline}
 \left\|\frac{K f}{\pel_0} \right\|_{L^{q'_2}_t([0,T);L_x^{\frac{3(N+2)}{2N+7}}L^1_\pel)}^{N+2}\ls \| K\|_{L^\infty_t([0,T); L^2_x)}^{N+2}\left\|\frac{f}{\pel_0} \right\|_{L^\infty_t([0,T);L_x^{\frac{6(N+2)}{N+8}}L^1_\pel)}^{N+2}
\\
\ls 
\left\|\frac{f}{\pel_0} \right\|_{L^\infty_t([0,T);L_x^{\frac{6(N+2)}{N+8}}L^1_\pel)}^{N+2},
\label{KS2.2.2}
\end{multline}
using the conservation law in Proposition \ref{cons.law.1}. Applying the interpolation inequality,  Proposition \ref{prop.interpolation}, with 
$$S=-1,
\quad \& \quad 
M=\frac{6(N+2)}{N+8}-2=\frac{4N-4}{N+8},
$$ 
we have  
\begin{equation}\label{KS2.2.3}
\left\|\frac{f}{\pel_0} \right\|_{L^\infty_t([0,T);L_x^{\frac{6(N+2)}{N+8}}L^1_\pel)}^{N+2}
\ls 
\|f \pel_0^{M}\|_{L^\infty_t([0,T); L^1_x L^1_\pel)}^{\frac{N+8}{6}}.
\end{equation}
As in the proof of Proposition \ref{KT.2}, we control this term by interpolating between the moment bound and the conserved energy. 
To this end, we apply\footnote{Notice in particular that $N>4$.} H{\"o}lder's inequality as in \eqref{Holder.T} with the exponent $q'$ (and $q$ determined by $\frac 1q+\frac 1{q'}=1$):  
$$
q'=\frac{(N-1)(N+8)}{3N-12}.
$$
Then  we apply the chain of H{\"o}lder's inequality as in \eqref{Holder.T} to obtain
$$
\|f \pel_0^{M}\|_{ L^1_x L^1_\pel}
=\|f^{\frac{1}{q}} \pel_0^{\frac{1}{q}} f^{\frac{1}{q'}} \pel_0^{\frac{N}{q'}} \|_{L^1_x L^1_\pel}
\ls  \|f \pel_0^N \|_{L^1_x L^1_\pel}^{\frac{1}{q'}}.
$$
We also used the conservation law in Proposition \ref{cons.law.1}.

Combining this last estimate with \eqref{KS2.2.1}, \eqref{KS2.2.2} and \eqref{KS2.2.3}, we have
$$\| K_{S,1} \|_{L^1_t([0,T);L^{N+2}_x)}^{N+2}\ls \| K_{S,1} \|_{L^{\frac{3k(N+2)}{(k-3)N+(2k-24)}}_t([0,T);L^{N+2}_x)}^{N+2}
\ls 
\|f \pel_0^N\|_{L^\infty_t([0,T);L^1_xL^1_\pel)}^{\frac{N-4}{2(N-1)}}$$
as desired.
\end{proof}

Finally, we notice that by Propositions \ref{KT.2D}, \ref{KS2.2D}, and \ref{KT.2}, $K_{S,2}$ obeys the same bounds as $K_T$ and therefore we have

\begin{proposition}\label{KS2.2}
For $N>13$ as in the initial data bound \eqref{ini.bd.2D.2} in the assumptions of Theorem \ref{main.theorem.2D}, there exists $q_2'<\infty$ sufficiently large such that we have the estimate
$$
\| K_{S,2}\|_{L^{1}_t([0,T); L^{N+2}_x)}^{N+2} \ls  \| f\pel_0^N\|_{L^\infty_t([0,T);L^1_xL^1_{\pel})}^{\frac{N+2}{5(N-1)}}\|f \pel_0^N\|_{L^{q_2'}_t([0,T);L^1_x L^1_\pel)}^{\frac{4N-7}{5(N-1)}},
$$
where the implicit constant depends at most polynomially on $T$.
\end{proposition}

Combining Propositions \ref{KT.2}, \ref{KS1.2} and \ref{KS2.2}, we have the following control for the moments of $f$:

\begin{proposition}\label{propagation.moment}
For $N>13$ as in the initial data bound \eqref{ini.bd.2D.2} in the assumptions of Theorem \ref{main.theorem.2D}, we have the estimate
$$\| f\pel_0^N\|_{L^\infty_t([0,T); L^1_x L^1_\pel)}\leq C_1 e^{C_1T^{k_1}}$$
for some constants $C_1>0$ and $k_1>0$.
\end{proposition}

\begin{proof}
By Proposition \ref{prop.moment} and the decompositions in Section \ref{sec.2D.GS}, we have
$$\| f\pel_0^N\|_{L^\infty_t([0,T); L^1_x L^1_\pel)}\ls\| f_0 \pel_0^N\|_{L^1_x L^1_\pel}
+1
+\|K_T\|_{L^1_t([0,T); L^{N+2}_x)}^{N+2}+\|K_S\|_{L^1_t([0,T); L^{N+2}_x)}^{N+2}.
$$
Then using the assumption \eqref{ini.bd.2D.2} on the initial data and Propositions \ref{KT.2}, \ref{KS1.2} and \ref{KS2.2} we obtain
\begin{equation}\label{moment.bd.main.est}
\begin{split}
&\| f\pel_0^N\|_{L^\infty_t([0,T); L^1_x L^1_\pel)}\\
\ls &1+ \| f\pel_0^N\|_{L^\infty_t([0,T); L^1_x L^1_\pel)}^{\frac{N-4}{2(N-1)}}+\| f\pel_0^N\|_{L^\infty_t([0,T);L^1_xL^1_{\pel})}^{\frac{N+2}{5(N-1)}}\|f \pel_0^N\|_{L^{q_2'}_t([0,T);L^1_x L^1_\pel)}^{\frac{4N-7}{5(N-1)}}.
\end{split}
\end{equation}
Note that since $N>13>8$, we have $\frac{N-4}{2(N-1)}>\frac{N+2}{5(N-1)}$. Then notice that either $\| f\pel_0^N\|_{L^\infty_t([0,T); L^1_x L^1_\pel)}\leq 1$ in which case we are done; or we can divide the \eqref{moment.bd.main.est} through by $\| f\pel_0^N\|_{L^\infty_t([0,T); L^1_x L^1_\pel)}^{\frac{N-4}{2(N-1)}}$ to get
\begin{equation}\label{moment.bd.main.est.2}
\begin{split}
\| f\pel_0^N\|_{L^\infty_t([0,T); L^1_x L^1_\pel)}^{\frac{N+2}{2(N-1)}}
\ls &1+\|f \pel_0^N\|_{L^{q_2'}_t([0,T);L^1_x L^1_\pel)}^{\frac{N+2}{2(N-1)}}.
\end{split}
\end{equation}
We can assume without loss of generality (after using H\"older's inequality and losing a constant depending at most polynomially in $T$) that $q_2'\geq \frac{N+2}{2(N-1)}$. The conclusion of the proposition thus follows from Lemma \ref{Gronwall} below (with $p=q_2'$) and noticing that the implicit constant in \eqref{moment.bd.main.est.2} depends at most polynomially on $T$.
\end{proof}

It remains to prove the following Gronwall-type lemma:

\begin{lemma}\label{Gronwall}
Let $g(t)\ge 0$ be a non-decreasing function satisfying
\begin{equation}\label{Gronwall.assumption}
g(t)\leq M(t)(1+ \|g\|_{L^p_s([0,t))})
\end{equation}
for some $p\in [1,\infty)$ and some non-decreasing positive function $M(t)$. Then 
$$g(t)\leq 2 M(t)\, e^{4^p t (M(t))^p}.$$
\end{lemma}

\begin{proof}
We can show this by a standard bootstrap argument. Assume as a bootstrap assumption that
\begin{equation}\label{BA}
g(t)\leq 4\,M(t)\, e^{4^p t (M(t))^p}.
\end{equation}
We will show that under this assumption, the stronger estimate
$$
g(t)\leq 2 M(t)\, e^{4^p t (M(t))^p}
$$
holds. This will imply the conclusion of the lemma. We now begin the proof. By \eqref{Gronwall.assumption} and \eqref{BA}, we have
\begin{multline*}
g(t)\leq 
M(t)(1+4 M(t) (\int_0^t e^{p 4^p s (M(t))^p} ds)^{\frac 1p})
<
M(t)(1+\frac{1}{p^{\frac 1p} } e^{ 4^p t (M(t))^p})
\\
\leq M(t)+M(t)\,e^{4^p t (M(t))^p}\leq 2 M(t)\, e^{4^p t (M(t))^p}.
\end{multline*}
This concludes the proof of the lemma.
\end{proof}

\subsection{Bounds for $\|K\|_{L^\infty_t([0,T);L^\infty_x)}$}\label{sec.K.Linfty.bd}
In this section, we show that once the moment bounds in the previous subsection are obtained, we can prove the $L^\infty$ bound for the electromagnetic field:

\begin{proposition}\label{KLinftybd}
The following $L^\infty$ estimate for $K$ holds:
\begin{equation}\notag
\|K\|_{L^\infty_t([0,T]; L^\infty_x)}\leq C e^{C T^k},
\end{equation}
for some constants $C>0$ and $k>0$.
\end{proposition}

\begin{proof}
In the proof of this proposition, we will allow the implicit constants in $\ls$ to depend on $T$ polynomially as before. First, by Proposition \ref{propagation.moment}, we have $\| f\pel_0^N\|_{L^\infty_t([0,T); L^1_x L^1_\pel)}\ls e^{C_1 T^{k_1}}$ for $N>13$ as in the initial data bound \eqref{ini.bd.2D.2} in the assumptions of Theorem \ref{main.theorem.2D}. Then, by Proposition \ref{prop.interpolation}, we have
\begin{multline}\label{quick.interp}
\|f \pel_0\|_{L^\infty_t([0,T); L^q_x L^1_\pel)}\ls \| f\pel_0^{3q-2}\|_{L^\infty_t([0,T); L^1_x L^1_\pel)}^{1/q}
\\
\ls 1+\| f\pel_0^N\|_{L^\infty_t([0,T); L^1_x L^1_\pel)}\ls e^{C_1 T^{k_1}}.
\end{multline}
Here we use Proposition \ref{prop.interpolation} with $S=1$ and $M=3q-2$ for $1\le q\leq 5$.  
Returning to the Glassey-Schaeffer representation of $K$ from Section \ref{sec.2D.GS} and the estimates from Proposition \ref{GSbd.2D}, we have using 
$$\frac{1}{1+\vh\cdot\xi}\ls \pel_0^2$$
that
\begin{multline}\label{k.upper.k}
|K|(t,x)
\ls 1+\int_0^t \int_{|y-x|\le t-s}\frac{|K|\int_{\mathbb R^2} f \pel_0 d\pel}{\sqrt{(t-s)^2-|y-x|^2}} dy ds
\\
+ 
\int_0^t \int_{|y-x|\le t-s}\frac{\int_{\mathbb R^2} f \pel_0 d\pel}{(t-s)\sqrt{(t-s)^2-|y-x|^2}} dy ds.
\end{multline}
Then notice that for any $q>2$, the conjugate satisfies
$$q'<1+\frac{q'}{2}<2.$$
Therefore, we can bound the singularity of the wave kernel in $L^{\frac{2+q'}{2}}_x$.  We have
$$
\int_{|y-x|\le t-s}\frac{1}{((t-s)^2-|y-x|^2)^{\frac{2+q'}{4}}} dy \ls (t-s)^{{\epsilon}(q)}
$$
for ${\epsilon}(q)=1-\frac{q'}{2}>0$. Therefore, for $2<q\leq 5$, 
This allows us to put $|K|\int_{\mathbb R^2} f \pel_0 d\pel$ in $L^{\frac{q'+2}{q'}}_x$ and we have
\begin{multline*}
\| K f \pel_0 \|_{L^{\frac{q'+2}{q'}}_x L^1_\pel}
\ls 
\|K\|_{ L^{\frac{q'(q'+2)}{2-q'}}_x}
\| f \pel_0 \|_{L^q_xL^1_\pel}
\\
\ls 
\|K\|_{L^{2}_x}^{\frac{2(2-q')}{q'(2+q')}}
\|K\|_{L^\infty_x}^{\frac{(q')^2+4q'-4}{q'(q'+2)}}
\| f \pel_0 \|_{L^q_xL^1_\pel}
\\
\ls 
\|K\|_{L^\infty_x}^{\frac{(q')^2+4q'-4}{q'(q'+2)}}\| f \pel_0 \|_{L^\infty_t([0,T); L^q_xL^1_\pel)},
\end{multline*}
where in the second line we used that $\frac{q'(q'+2)}{2-q'}>2$ and in the last line we have used the conservation law $\|K\|_{L^\infty_t([0,T); L^{2}_x)}\ls 1$ in Proposition \ref{cons.law.1}.
Therefore, for every $0\le t\leq T$, we use \eqref{k.upper.k} to get
\begin{multline*}
\|K(t)\|_{ L^\infty_x}
\ls \int_0^t \|K(s)\|_{L^\infty_x}^{\frac{(q')^2+4q'-4}{q'(q'+2)}}
\| f \pel_0 \|_{L^\infty_t([0,T); L^q_xL^1_\pel)} (t-s)^{{\epsilon}(q)}ds
\\
+ \int_0^t \| f \pel_0 \|_{L^\infty_t([0,T); L^q_xL^1_\pel)} (t-s)^{-1+\epsilon(q)} ds
\\
\ls e^{C_1T^{k_1}}\left(\int_0^t \|K(s)\|_{L^\infty_x}^{\frac{(q')^2+4q'-4}{q'(q'+2)}}ds+1\right).
\end{multline*}
Taking the supremum over $t$, we thus have
$$\|K\|_{L^\infty([0,T); L^\infty_x)}\ls  e^{C_1T^{k_1}}\left(\|K\|_{L^\infty_t([0,T); L^\infty_x)}^{\frac{(q')^2+4q'-4}{q'(q'+2)}}+1\right), $$
where the implicit constant depends polynomially on $T$. Since $\frac{(q')^2+4q'-4}{q'(q'+2)}<1$, this allows one to establish a bound for $\|K\|_{L^\infty_t([0,T); L^\infty_x)}$ that is growing at most $Ce^{CT^k}$ after taking $C$ and $k$ to be sufficiently large.
\end{proof}

\subsection{Higher Regularity}\label{sec.2D.hr}
In this subsection, we use the $L^\infty_x$ bound for $K$ to obtain estimates for the first derivatives of $K$ and $f$. From now on, in order to lighten the notation, we will allow \emph{arbitrary} dependence of the implict constants in $\ls$ on $T$. 

In the proposition below, and in the rest of this section, we use the notation
\begin{equation}\label{forw.def}
\fowC(t) \eqdef
1+\sup_{s\in[0,t],x,\pel\in \mathbb R^2} 
\left(\left| \nabla_{x,p}X(s;0,x,\pel)  \right|
+
\left| \nabla_{x,p}V(s;0,x,\pel)  \right|
\right)
\end{equation}
and
\begin{equation}\label{back.def}
\bakC(t) \eqdef
1+\sup_{s\in[0,t],x,\pel\in \mathbb R^2} 
\left(\left| \nabla_{x,p}X(0;s,x,\pel)  \right|
+
\left| \nabla_{x,p}V(0;s,x,\pel)  \right|
\right).
\end{equation}
Here $\fowC(t)$ estimate the maximum values of the forward characteristics up to time $t$, and $\bakC(t)$ similarly  estimate the maximum values of the backward characteristics.  These are similarly defined in $2\frac 12$D.   
We now use the $L^\infty_x$ bounds for $K$ and the ODE's for the characteristics of the Vlasov equation \eqref{char1} and \eqref{char2} to control the particle density $f$.  
We  have the estimates

\begin{proposition}\label{char.est.prop}
We have the following bounds:
\begin{equation}\label{fbd}
\sup_{t\in [0,T],~x\in \mathbb R^2}\int_{\mathbb R^2} f(t,x,\pel) \pel_0^3 d\pel\ls 1,
\end{equation}
and 
\begin{equation}\label{dfbd}
\sup_{t\in [0,T],~x\in \mathbb R^2}\int_{\mathbb R^2} \left|(\nabla_{x,\pel} f)(t,x,\pel) \right| \pel_0^3 d\pel
\ls 
\bakC(T),
\end{equation}
and further
\begin{equation}
\label{dfbd.2}
\sup_{t\in [0,T],~x,\pel\in \mathbb R^2} \left|(\nabla_{x,\pel} f)(t,x,\pel) \right|
\ls 
\bakC(T).
\end{equation}
\end{proposition}

\begin{proof}
We first integrate along the characteristics \eqref{char1} and \eqref{char2} with \eqref{char.data} to obtain the standard formula
\begin{equation}\label{along.char}
f(t,x,\pel) = f_0(X(0;t,x,\pel), V(0;t,x,\pel)).
\end{equation}
We use this formula \eqref{along.char} along with the chain rule, the initial assumptions \eqref{ini.bd.2D.3}, \eqref{ini.bd.2D.4}, \eqref{ini.bd.2D.5} and the estimate in Proposition \ref{KLinftybd} to establish the conclusion of the proposition.  In particular, we use  that Proposition \ref{KLinftybd} directly implies that for any finite time interval $[0,T]$, the momentum-distance ``travelled'' by a characteristic is bounded by a uniform constant depending  only on $T$ and the conservation laws.
\end{proof}

Then, to close these estimates, we thus need to control the derivatives of the characterisitcs. This in turn requires us to control the derivatives of $E$ and $B$. Slightly abusing notation, we denote by $\nab_x K_S$ and $\nab_x K_T$ the sum of the absolute values of all spatial derivatives of $E_S$, $B_S$, etc. More precisely,
$$
|\nab_xK| \eqdef \sum_{i,j}(|\nab_{x^i}E^j|+|\nab_{x^i}B^j|),
\quad 
|\nab_xK_S| \eqdef \sum_{i,j}(|\nab_{x^i}E_S^j|+|\nab_{x^i}B_S^j|),
$$
$$
|\nab_xK_T| \eqdef \sum_{i,j}(|\nab_{x^i}E_T^j|+|\nab_{x^i}B_T^j|)
$$
Recalling the Glassey-Schaeffer decomposition in Section \ref{sec.2D.GS}, in order to estimate the first derivatives of $K$, we need to control the first derivatives of $\tilde K_0$, $K_S$ and $K_T$. As remarked in Section \ref{sec.2D.GS}, the first derivatives of $\tilde K_0$ are bounded a priori by the initial data norms \eqref{ini.bd.2D.2} - \eqref{ini.bd.2D.5} for $f_0$. It is shown by Glassey-Schaeffer \cite{GS2D1} that $\nabla_x K_S$ and $\nabla_x K_T$ can be further decomposed into $\nab_xK_{SS}$, $\nab_xK_{ST}$, $\nab_xK_{TS}$ and $\nab_xK_{TT}$ with the following estimates.\footnote{
In Theorem 3 of \cite{GS2D1}, the bounds for the derivatives of $K$ were derived without tracking the exact dependence on $\pel_0$. This is sufficient for the argument in \cite{GS2D1} since the momentum support is bounded. In our situation, we need a more precise bound in terms of $\pel_0$. Nevertheless, for the terms with the magnetic field $B$, this can be easily read off after performing integration by parts as in \cite{GS2D1} on the formulas (4.2), (4.3) and (4.12) in \cite{GS2D1}. The analogous estimates for $E$ can be derived similarly.
}

\begin{proposition}[Glassey-Schaeffer \cite{GS2D1}]\label{dKdecomposition}
$K_S$ and $K_T$ can be decomposed as
$$K_S=K_{SS}+K_{ST},\quad K_T=K_{TS}+K_{TT}$$
such that the following bounds hold:
$$|\nab_xK_{SS}|\ls \int_0^t\int_{|y-x|\le t-s} \int_{\mathbb R^2} \frac{\pel_0^3 (|K|^2f)(s,y,\pel)}{\sqrt{(t-s)^2-|y-x|^2}}d\pel\, dy\, ds,
$$
and\footnote{Strictly speaking in \cite{GS2D1}, Glassey-Schaeffer obtain $|\partial_t K|$ instead of $\rho$ in the upper bound of $|\nab_xK_{ST}|$.   However one easily obtains the bound of $\rho$ by using equation \eqref{maxwell}.
}
\begin{equation*}
\begin{split}
|\nab_xK_{ST}|\ls &\mbox{Data}+\int_0^t\int_{|y-x|\le t-s} \int_{\mathbb R^2} \frac{\pel_0^3\,\big(|K| f\big)(s,y,\pel)}{(t-s)\sqrt{(t-s)^2-|y-x|^2}}d\pel\, dy\, ds\\
&+\int_0^t\int_{|y-x|\le t-s} \int_{\mathbb R^2} \frac{\pel_0^3\,\big((|\nabla_x K|+\rho)f\big)(s,y,\pel)}{\sqrt{(t-s)^2-|y-x|^2}}d\pel\, dy\, ds,
\end{split}
\end{equation*}
and further
$$
|\nab_xK_{TS}|\ls \mbox{Data}+\int_0^t\int_{|y-x|\le t-s} \int_{\mathbb R^2} \frac{\pel_0^3(|K|f)(s,y,\pel)}{(t-s)\sqrt{(t-s)^2-|y-x|^2}}d\pel\, dy\, ds,
$$
and lastly (for any small $\delta>0$) we have
\begin{equation*}
\begin{split}
|\nab_xK_{TT}|\ls &\mbox{Data}+\int_0^{t-\delta}\int_{|y-x|\le t-s} \int_{\mathbb R^2} \frac{\pel_0^3\,f(s,y,\pel)}{(t-s)^2\sqrt{(t-s)^2-|y-x|^2}}d\pel\, dy\, ds\\
&+\int_{t-\delta}^t\int_{|y-x|\le t-s} \int_{\mathbb R^2} \frac{\pel_0^3\,|\nabla_{x,\pel} f|(s,y,\pel)}{(t-s)\sqrt{(t-s)^2-|y-x|^2}}d\pel\, dy\, ds\\
&+\frac{1}{\delta}\int_{|y-x|\le \delta} \int_{\mathbb R^2} \frac{\pel_0^3\,f(s=t-\delta,y,\pel)}{\sqrt{\delta^2-|y-x|^2}}d\pel\, dy,
\end{split}
\end{equation*}
where ``$\mbox{Data}$'' denotes a term that can be bounded\footnote{More precisely, these terms can be bounded point wise by
$$\frac{1}{t}\int_{|y-x|\leq t}\int_{\mathbb R^2} \frac{\pel_0 f_0}{\sqrt{t^2-|y-x|^2}}\,d\pel\,dy,\quad \frac{1}{t}\int_{|y-x|\leq t}\int_{\mathbb R^2} \frac{\pel_0 |K_0| f_0}{\sqrt{t^2-|y-x|^2}}\,d\pel\,dy.$$
One can then easily show that these terms can be dominated by the norms $\eqref{ini.bd.2D.2} - \eqref{ini.be.2D.6}$ after applying Sobolev embedding theorem.} depending only on the initial data norms \eqref{ini.bd.2D.2} - \eqref{ini.be.2D.6} for $f_0$, $E_0$ and $B_0$.
\end{proposition}

In order to obtain the bounds for the derivatives of $K$, we need to prove at the same time the estimates for the derivatives of the characteristics. To achieve this, we need the following simple lemma connecting the bounds for the forward and backward characteristics (see Lemma 3.1 in \cite{KS}).

\begin{lemma}\label{lemm.forw.back}
For any $t\in [0, T_*]$ we have 
$$
\bakC(t) \ls \fowC(t)^{3+i},
$$
where $i=0$ in the $2$D case and $i=1$ in $2\frac 12$D.
\end{lemma}

This lemma is proven in our companion paper \cite[Section 5]{LS}.  It allows us to derive higher regularity for $K$ and $f$:

\begin{proposition}\label{prop.higher.reg} We have the following uniform bounds for $K$ and $f$:
$$\|\nabla_x K\|_{L^\infty_t([0,T]; L^\infty_x)}\ls 1,
\quad 
\|\nabla_{x,\pel} f\|_{L^\infty_t([0,T]; L^\infty_x L^\infty_\pel )}\ls 1.$$
In particular, after additionally using \eqref{char1} and \eqref{char2}, we see that the characteristics for the Vlasov equation are Lipschitz.
\end{proposition}

\begin{proof}
For this proof, we will estimate $\fowC(t)$.  We will derive an estimate for $\fowC(t)$ by controlling $\nabla_x K$. First, as remarked in the beginning of Section \ref{sec.2D.GS}, the terms derivatives of the terms $\tilde{E}^i_0$ and $\tilde B_0$ are a priori bounded by the data norms.

We now control the terms in Proposition \ref{dKdecomposition}. Combining Proposition \ref{dKdecomposition}, Proposition \ref{KLinftybd} and \eqref{fbd} we have the estimate
$$
|\nabla_x K_{SS}|+|\nabla_x K_{TS}|\ls 1.
$$
The data term and the second term in the estimates for $\nabla_x K_{ST}$ in Proposition \ref{dKdecomposition} are bounded using Proposition \ref{KLinftybd} and \eqref{fbd} just as above. The third term for $\nabla_x K_{ST}$ can be controlled, using \eqref{fbd}, as
\begin{equation*}
\begin{split}
&\int_0^t\int_{|y-x|\le t-s} \int_{\mathbb R^2} \frac{\pel_0^3\,\big((|\nabla_x K|+\rho)f\big)(s,y,\pel)}{\sqrt{(t-s)^2-|y-x|^2}}d\pel\, dy\, ds\\
\ls &1+\int_0^t \sup_{y\in \mathbb R^2} |\nabla_x K(s,y)| ds.
\end{split}
\end{equation*}
In particular, in the above, we have used $\sup_{s,y}|\rho(s,y)|\ls $ by \eqref{fbd}.
Lastly we estimate $\nabla_x K_{TT}$. The data term is bounded as just before. The second term can also be controlled using \eqref{fbd} but suffers a logarithmic loss in the integration in time. More precisely,
$$
\int_0^{t-\delta}\int_{|y-x|\le t-s} \int_{\mathbb R^2} \frac{\pel_0^3\,f(s,y,\pel)}{(t-s)^2\sqrt{(t-s)^2-|y-x|^2}}d\pel\, dy\, ds
\ls 1+ \left| \log\left(\frac{t}{\delta}\right) \right|.
$$
The third term in the $\nabla_x K_{TT}$ upper bound can be estimated using \eqref{dfbd} as
$$
\int_{t-\delta}^t\int_{|y-x|\le t-s} \int_{\mathbb R^2} \frac{\pel_0^3\,|\nabla_{x,\pel} f|(s,y,\pel)}{(t-s)\sqrt{(t-s)^2-|y-x|^2}}d\pel\, dy\, ds
\ls \delta \bakC(t).
$$
The fourth term in the $\nab_x K_{TT}$ upper bound can be controlled using \eqref{fbd} by
$$\frac{1}{\delta}\int_{|y-x|\le \delta} \int_{\mathbb R^2} \frac{\pel_0^3\,f(s=t-\delta,y,\pel)}{\sqrt{\delta^2-|y-x|^2}}d\pel\, dy\ls 1$$
since
$$\frac{1}{\delta}\int_{|y-x|\le \delta} \frac{1}{\sqrt{\delta^2-|y-x|^2}} dy\ls 1.$$
We now let $\delta$ be defined by
$$
\delta=\frac{t}{1+\bakC(t)}.
$$
Therefore, collecting all the prior estimates, we have
$$
|\nabla_x K_{TT}(t,x)|\ls 1+\left|\log\big(1+\bakC(t)\big) \right|.
$$
Combining all the estimates for $\nabla_x K_{SS}$, $\nabla_x K_{TS}$, $\nabla_x K_{ST}$ and $\nabla_x K_{TT}$, we have
\begin{equation*}
|\nabla_x K(t,x)|
\ls 1+\int_0^t \sup_{y\in \mathbb R^2} |\nabla_x K(s,y)| ds+\left|\log\big(1+\bakC(t)\big) \right|.
\end{equation*}
By Gronwall's inequality, we therefore have
\begin{equation}\label{hr.1}
\sup_{x\in\mathbb R^2}|\nabla_x K(t,x)|
\ls 1+\left|\log\big(1+\bakC(t)\big) \right|.
\end{equation}
After differentiating and then integrating the characteristics \eqref{char1} and \eqref{char2} in time we have
\begin{equation}\label{hr.2}
\fowC(t)\ls 1+\int_0^t \fowC(s)(1+\sup_x|(\nabla_x K)(s,x)|) ds.
\end{equation}
Combining \eqref{hr.1} and \eqref{hr.2}, and using Lemma \ref{lemm.forw.back}, we have
$$\fowC(t)\ls 1+\int_0^t \fowC(s)\left(1+\left|\log\big(1+\fowC(s)\big) \right| \right) ds.$$
This implies by a bootstrap argument that
\begin{equation}\label{first.c.bound}
\fowC(t)\ls 1.
\end{equation}
Returning to \eqref{hr.1}, and applying Lemma \ref{lemm.forw.back} again, we therefore also have
$$\sup_{x\in\mathbb R^2}|\nabla_x K(t,x)| \ls 1.$$
Finally, by \eqref{dfbd.2}, \eqref{first.c.bound} and Lemma \ref{lemm.forw.back}, we also have
$\|\nabla_{x,\pel} f\|_{L^\infty_t([0,T]; L^\infty_x L^\infty_\pel)}\ls 1$
as desired.
\end{proof}

\subsection{Conclusion of proof of Theorem \ref{main.theorem.2D}.}\label{sec.2D.con}
On any finite time interval $[0,T)$, we have now obtained the a priori bounds 
$$\|K\|_{L^\infty([0,T);L^{\infty}_x)}+\|\nab_x K\|_{L^\infty([0,T);L^{\infty}_x)}\ls 1.$$ 
Moreover, using the bound \eqref{first.c.bound} for the first derivatives of the characteristics, we can estimate $\|\pel_0^3\nab_{x,\pel}f\|_{L^\infty([0,T);L^\infty_xL^1_{\pel})}$ and $\|\nab_{x,\pel}f\|_{L^\infty([0,T);L^\infty_xL^\infty_{\pel})}$ by \eqref{dfbd} and \eqref{dfbd.2} respectively. This in turn implies
\begin{multline*}
\sup_{t\in [0,T],~x\in \mathbb R^2}\int_{\mathbb R^2} \left|(\nabla_{x,\pel} f)(t,x,\pel) \right|^2 w_2(p)^2 d\pel
\\
\ls 
\|\nab_{x,\pel}f\|_{L^\infty([0,T);L^\infty_xL^\infty_{\pel})}
\|\nabla_{x,\pel} f  w_2^2 \|_{L^\infty([0,T);L^\infty_x L^{1}_\pel)}
\ls 1.
\end{multline*}
By Theorem \ref{theorem.local.existence}, we have therefore concluded the proof of Theorem \ref{main.theorem.2D}.
\hfill Q.E.D.

\section{The two-and-one-half dimensional case}\label{2hD.sec}

We now turn to the $2\frac12$D problem. Recall from the introduction that we now have all components of $E$ and $B$, i.e., $E$ and $B$ satisfy \eqref{25D.ansatz}.  Moreover, $\pel=(\pel_1,\pel_2,\pel_3)$ is now a $3$-dimensional vector and according to \eqref{vh.def} we have
$$\pel_0=\sqrt{1+\pel_1^2+\pel_2^2+\pel_3^2}.$$
Furthermore if we are using a two dimensional vector, such as $\xi =(\xi_1, \xi_2)$,   then in this section we will always consider it to be a three dimensional vector as for instance $\xi =(\xi_1, \xi_2, 0)$, or $\omega=(\omega_1, \omega_2, 0)$, which allows us to make sense of dot products and cross products.

As pointed out in Section \ref{2.5D.intro}, our strategy is to apply the conservation law used in \cite{GS2.5D} (see Proposition \ref{add.cons.law} below) to show that in the $2\frac 12$-D problem, the $2$-D interpolation inequalities as well as the bounds in \eqref{KS1.gs.est} and Propositions \ref{KT.2D} and \ref{KS2.2D} also hold. We can then follow Sections \ref{sec.2D.moment} and \ref{sec.2D.hr} to conclude global existence and uniqueness of solutions.

\subsection{Local existence}\label{sec.2hD.local.existence}
In order to prove global existence and uniqueness of solutions in the $2\frac 12$-dimensional case, we need a theorem on the local existence and a continuation criterion analogous to Theorem \ref{theorem.local.existence}. We will state the $2\frac 12$-dimensional theorem below. The only difference in the statement of Theorem \ref{theorem.2h.local.existence} below compared to Theorem \ref{theorem.local.existence} is that we need extra weights in $\pel_0$ in the definition of the energies. Notice that while in the proof of Theorem \ref{theorem.local.existence}, we have used the two-dimensional Gagliardo-Nirenberg inequality \eqref{GN} and Sobolev embedding \eqref{SE.used}, it is only applied in the spatial variables. Therefore, they can also be applied in the context of the $2\frac 12$ dimensional relativistic Vlasov-Maxwell system. The proof of Theorem \ref{theorem.2h.local.existence} is the same as the proof of Theorem \ref{theorem.local.existence} and will be omitted.

\begin{theorem}\label{theorem.2h.local.existence}
Given initial data $(f_0(x,\pel),E_0(x),B_0(x))$ to the $2\frac 12$-dimensional relativistic Vlasov-Maxwell system which satisfies \eqref{25D.ansatz}, the constraints \eqref{constraints} and for $\nuD \ge 3$ that
\begin{equation}\label{f.energy.est.2hD}
\Dinit\eqdef \sum_{0\leq k\leq \nuD}\left(\|\nab_x^k K_0 \|_{L^2_x }^2
+
\|w_3\nab_{x,\pel}^k f_0\|_{L^2_x L^2_{\pel}}^2\right)<\infty,
\end{equation}
Then there exists a $T=T(\Dinit,\nuD)>0$ such that there exists a unique local solution to the relativistic Vlasov-Maxwell system in $[0,T]$ where the bound
\begin{equation}\label{f.energy.T.2hD}
\enerD \eqdef
\sum_{0\leq k\leq \nuD}\left(\|\nab_x^k K \|_{L^\infty_t ([0,T];L^2_x)}^2
+
\|w_3\nab_{x,\pel}^k f\|_{L^\infty_t([0,T]; L^2_x L^2_{\pel})}^2\right)
\ls \Dinit 
\end{equation}
holds. Moreover, if $[0,T_*)$ is the maximal time interval of existence and uniqueness and $T_*< +\infty$
then 
$
\lim_{s \uparrow T_*}
\|
 \mathcal{A}
\|_{L^1_t ([0,s))}=+\infty,
$
where
\begin{equation}\notag 
 \mathcal{A}(t)
\eqdef
\|(K,\nab_x K)\|_{L^\infty_x}(t)
+\|w_3\nab_{x,\pel} f\|_{L^\infty_x L^2_{\pel}}(t).
\end{equation}
\end{theorem}

\subsection{An additional conservation law}
In \cite{GS2.5D}, Glassey-Schaeffer observed that the solution to the $2\frac 12$-dimensional relativistic Vlasov-Maxwell system obeys an additional conservation law
$$V_3(t)+A_3(t,X(t),V(t)) = \mbox{constant}$$
along every characteristic \eqref{char1}-\eqref{char2}. 
Here, $V_3$ is the $\pel_3$ component of the characteristic; and $(\phi, A)$ is the electromagnetic potential satisfying
$$E=-\nabla_x \phi-\frac{\partial}{\partial t} A,\quad B=\nabla_x\times A$$
and $A_3$ is the $x_3$ component of $A$. One requires additionally that 
$$\frac{\partial}{\partial x^3} \phi=\frac{\partial}{\partial x^3} A^i=0,\quad\mbox{for }i=1,2,3.$$
Such a gauge exists due to the symmetry assumption on $E$ and $B$. Moreover, although there is still a gauge freedom, the $A_3$ component is independent of the choice of gauge within this class.

Furthermore, Glassey-Schaeffer \cite{GS2.5D} showed that the conservation laws mentioned in Section \ref{sec.cons.law} are sufficiently strong to give an a priori control of $A_3$, which then implies an a priori bound on the $\pel_3$ distance travelled by any characteristic over any finite time interval. We summarize this in the following proposition\footnote{We note that the proof of Lemma 3.2 in \cite{GS2.5D} does not require the initial momentum support to be bounded. It does, however, need that $|\tilde{A}_3|$ to be bounded on $[0,T]$ for $\tilde{A}_3$ being the solution to $\Box \tilde{A}_3=0$ with the same initial data as $A_3$. Notice that $A_3(0)=-\Delta^{-1}(\partial_1 B_2-\partial_2 B_1)(0)$ and $(\partial_t A_3)(0)=-E_3(0)$. The assumption \eqref{ini.be.2D.2h.7} is thus more than sufficient to guarantee that $|\tilde{A}_3|$ remains bounded.}:

\begin{proposition}[Glassey-Schaeffer (Lemmas 2.1 and 3.2, \cite{GS2.5D})]\label{add.cons.law}
Along every characteristic we have
\begin{equation}\label{25D.add.cons.law} 
V_3(t;s,x,\pel)+A_3(t,X(t;s,x,\pel),V(t;s,x,p)) = \mbox{constant}(s,x,\pel).
\end{equation}
Moreover, for any $T>0$ we have
$$\sup_{0\le t\leq T} |A_3(t)|\leq C_T <\infty,$$
for a constant $C_T$ depending only on $T$ and the initial data. 

Further, along every characteristic, we have
$$\sup_{0\le s,t,r\leq T}|V_3(s;r,x,p)-V_3(t;r,x,p)|\leq C_T<\infty,$$
for a (different) constant $C_T$ depending only on $T$.
\end{proposition}

In the context of \cite{GS2.5D}, this gives an a priori bound for $|\pel_3|$ in the momentum support. This allowed Glassey-Schaeffer to prove global existence and uniqueness of solutions using the strategy in the $2$-D setting \cite{GS2D1}, \cite{GS2D2}. In our present paper, however, $|\pel_3|$ is not bounded even in the initial data and we cannot get any a priori bounds for $|\pel_3|$. Nevertheless, Proposition \ref{add.cons.law} implies an integral estimate (see Proposition \ref{add.cons.law.2}).

First let us point out that the implicit constant in Proposition \ref{add.cons.law.2} below, and more generally all the implicit constants in the remainder of this section, will  depend also on the constant in the conservation law \eqref{25D.add.cons.law} in Proposition \ref{add.cons.law}, in addition to the constant bounds in the previous conservation laws from Propositions \ref{cons.law.1} - \ref{cons.law.3}. Moreover, unless otherwise stated, the constants will also be allowed to depend \emph{polynomially} on $T$.

\begin{proposition}\label{add.cons.law.2} We have the following estimate for a solution
\begin{equation}\label{mom.2.5.bound}
\left\| \int_{-\infty}^{\infty} f(t,x,\pel) \langle\pel_3 \rangle^{5+\delta} dp_3 \right\|_{L^\infty_t([0,T]; L^\infty_x L^\infty_{(\pel_1,\pel_2)})} \ls 1,
\end{equation}
where $\delta$ is as in \eqref{ini.bd.2D.2h.4} in the assumptions of Theorem \ref{main.theorem.2D.2h}.
\end{proposition}

Proposition \ref{add.cons.law.2} is easily proved by combining the representation \eqref{along.char}, the initial data assumption \eqref{ini.bd.2D.2h.4}, and Proposition \ref{add.cons.law}.

Given Proposition \ref{add.cons.law.2}, we can obtain improved interpolation inequalities (see Proposition \ref{imp.moment} below) as well as bounds for $E$ and $B$ analogous to Propositions \ref{KT.2D} and \ref{KS2.2D} and \eqref{KS1.gs.est}.

In particular, the interpolation inequality Proposition \ref{prop.interpolation.0} applies in $2\frac 12$ dimensions, but it would require different exponents from the $2$ dimensional case.  Below we prove the interpolation in Proposition \ref{imp.moment} in $2\frac 12$ dimensions, with the same exponents as in the $2$ dimensional case,  by assuming that $g=g(x,p)\ge 0$ satisfies \eqref{mom.2.5.bound} instead of $g\in { L^\infty_x L^\infty_\pel}$.

\begin{proposition}\label{imp.moment}
For $\min\{M,5+\delta\}\geq S>-2$, the following estimate holds:
$$
\| \pel_0^S g \|_{ L^{\frac{M+2}{S+2}}_x L^1_\pel}\ls
\| \pel_0^M g \|_{ L^1_x L^1_\pel}^{\frac{S+2}{M+2}}.
$$
Above the implied constant depends only on 
$\left\|  \int_{\mathbb{R}} g
\langle\pel_3 \rangle^{5+\delta}dp_3 \right\|_{ L^\infty_x L^\infty_{(\pel_1, \pel_2)}}$.
\end{proposition}

\begin{proof}
Fix $R>0$ to be determined at the end of the proof.  
We divide the domain of integration into the regions 
$\langle (\pel_1,\pel_2) \rangle\leq R$ and 
$\langle (\pel_1,\pel_2) \rangle> R$. 
On the bounded region we have
\begin{multline*}
\int_{\langle (\pel_1,\pel_2) \rangle\leq R} \pel_0^S g(x,\pel) d\pel 
\ls 
\int_{\langle (\pel_1,\pel_2) \rangle\leq R}  \langle (\pel_1,\pel_2) \rangle^{S} 
\langle \pel_3 \rangle^{|S|}  g(x,\pel) d\pel
\\
\ls 
\left\|  \int_{\mathbb{R}} g \langle\pel_3 \rangle^{|S|} dp_3 \right\|_{ L^\infty_x L^\infty_{(\pel_1, \pel_2)}}
\iint_{\langle (\pel_1,\pel_2) \rangle\leq R}
\langle (\pel_1,\pel_2) \rangle^{S} 
 d\pel_1\,d\pel_2
\ls R^{S+2}
\end{multline*}
In the unbounded region where $\langle (\pel_1,\pel_2) \rangle> R$, we have
$$
\int_{\langle (\pel_1,\pel_2) \rangle> R} \pel_0^S g(x,\pel) d\pel 
\leq R^{-(M-S)}\int_{\mathbb{R}^3} \pel_0^M g(x,\pel) d\pel.
$$
We interpolate by choosing $R=\big(\int_{\mathbb{R}^3} \pel_0^M g(x,\pel) d\pel\big)^{\frac{1}{M+2}}$ to get
$$
\int_{\mathbb{R}^3} \pel_0^S g(x,\pel) d\pel \ls \left(\int_{\mathbb{R}^3} \pel_0^M g(x,\pel) d\pel\right)^{\frac{S+2}{M+2}}.
$$
Taking $\frac{M+2}{S+2}$-th power and integrating in $x$ gives the desired inequality.
\end{proof}
Using Proposition \ref{imp.moment}, we also show that an analogue of Proposition \ref{prop.moment} holds with the same exponents as in the $2$-dimensional case:

\begin{proposition}\label{prop.moment.2h}
Let $(f, E, B)$ be a solution to the $2\frac 12$-dimensional relativistic Vlasov-Maxwell system. For $N> 0$ we have the estimate
$$
\| \pel_0^N f \|_{L^\infty_t([0,T); L^1_x L^1_\pel)}
\ls 
\| \pel_0^N f_0 \|_{L^1_x L^1_\pel}+\| K\|_{L^1_t([0,T); L^{N+2}_x)}^{N+2}.
$$
\end{proposition}

\begin{proof}
Differentiating the $N$-th moment of \eqref{vlasov} in time and integrating by parts in $x$ and $\pel$, we get
$$
\frac{d}{dt}\| \pel_0^N f(t) \|_{ L^1_x L^1_\pel}\ls \| \pel_0^{N-1} f(t) |K(t)|\|_{ L^1_x L^1_\pel}.$$
Then H\"older's inequality implies that
$$\frac{d}{dt}\| \pel_0^N f(t) \|_{ L^1_x L^1_\pel}\ls \| \pel_0^{N-1} f(t) \|_{ L^{\frac{N+2}{N+1}}_x L^1_\pel}\|K(t)\|_{ L^{N+2}_x}.$$
We use Proposition \ref{imp.moment} with $M=N$, $S=N-1$ to obtain
$$\| \pel_0^{N-1} f(t) \|_{ L^{\frac{N+2}{N+1}}_x L^1_\pel}\ls \| \pel_0^N f(t) \|_{ L^1_x L^1_\pel}^{\frac{N+1}{N+2}},$$
which implies
$$\frac{d}{dt}\| \pel_0^N f(t) \|_{ L^1_x L^1_\pel}\leq \| \pel_0^N f(t) \|_{ L^1_x L^1_\pel}^{\frac{N+1}{N+2}}\|K(t)\|_{ L^{N+2}_x}.$$
Dividing both sides by $\| \pel_0^N f(t) \|_{ L^1_x L^1_\pel}^{\frac{N+1}{N+2}}$, we get
$$\frac{d}{dt}\| \pel_0^N f(t) \|_{ L^1_x L^1_\pel}^{\frac{1}{N+2}}\ls \|K(t)\|_{ L^{N+2}_x}.$$
Integrating in time gives the desired result.
\end{proof}
\subsection{The Glassey-Schaeffer decomposition}\label{sec.GS.2h}

As mentioned before, the estimate in Proposition \ref{add.cons.law.2} implies that we can use the $2\frac 12$ dimensional Glassey-Schaeffer representation of the electromagnetic field and show that it obeys similar estimates as in the two dimensional case.

More precisely, in Theorem 4.1 in \cite{GS2.5D}, Glassey-Schaeffer decomposed the electromagnetic fields into the following terms
$$E^i(t,x)=E^i=\tilde{E}_0^i+E^i_S+E^i_T, $$
and
$$B^i(t,x)=B^i=\tilde{B}_0^i+B^i_S+B^i_T,$$
where $\tilde{E}^i_0$ and $\tilde{B}^i_0$ depend only on the initial data.\footnote{As in the $2$-dimensional case, these terms have $C^1$ and $H^3$ norms depending only on the norms of the initial data in Theorem \ref{main.theorem.2D.2h}.}

Recall that we have $\xi\eqdef\frac{y-x}{t-s}$. 
The other terms in $E^i$ are
$$E^i_T= \int_0^t \int_{|y-x|\leq t-s}\int_{\mathbb R^2} \frac{e^i_T f(s,y,\pel)}{(t-s)\sqrt{(t-s)^2-|y-x|^2}} d\pel\, dy\, ds$$
and
\begin{equation*}
\begin{split}
E^i_S
=\int_0^t \int_{|y-x|\leq t-s}\int_{\mathbb R^2} \frac{\partial e^i_S}{\partial \pel^j}\frac{((E+\vh\times B)_j f)(s,y,\pel)}{\sqrt{(t-s)^2-|y-x|^2}} d\pel\, dy\, ds
\end{split}
\end{equation*}
and the other terms in $B$
are
$$
B_T^i= \int_0^t \int_{|y-x|\leq t-s}\int_{\mathbb R^2} \frac{b^i_T f(s,y,\pel)}{(t-s)\sqrt{(t-s)^2-|y-x|^2}} d\pel\, dy\, ds
$$
and
\begin{equation*}
\begin{split}
B^i_S
=&\int_0^t \int_{|y-x|\leq t-s}\int_{\mathbb R^2} \frac{\partial b^i_S}{\partial \pel^j}\frac{((E+\vh\times B) f)(s,y,\pel)}{\sqrt{(t-s)^2-|y-x|^2}} d\pel\, dy\, ds.
\end{split}
\end{equation*}
where 
\begin{equation*}
\begin{split}
e^i_T=&\frac{-2(1-\vh_1^2-\vh_2^2)(\xi_i+\vh_i)}{(1+\vh\cdot\xi)^2}\quad\mbox{for }i=1,2,\\
e^3_T=&\frac{2\vh_3(\vh_1(\xi_1+\vh_1)+\vh_2(\xi_2+\vh_2))}{(1+\vh\cdot\xi)^2},\\
e^i_S=&\frac{-2(\xi_i+\vh_i)}{(1+\vh\cdot\xi)}\quad\mbox{for }i=1,2,\\
e^3_S=&\frac{-2\vh_3}{(1+\vh\cdot\xi)},
\end{split}
\end{equation*}
and
\begin{equation*}
\begin{split}
b^1_T=&\frac{2\vh_3((1+\xi_1\vh_1)(\xi_2+\vh_2)-\xi_2\vh_1(\xi_1+\vh_1))}{(1+\vh\cdot\xi)^2},\\
b^2_T=&\frac{2\vh_3(\xi_1\vh_2(\xi_2+\vh_2)-(1+\xi_2\vh_2)(\xi_1+\vh_1))}{(1+\vh\cdot\xi)^2},\\
b^3_T=&\frac{2((\vh_2+\xi_2(\vh_1^2+\vh_2^2))(\xi_1+\vh_1)-(\vh_1+\xi_1(\vh_1^2+\vh_2^2))(\xi_2+\vh_2))}{(1+\vh\cdot\xi)^2},\\
b^1_S=&\frac{2\xi_2\vh_3}{1+\vh\cdot\xi},\\
b^2_S=&-\frac{2\xi_1\vh_3}{1+\vh\cdot\xi},\\
b^3_S=&\frac{2(\xi_1\vh_2-\xi_2\vh_1)}{1+\vh\cdot\xi}.
\end{split}
\end{equation*}
In the above, we have again used the convention that repeated indices are summed over. Moreover, since $\xi$ is $2$-dimensional, we used the following dot product notation
$$1+\vh\cdot\xi \eqdef 1+\vh_1\xi_1+\vh_2\xi_2.$$
This technically corresponds to setting $\xi =(\xi_1, \xi_2, 0)$.

As in the $2$-dimensional case, we will use the following abbreviated notations
$$
|K| \eqdef |E|+|B|, \quad |K_S| \eqdef |E_S|+|B_S|, \quad |K_T| \eqdef |E_T|+|B_T|.
$$ 
We further divide $K_S$ into $K_{S,1}, K_{S,2}$ where
$$
|K_S|\le |K_{S,1}|+|K_{S,2}|,
$$ 
and $K_{S,2}$ contains only the good components $K_g$ from \eqref{good.comp.2hD} but $K_{S,1}$ is less singular. We now gather the estimates for each of these terms from \cite{GS2.5D}.

\begin{proposition}[Glassey-Schaeffer \cite{GS2.5D}]\label{GSbd.2D.2h}
$K_T$ satisfies the following bound:
\begin{equation}\label{KT.gs.est.2h}
|K_T(t,x)|
\ls \int_0^t \int_{|y-x|\leq t-s}\int_{\mathbb R^3} \frac{f(s,y,\pel) \langle\pel_3\rangle^3 d\pel\, dy\, ds}{\pel_0(1+\vh\cdot\xi)(t-s)\sqrt{(t-s)^2-|y-x|^2}},
\end{equation}
and $K_{S,1}$ satisfies the bound:
\begin{equation}\label{KS1.gs.est.2h}
|K_{S,1}(t,x)|
\ls \int_0^t \int_{|y-x|\leq t-s}\int_{\mathbb R^3} \frac{(|K| f)(s,y,\pel) \big(\mathcal{I}_{S,1}+\mathcal{II}_{S,1}\big)~ d\pel\, dy\, ds}{\sqrt{(t-s)^2-|y-x|^2}},
\end{equation}
where
$
\mathcal{I}_{S,1} \eqdef
\frac{1}{\pel_0}
$
and
$
\mathcal{II}_{S,1}\eqdef
\frac{\langle p_3 \rangle^2}{\pel_0^3(1+\vh\cdot\xi)}.
$
Finally, $K_{S,2}$ satisfies the bound:
\begin{equation}\label{KS2.gs.est.2h}
|K_{S,2}(t,x)|
\ls \int_0^t \int_{|y-x|\leq t-s}\int_{\mathbb R^3} \frac{(K_g f)(s,y,\pel)\langle\pel_3\rangle^2 ~ d\pel\, dy\, ds}{\pel_0(1+\vh\cdot\xi)\sqrt{(t-s)^2-|y-x|^2}}.
\end{equation}
\end{proposition}

In the following proof we will encounter cross products of two dimensional vectors, such as $\xi$ and $\omega$, with three dimensional vectors.  In these situations, if $b=(b_1, b_2, b_3)$ is a three dimensional vector then we set  $\xi=(\xi_1, \xi_2, 0)$ and $\omega=(\omega_1, \omega_2, 0)$ and we use the notation
$$
\xi \times b = (\xi_2 b_3 ,  - \xi_1 b_3,  \xi \wedge b),
\quad 
\xi \wedge b \eqdef \xi_1 b_2 - \xi_2 b_1,
$$
and similarly for $\omega$.  

\begin{proof}
The computations below follow exactly  those in \cite{GS2.5D}. The only difference is that in our setting, $\pel_3$ is not a priori bounded, and we need to track the exact dependence on $\pel_3$ in each of the terms. We will frequently use without mention the following bounds from Lemma 5.1 in \cite{GS2.5D}:
\begin{align*}
|\xi+\vh|\leq& \sqrt{2} (1+\vh\cdot\xi)^{\frac 12},\quad &|\vh\times\om|\leq \sqrt{2}(1+\vh\cdot\xi)^{\frac 12},\\
|\xi-\om|=&1-|\xi|\leq 1+\vh\cdot\xi,\quad &1+\vh\cdot\omega \leq  4 (1+\vh\cdot\xi),\\
\frac 1{\pel_0}\leq &\sqrt{2}(1+\vh\cdot\xi)^{\frac 12},\quad &|\xi_k\vh\cdot\xi-\vh_k|\leq \sqrt{8}(1+\vh\cdot\xi)^{\frac 12},
\end{align*}
for $k=1,2$.

We now turn to the estimates. We begin with the $T$ terms, which give rise to the desired bounds for $K_T$. In fact, the estimates for the $K_T$ terms can be read off directly after bounding $e^i_T$, $e^3_T$, $b^1_T$, $b^2_T$ and $b^3_T$. More precisely, we have the following bound for $e^i_T$ ($i=1,2$):
\begin{equation*}
\begin{split}
e^i_T=&\frac{-2(1-\vh_1^2-\vh_2^2)(\xi_i+\vh_i)}{(1+\vh\cdot\xi)^2}\ls \frac{1+\pel_3^2}{\pel_0^2(1+\vh\cdot\xi)^{\frac 32}}\ls \frac{\langle\pel_3\rangle^2}{\pel_0(1+\vh\cdot\xi)},
\end{split}
\end{equation*}
and the following estimate for $e^3_T$:
\begin{equation*}
\begin{split}
e^3_T=&\frac{2\vh_3(\vh_1(\xi_1+\vh_1)+\vh_2(\xi_2+\vh_2))}{(1+\vh\cdot\xi)^2}=\frac{2\vh_3(1+\vh\cdot\xi-(1-\vh_1^2-\vh_2^2))}{(1+\vh\cdot\xi)^2}\\
\ls &\frac{|\pel_3|}{\pel_0(1+\vh\cdot\xi)}+\frac{\langle\pel_3\rangle^3}{\pel_0^3(1+\vh\cdot\xi)^2}\ls \frac{\langle\pel_3\rangle^3}{\pel_0(1+\vh\cdot\xi)}.
\end{split}
\end{equation*}
We have the following estimate\footnote{The corresponding estimate for $b^2_T$ is similar after swapping $1$ and $2$ and changing appropriate signs.} for $b^1_T$:
\begin{equation*}
\begin{split}
b^1_T=&\frac{2\vh_3((1+\xi_1\vh_1)(\xi_2+\vh_2)-\xi_2\vh_1(\xi_1+\vh_1))}{(1+\vh\cdot\xi)^2}\\
=&\frac{2\vh_3(\xi_2(1-\vh_1^2-\vh_2^2)+\vh_2(1+\vh\cdot\xi))}{(1+\vh\cdot\xi)^2},\\
\ls &\frac{\langle\pel_3\rangle^3}{\pel_0^3(1+\vh\cdot\xi)^2}+\frac{|\pel_3|}{\pel_0(1+\vh\cdot\xi)}\ls \frac{\langle\pel_3\rangle^3}{\pel_0(1+\vh\cdot\xi)},\\
\end{split}
\end{equation*}
and the following bound for $b^3_T$:
\begin{equation*}
\begin{split}
b^3_T=&\frac{2((\vh_2+\xi_2(\vh_1^2+\vh_2^2))(\xi_1+\vh_1)-(\vh_1+\xi_1(\vh_1^2+\vh_2^2))(\xi_2+\vh_2))}{(1+\vh\cdot\xi)^2}\\
=&\frac{2((\xi_1\vh_2-\xi_2\vh_1)(1-\vh_1^2-\vh_2^2))}{(1+\vh\cdot\xi)^2}
\ls 
\frac{\langle\pel_3\rangle^2}{\pel_0(1+\vh\cdot\xi)}.
\end{split}
\end{equation*}
The bounds above easily imply the desired estimates for $K_T$.

We now turn to the estimates for the $K_S$ terms. For these terms, we directly take the algebraic manipulation of \cite{GS2.5D} and show that $\tilde K\cdot\nab_\pel e^i_S$, $\tilde K\cdot\nab_\pel e^3_S$, $\tilde K\cdot\nab_\pel b^1_S$, $\tilde K\cdot\nab_\pel b^2_S$ and $\tilde K\cdot\nab_\pel b^3_S$ satisfy the desired bounds. Here, we remind the readers that we have used the notation $\tilde{K}=E+\vh\times B$.

We now make a brief comment on how these terms will be controlled. Each of these terms will be decomposed as a linear combination of the components $\om\cdot E$, $\om\cdot B$, $E_3$, $B_3$, $\om\wedge E$ and $\om\wedge B$. Since $\om\cdot E$ and $\om\cdot B$ are good components, i.e. one of the $K_g$ components from \eqref{good.comp.2hD}, we will classify these terms as $K_{S,2}$ terms. It then suffices to show that their coefficients are bounded by $\frac{\langle\pel_3\rangle^2}{\pel_0(1+\vh\cdot\xi)}$. For the remaining terms, we combine $E_3$ and $\om\wedge B$ (resp. $B_3$ and $\om\wedge E$) and rewrite them as a term $E_3-\om\wedge B$ (resp. $B_3+\om\wedge E$) and another term that is bounded by $|E|$ and $|B|$. The $E_3-\om\wedge B$ term (resp. $B_3+\om\wedge E$ term) will be considered as a $K_{S,2}$ term and we bound the coefficient by $\frac{\langle\pel_3\rangle^2}{\pel_0(1+\vh\cdot\xi)}$. The remaining term will be considered as a $K_{S,1}$ term and we control the coefficient by $\frac{1}{\pel_0}+\frac{\langle\pel_3\rangle^2}{\pel_0^3(1+\vh\cdot\xi)}+\frac{|\pel_3|}{\pel_0^2(1+\vh\cdot\xi)^{\frac 12}}$. Notice that by Young's inequality, we have
$$\frac{|\pel_3|}{\pel_0^2(1+\vh\cdot\xi)^{\frac 12}}\ls \frac{1}{\pel_0}+\frac{\langle\pel_3\rangle^2}{\pel_0^3(1+\vh\cdot\xi)}$$
and thus this bound is acceptable for the $K_{S,1}$ terms according to the statement of the proposition.

We now proceed to control $\tilde K \cdot \nab_{\pel} e^1_S$. By equation (5.26) in \cite{GS2.5D}, we have
\begin{equation*}
\begin{split}
&-\frac{\pel_0(1+\vh\cdot\xi)^2}{2}(\tilde K \cdot \nab_{\pel} e^1_S)\\
=&[\om_1(1+\vh\cdot\xi)+(\om\cdot\vh)(\xi_1\vh\cdot\xi-\vh_1)-|\xi|(\xi_1+\vh_1)](\om\cdot E)\\
&-\vh_3\om_2(1+\vh\cdot\xi)(\om\cdot B)+\vh_3(\xi_1\vh\cdot\xi-\vh_1)E_3\\
&+[(\om\wedge\vh)(\om_1(1+\vh\cdot\xi)-|\xi|(\xi_1+\vh_1))+\om_2(\om\cdot\vh)(1+\vh\cdot\xi)]B_3\\
&+[-\om_2(1+\vh\cdot\xi)+(\om\wedge\vh)(\xi_1\vh\cdot\xi-\vh_1)](\om\wedge E)\\
&-\vh_3[\om_1(1+\vh\cdot\xi)-|\xi|(\xi_1+\vh_1)](\om\wedge B).
\end{split}
\end{equation*}
The coefficient of $\om\cdot E$ in $\tilde K \cdot \nab_{\pel} e^1_S$ satisfies
\begin{equation*}
\begin{split}
&|-\frac{2}{\pel_0(1+\vh\cdot\xi)^2}[\om_1(1+\vh\cdot\xi)+(\om\cdot\vh)(\xi_1\vh\cdot\xi-\vh_1)-|\xi|(\xi_1+\vh_1)]|\\
=&|-\frac{2}{\pel_0(1+\vh\cdot\xi)^2}[\om_1(1+\vh\cdot\xi)-\xi_1|\xi|(1-(\om\cdot\vh)^2)-\vh_1((|\xi|-1)+1+\om\cdot\vh)|\\
\ls &\frac{1}{\pel_0(1+\vh\cdot\xi)}.
\end{split}
\end{equation*}
The coefficient of $\om\cdot B$ in $\tilde K \cdot \nab_{\pel} e^1_S$ satisfies
\begin{equation*}
\begin{split}
|\frac{2\vh_3\om_2(1+\vh\cdot\xi)}{\pel_0(1+\vh\cdot\xi)^2}|
\ls &\frac{1}{\pel_0(1+\vh\cdot\xi)}.
\end{split}
\end{equation*}
Combining the terms $E_3$ and $\om\wedge B$ in $\tilde K \cdot \nab_{\pel} e^1_S$, we get the bound
\begin{equation*}
\begin{split}
&|\frac{2}{\pel_0(1+\vh\cdot\xi)^2}\vh_3(\xi_1(1+\vh\cdot\xi)-(\xi_1+\vh_1))(E_3-\om\wedge B)|\\
&+|\frac{2\vh_3}{\pel_0(1+\vh\cdot\xi)^2}((\xi_1-\om_1)(1+\vh\cdot\xi)+(\xi_1+\vh_1)(|\xi|-1))(\om\wedge B)|\\
\ls &\frac{|\pel_3|}{\pel_0(1+\vh\cdot\xi)}|E_3-\om\wedge B|+\frac{|\pel_3|}{\pel_0^2(1+\vh\cdot\xi)^{\frac 12}}|B|.
\end{split}
\end{equation*}
Combining the terms $B_3$ and $\om\wedge E$ in $\tilde K \cdot \nab_{\pel} e^1_S$, 
we get the estimate
\begin{equation*}
\begin{split}
&|\frac{2}{\pel_0(1+\vh\cdot\xi)^2}[-\om_2(1+\vh\cdot\xi)+(\om\wedge\vh)(\xi_1\vh\cdot\xi-\vh_1)](\om\wedge E+ B_3)|\\
&+|\frac{2}{\pel_0(1+\vh\cdot\xi)^2}[(\om\wedge\vh)(1-|\xi|)(\om_1(1+\vh\cdot\xi)+(\vh_1+\xi_1))] B_3|\\
&+|\frac{2}{\pel_0(1+\vh\cdot\xi)^2}(\om_2(1+\vh\cdot\xi)(1+\vh\cdot\om)) B_3|\\
\ls 
&\frac{1}{\pel_0(1+\vh\cdot\xi)}|\om\wedge E+ B_3|+\frac{1}{\pel_0}|B|.
\end{split}
\end{equation*}
Thus all of the terms obey the required bounds. The term $\tilde K\cdot\nab_\pel e_S^2$ obviously satisfies the same estimates after repeating the above computations with some subscripts $1$ replaced by $2$.

We now move to the bounds for the term $\tilde K\cdot\nab_\pel b_S^1$. By equation (5.18) in \cite{GS2.5D}, we have
\begin{equation*}
\begin{split}
&\frac{\pel_0(1+\vh\cdot\xi)^2}{2\xi_2}(\tilde K\cdot \nab_\pel b^1_S)\\
=&-\vh_3(\omega\cdot(\xi+\vh))(\omega\cdot E)-(\om\wedge\vh)(1+\vh\cdot\xi)(\omega\cdot B)\\
&-\vh_3|\xi|(\om\wedge\vh) B_3+(1+\vh\cdot\xi-\vh_3^2)E_3\\
&-\vh_3(\om\wedge\vh)(\om\wedge E)+[\vh_3^2(\om\cdot(\xi+\vh))+(\om\cdot\vh)(1+\vh\cdot\xi-\vh_3^2)](\om\wedge B).
\end{split}
\end{equation*}
The coefficient for the $\om\cdot E$ term obeys the bound
$$|-\frac{2\xi_2\vh_3(\om\cdot(\xi+\vh))}{\pel_0(1+\vh\cdot\xi)^2}|\ls \frac{|\pel_3|}{\pel_0(1+\vh\cdot\xi)}.$$
The coefficient for the $\om\cdot B$ term satisfies the bound
$$|-\frac{2\xi_2(\om\wedge\vh)(1+\vh\cdot\xi)}{\pel_0(1+\vh\cdot\xi)^2}|\ls \frac{1}{\pel_0(1+\vh\cdot\xi)^{1/2}}.$$
The terms with $B_3$ and $\om\wedge E$ can be combined to yield
\begin{equation*}
\begin{split}
&|-\frac{2\xi_2\vh_3(\om\wedge\vh)}{\pel_0(1+\vh\cdot\xi)^2}[(|\xi|-1)B_3+(B_3+\om\wedge E)]|\\
\ls & \frac{|\pel_3|}{\pel_0^2(1+\vh\cdot\xi)^{\frac 12}}|B_3|+\frac{|\pel_3|}{\pel_0(1+\vh\cdot\xi)}|B_3+\om\wedge E|.
\end{split}
\end{equation*}
The terms with $E_3$ and $\om\wedge B$ can be combined and bounded by
\begin{equation*}
\begin{split}
&|\frac{2\xi_2}{\pel_0(1+\vh\cdot\xi)}[(E_3-\om\wedge B)+(1+\om\cdot\vh)(\om\wedge B)\\
&+|\frac{2\xi_2\vh_3^2}{\pel_0(1+\vh\cdot\xi)^2}[(E_3(1-|\xi|)+|\xi|(E_3-\om\wedge B)]|\\
\ls & \frac{\pel_3^2}{\pel_0(1+\vh\cdot\xi)}|E_3-\om\wedge B|+\frac{1}{\pel_0}|B|+\frac{\pel_3^2}{\pel_0^3(1+\vh\cdot\xi)}|E|.
\end{split}
\end{equation*}
Above in particular we used that $|\xi| = \om \cdot \xi$.  Notice that some of the bounds above are better than what is needed.
This concludes the bounds for $\tilde K\cdot\nab_\pel b_S^1$. Clearly, the desired estimates for $\tilde K\cdot\nab_\pel b_S^2$ can be obtained in an identical manner since $b^1_S$ and $b^2_S$ differ only by a factor of $-\frac {\xi_2}{\xi_1}$, which is independent of $\pel$. Similarly, the desired estimates for $\tilde K\cdot\nab_\pel e_S^3$ can also be derived in an identical manner since $b^1_S$ and $e^3_S$  differ only be a factor of $-\xi_2$, which is also independent of $\pel$.

Finally, we move to the bounds for the term $\tilde K\cdot\nab_\pel b_S^3$. By equation (5.23) in \cite{GS2.5D}, we have
\begin{equation*}
\begin{split}
&\frac{\pel_0(1+\vh\cdot\xi)^2}{2|\xi|}(\tilde K\cdot\nab_\pel b_S^3)\\
=&-(\om\wedge\vh)[(|\xi|-1)+(1+\om\cdot\vh)](\om\cdot E)+\vh_3(1+\vh\cdot\xi)(\om\cdot B)\\
&-[(\om\wedge\vh)^2(|\xi|+\om\cdot\vh)+(\om\cdot\vh)(1+\vh\cdot\xi-(\om\wedge\vh)^2)]B_3\\
&-\vh_3(\om\wedge\vh)E_3+(1+\vh\cdot\xi-(\om\wedge\vh)^2)(\om\wedge E)+|\xi|\vh_3(\om\wedge\vh)(\om\wedge B).
\end{split}
\end{equation*}
The coefficient of $\om\cdot E$ is bounded by 
\begin{equation*}
\begin{split}
|\frac{2|\xi|}{\pel_0(1+\vh\cdot\xi)^2}(\om\wedge\vh)[(|\xi|-1)+(1+\om\cdot\vh)]|
\ls \frac {1}{\pel_0(1+\vh\cdot\xi)^{\frac 12}}.
\end{split}
\end{equation*}
That of $\om\cdot B$ is bounded by
\begin{equation*}
\begin{split}
&|\frac{2|\xi|}{\pel_0(1+\vh\cdot\xi)^2}\vh_3(1+\vh\cdot\xi)|\ls \frac {1}{\pel_0(1+\vh\cdot\xi)}.
\end{split}
\end{equation*}
We combine the terms with $B_3$ and $\om\wedge E$ to get the bound
\begin{equation*}
\begin{split}
&|\frac{2|\xi|}{\pel_0(1+\vh\cdot\xi)^2}[(\om\wedge\vh)^2(|\xi|-1)-(1+\om\cdot\vh)(1+\vh\cdot\xi)]B_3|\\
+&|\frac{2|\xi|}{\pel_0(1+\vh\cdot\xi)^2}(1+\vh\cdot\xi-(\om\wedge\vh)^2)(B_3+\om\wedge E)|\\
\ls &\frac{1}{\pel_0}|B|+\frac{1}{\pel_0(1+\vh\cdot\xi)}|B_3+\om\wedge E|.
\end{split}
\end{equation*}
Finally, we combine the terms with $E_3$ and $\om\wedge B$ to obtain the estimate
\begin{equation*}
\begin{split}
&|\frac{2\vh_3(\om\wedge\vh)(1-|\xi|)|\xi|}{\pel_0(1+\vh\cdot\xi)^2}E_3|+|\frac{2\vh_3(\om\wedge\vh)|\xi|^2}{\pel_0(1+\vh\cdot\xi)^2}(E_3-\om\wedge B)|\\
\ls &\frac{|\pel_3|}{\pel_0^2(1+\vh\cdot\xi)^{\frac 12}}|E|+\frac{|\pel_3|}{\pel_0(1+\vh\cdot\xi)}|E_3-\om\wedge B|.
\end{split}
\end{equation*}
This concludes the proof of the proposition.
\end{proof}

We will further derive estimates for each of the terms in this decomposition so that we can apply Strichartz estimates. First, we need to following analogue of Lemma \ref{sing.lemma}, which allows for weights in $\langle\pel_3\rangle$:

\begin{lemma}\label{sing.lemma.2.5D}
In the region $|y-x|\leq (t-s)$, we have the estimates
\begin{equation}\label{sing.2h.1}
\int_{\mathbb R^3} \frac{f(s,y,\pel)\langle\pel_3\rangle^3}{\pel_0(1+\vh\cdot\xi)}\,d\pel\ls 
\frac{(\int_{\mathbb R^3} \pel_0^2 f(s,y,\pel)\,d\pel)^{\frac 25}}{(1-|\xi|^2)^{\frac 25}}
\end{equation}
and 
\begin{equation}\label{sing.2h.2}
\int_{\mathbb R^3} \frac{f(s,y,\pel)\langle\pel_3\rangle^3}{\pel_0(1+\vh\cdot\xi)}\,d\pel\ls 
(\int_{\mathbb R^3} \pel_0^4 f(s,y,\pel)\,d\pel)^{\frac 25}.
\end{equation}
\end{lemma}
\begin{proof}
The proof is similar to that of Lemma \ref{sing.lemma}, except that we need to use the weighted $\pel_3$ integral bound of $f$ in \eqref{mom.2.5.bound} instead of simply using the $L^\infty$ estimate for $f$. We first prove \eqref{sing.2h.1}. As in the proof of Lemma \ref{sing.lemma}, we will control the integral separately large and small momenta. However, in the case that we are considering, we need to do this do this for both the $r=\sqrt{\pel_1^2+\pel_2^2}$ variable and the $\pel_3$ variable. To proceed, we define the angle $\theta\in (-\pi,\pi]$ by
\begin{equation}\label{angle.imp}
-\vh\cdot\xi=\frac{r|\xi|}{\pel_0}  \cos\theta,
\end{equation}
where again
$r\eqdef \sqrt{p_1^2+p_2^2}.$
The following analogue of \eqref{sing.2D.est} holds:
\begin{equation}\label{sing.2D.2h.est}
(1+\vh\cdot\xi)^{-1}\ls \min\left\{\frac{1}{\th^{2}},\, \frac{1}{1-|\xi|^2},\, \langle r \rangle^2\right\}.
\end{equation}
Moreover, for $f_0\eqdef \mathbbm 1_{\{\langle \pel_3 \rangle<2\}} f$,  $f_n\eqdef\mathbbm 1_{\{2^n\leq \langle \pel_3 \rangle <2^{n+1}\}} f$ for $n\ge 1$ we write
$$f=\sum_{n=0}^{\infty} f_n,$$
where $\mathbbm 1_D$ denotes the characteristic function of set $D$.  
Furthermore $R = R(n) \ge 1$ be a constant to be chosen later. For $|r|\leq R$, we write the $d\pel$ integral in polar coordinates in the $p_1$-$p_2$ plane. 
This grants us the following estimate for each $f_n$ (including when $n=0$) in the region $|r|\leq R$:
\begin{multline}\label{sing.lemma.2h.1}
\int_{|r|\leq R} \frac{\langle\pel_3\rangle^3 f_n(s,y,\pel)d\pel}{\pel_0(1+\vh\cdot\xi)}
\\
\ls 
\left(\int_0^R \int_{-\pi}^{\pi} \frac{1}{\langle r \rangle(1+\vh\cdot\xi)} r\,d\theta\,dr\right)\left(\sup_{\pel_1,\pel_2}\int_{\mathbb R} \langle\pel_3\rangle^3 f_n(s,y,\pel)\,d\pel_3\right)
\\
\ls 
\left(\int_0^R \int_{[-\pi,\pi]\setminus [-R^{-1},R^{-1}]} \frac{1}{\th^2} \,d\theta\,dr+\int_0^R \int_{-R^{-1}}^{R^{-1}} \langle r \rangle^2 \,d\theta\,dr\right)
\\
\times \left(\sup_{\pel_1,\pel_2}\int_{\mathbb R} \langle\pel_3\rangle^3 f_n(s,y,\pel)\,d\pel_3\right)
\\
\ls 
\langle R\rangle^2\left(\sup_{\pel_1,\pel_2}\int_{\mathbb R} \langle\pel_3\rangle^3 f_n(s,y,\pel)\,d\pel_3\right).
\end{multline}
For $|r|\geq R$, we use $\frac{1}{1+\vh\cdot\xi}\ls \frac 1{1-|\xi|^2}$ to get
\begin{equation}\label{sing.lemma.2h.2}
\int_{|r|\geq R} \frac{\langle\pel_3\rangle^3 f_n(s,y,\pel)d\pel}{\pel_0(1+\vh\cdot\xi)}\ls 
\frac{2^{3n}}{\langle R \rangle^3(1-|\xi|^2)}\int_{|r|\geq R} \pel_0^2 f(s,y,\pel)d\pel.
\end{equation}
Taking\footnote{Of course we we only apply this estimate under the assumption that $\sup_{\pel_1,\pel_2}\int_{\mathbb R} \langle\pel_3\rangle^3 f_n(s,y,\pel)\,d\pel_3\neq 0$. If it is equal to zero, then $f_n=0$ and we have the trivial estimate
$$\int_{\mathbb R^3} \frac{\langle\pel_3\rangle^3 f_n(s,y,\pel)\,d\pel}{\pel_0(1+\vh\cdot\xi)}=0.$$
 } $R=\frac{2^{\frac{3n}{5}}(\int_{\mathbb R^2} f\pel_0^2\,d\pel)^{\frac 15}}{(1-|\xi|^2)^{\frac 15}(\sup_{\pel_1,\pel_2}\int_{\mathbb R} \langle\pel_3\rangle^3 f_n(s,y,\pel)\,d\pel_3)^{\frac 15}}$, when $R\ge 1$ \eqref{sing.lemma.2h.1} and \eqref{sing.lemma.2h.2} imply the estimate
\begin{equation}\label{sing.lemma.2h.3}
\begin{split}
&\int_{\mathbb R^3} \frac{\langle\pel_3\rangle^3 f_n(s,y,\pel)\,d\pel}{\pel_0(1+\vh\cdot\xi)}\\
\ls &\frac{2^{\frac{6n}{5}}(\sup_{\pel_1,\pel_2}\int_{\mathbb R} \langle\pel_3\rangle^3 f_n(s,y,\pel)\,d\pel_3)^{\frac 35}(\int_{\mathbb R^3} \pel_0^{2} f(s,y,\pel)\,d\pel)^{\frac 25}}{(1-|\xi|^2)^{\frac 25}}\\
\ls &\frac{2^{-\frac{3\delta n}{5}}(\sup_{\pel_1,\pel_2}\int_{\mathbb R} \langle\pel_3\rangle^{5+\delta} f_n(s,y,\pel)\,d\pel_3)^{\frac 35}(\int_{\mathbb R^3} \pel_0^2 f(s,y,\pel)\,d\pel)^{\frac 25}}{(1-|\xi|^2)^{\frac 25}}\\
\ls &\frac{2^{-\frac{3\delta n}{5}}(\int_{\mathbb R^3} \pel_0^2 f(s,y,\pel)\,d\pel)^{\frac 25}}{(1-|\xi|^2)^{\frac 25}},
\end{split}
\end{equation}
since $\|\int_{\mathbb R} \langle\pel_3\rangle^{5+\delta} f_n(s,y,\pel)\,d\pel_3\|_{L^\infty_t([0,T);L^\infty_xL^\infty_{(\pel_1,\pel_2)})}\ls 1$ by \eqref{mom.2.5.bound}. 
Suppose for this $R$ that $R\ge 1$ does not hold, and instead for some $n$ we have
$$
\frac{(\int_{\mathbb R^2} f\pel_0^2\,d\pel)^{\frac 15}}{(1-|\xi|^2)^{\frac 15}}
\le 
2^{-\frac{3n}{5}}(\sup_{\pel_1,\pel_2}\int_{\mathbb R} \langle\pel_3\rangle^3 f_n(s,y,\pel)\,d\pel_3)^{\frac 15}
\ls 
2^{-\frac{(5+\delta)n}{5}}
$$
Taking this estimate to the third power, and using \eqref{sing.lemma.2h.2} with instead $R=0$ we similarly obtain the upper bound from \eqref{sing.lemma.2h.3}.  Summing \eqref{sing.lemma.2h.3} in $n$, we obtain \eqref{sing.2h.1}.  

We now turn to the second estimate \eqref{sing.2h.2}. As before, we will decompose $f$ into $f_n$'s and consider large and small $|r|$ separately. For $|r|\leq R$, we will use the estimate \eqref{sing.lemma.2h.1}.
For $|r|\geq R$, we use $\frac{1}{1+\vh\cdot\xi}\ls \langle r \rangle^2$ to get
\begin{equation}\label{sing.lemma.2h.5}
\begin{split}
\int_{|r|\geq R} \frac{\langle\pel_3\rangle^3 f_n(s,y,\pel)d\pel}{\pel_0(1+\vh\cdot\xi)}\ls &\frac{2^{3n}}{\langle R \rangle^3}\int_{|r|\geq R} \pel_0^4 f_n(s,y,\pel)d\pel.
\end{split}
\end{equation}
Let $R=\frac{2^{\frac{3n}{5}}(\int_{\mathbb R^2} f\pel_0^4\,d\pel)^{\frac 15}}{(\sup_{\pel_1,\pel_2}\int_{\mathbb R} \langle\pel_3\rangle^3 f_n(s,y,\pel)\,d\pel_3)^{\frac 15}}$. Then, when $R\ge 1$, \eqref{sing.lemma.2h.1} and \eqref{sing.lemma.2h.5} imply
\begin{equation}\label{sing.lemma.2h.6}
\begin{split}
&\int_{\mathbb R^3} \frac{\langle\pel_3\rangle^3 f_n(s,y,\pel)\,d\pel}{\pel_0(1+\vh\cdot\xi)}\\
\ls &2^{\frac{6n}{5}}(\sup_{\pel_1,\pel_2}\int_{\mathbb R} \langle\pel_3\rangle^3 f_n(s,y,\pel)\,d\pel_3)^{\frac 35}(\int_{\mathbb R^3} \pel_0^4 f(s,y,\pel)\,d\pel)^{\frac 25}\\
\ls &2^{-\frac{3\delta n}{5}}(\sup_{\pel_1,\pel_2}\int_{\mathbb R} \langle\pel_3\rangle^{5+\delta} f_n(s,y,\pel)\,d\pel_3)^{\frac 35}(\int_{\mathbb R^3} \pel_0^4 f(s,y,\pel)\,d\pel)^{\frac 25}\\
\ls &2^{-\frac{3\delta n}{5}}(\int_{\mathbb R^3} \pel_0^4 f(s,y,\pel)\,d\pel)^{\frac 25},
\end{split}
\end{equation}
using \eqref{mom.2.5.bound}. A similar modification to the previous case can be used to also obtain this estimate when our chosen $R$ satisfies $R<1$.  Summing \eqref{sing.lemma.2h.6} over $n$, we obtain the second estimate \eqref{sing.2h.2}. This concludes the proof of the lemma.  
\end{proof}

Using the above lemma, we derive bounds for $K_T$ and $K_{S,2}$. As noted in Remark \ref{KT.2D.rmk}, the proof of Proposition \ref{KT.2D} can be applied in more general situations. In particular, it can be used to obtain estimates for $K_T$ in the $2\frac 12$ dimensional case:
\begin{proposition}\label{KT.2D.2h}
For every $\ep\in (0,1]$, $K_T$ obeys the following estimate:
\begin{multline*}
|K_T(t,x)|
\ls \ep^{-\frac 1{10}}\left(\int_0^t \int_{|y-x|\leq t-s} \frac{\int_{\mathbb R^3} f(s,y,\pel)\pel_0^2 d\pel}{\sqrt{(t-s)^2-|y-x|^2}}  dy ds\right)^{\frac 25}\\
+\ep^{\frac 3{10}}\left(\int_0^t \int_{|y-x|\leq t-s} \frac{\int_{\mathbb R^3} f(s,y,\pel)\pel_0^4 d\pel}{\sqrt{(t-s)^2-|y-x|^2}}  dy ds\right)^{\frac{2}{5}}.
\end{multline*}
\end{proposition}
\begin{proof}
We use Remark \ref{KT.2D.rmk} with
\begin{equation*}
\begin{split}
F(t,s,x,y)=&\int_{\mathbb R^3} \frac{f(s,y,\pel)\langle\pel_3\rangle^3}{\pel_0(1+\vh\cdot\xi)}\,d\pel,\\
G(s,y)=&\int_{\mathbb R^3} f(s,y,\pel)\pel_0^2 d\pel,\\
H(s,y)=&\int_{\mathbb R^3} f(s,y,\pel)\pel_0^4 d\pel.
\end{split}
\end{equation*}
By \eqref{KT.gs.est.2h}, we have
$$|K_T(t,x)|\ls \int_0^t\int_{|y-x|\leq (t-s)}\frac{F(t,s,x,y)}{(t-s)\sqrt{(t-s)^2-|y-x|^2}}\,dy\,ds.$$
On the other hand, Remark \ref{KT.2D.rmk} shows that
\begin{equation*}
\begin{split}
&\int_0^t \int_{|y-x|\leq (t-s)} \frac{F(t,s,x,y)\, dy\,ds}{(t-s)\sqrt{(t-s)^2-|y-x|^2}}\\
\ls &\ep^{-\frac 1{10}}\left(\int_0^t \int_{|y-x|\leq t-s} \frac{G(s,y)}{\sqrt{(t-s)^2-|y-x|^2}}  dy ds\right)^{\frac 25}\\
&+\ep^{\frac 3{10}}\left(\int_0^t \int_{|y-x|\leq t-s} \frac{H(s,y)}{\sqrt{(t-s)^2-|y-x|^2}}  dy ds\right)^{\frac{2}{5}}
\end{split}
\end{equation*}
for every $\ep\in (0,1]$. The result follows.
\end{proof}
Similarly, we apply Lemma \ref{sing.lemma.2.5D} to the bound \eqref{KS2.gs.est.2h} to get
\begin{proposition}\label{KS2.2D.2h}  
For every $\ep\in (0,1]$, $K_{S,2}$ obeys the following estimate:
\begin{multline*}
|K_{S,2}(t,x)|
\ls \ep^{-\frac 1{10}}\left(\int_0^t \int_{|y-x|\leq t-s} \frac{\int_{\mathbb R^3} f(s,y,\pel)\pel_0^2 d\pel}{\sqrt{(t-s)^2-|y-x|^2}}  dy ds\right)^{\frac 25}\\
+\ep^{\frac 3{10}}\left(\int_0^t \int_{|y-x|\leq t-s} \frac{\int_{\mathbb R^3} f(s,y,\pel)\pel_0^4 d\pel}{\sqrt{(t-s)^2-|y-x|^2}}  dy ds\right)^{\frac{2}{5}}.
\end{multline*}
\end{proposition}
\begin{proof}
We use Remark \ref{KS2.2D.rmk} with
\begin{equation*}
\begin{split}
F(t,s,x,y)=&\int_{\mathbb R^3} \frac{f(s,y,\pel)\langle\pel_3\rangle^3}{\pel_0(1+\vh\cdot\xi)}\,d\pel,\\
G(s,y)=&\int_{\mathbb R^3} f(s,y,\pel)\pel_0^2 d\pel,\\
H(s,y)=&\int_{\mathbb R^3} f(s,y,\pel)\pel_0^4 d\pel.
\end{split}
\end{equation*}
By \eqref{KS2.gs.est.2h}, we have
$$|K_{S,2}(t,x)|\ls \int_0^t\int_{|y-x|\leq (t-s)}\frac{|K_g(s,y)| F(t,s,x,y)}{\sqrt{(t-s)^2-|y-x|^2}}\,dy\,ds.$$
On the other hand, Remark \ref{KT.2D.rmk} shows that
\begin{equation*}
\begin{split}
&\int_0^t \int_{|y-x|\leq (t-s)} \frac{|K_g(s,y)|F(t,s,x,y)\, dy\,ds}{\sqrt{(t-s)^2-|y-x|^2}}\\
\ls &\ep^{-\frac 1{10}}\left(\int_0^t \int_{|y-x|\leq t-s} \frac{G(s,y)}{\sqrt{(t-s)^2-|y-x|^2}}  dy ds\right)^{\frac 25}\\
&+\ep^{\frac 3{10}}\left(\int_0^t \int_{|y-x|\leq t-s} \frac{H(s,y)}{\sqrt{(t-s)^2-|y-x|^2}}  dy ds\right)^{\frac{2}{5}}
\end{split}
\end{equation*}
for every $\ep\in (0,1]$. The result follows.
\end{proof}

We now control $K_{S,1}$ in the following proposition:

\begin{proposition}\label{KS1.2D.2h}
For any small $\ep'>0$, the following estimate holds:
\begin{multline*}
|K_{S,1}(t,x)|
\ls  \int_0^t \int_{|y-x|\leq t-s}\int_{\mathbb R^3} \frac{(|K|f)(s,y,\pel)}{\pel_0\sqrt{(t-s)^2-|y-x|^2}} d\pel\, dy\, ds\\
+\int_0^t \int_{|y-x|\leq t-s} \frac{K(s,y)}{\sqrt{(t-s)^2-|y-x|^2}}\left(\int_{\mathbb R^3} \frac{ f(s,y,\pel)}{\pel_0^{1-3\ep'}}\,d\pel\right)^{\frac 12}\, dy\, ds.
\end{multline*}
\end{proposition}

\begin{proof}
It is shown in \eqref{KS1.gs.est.2h} that we have $|K_{S,1}| \ls K_{S,1}^{\mathcal{I}} + K_{S,1}^{\mathcal{II}}$ where
\begin{equation*}
K_{S,1}^{\mathcal{I}} \eqdef
\int_0^t \int_{|y-x|\leq t-s}\int_{\mathbb R^3} \frac{(|K| f)(s,y,\pel)}{\pel_0\sqrt{(t-s)^2-|y-x|^2}}~ d\pel\, dy\, ds
\end{equation*}
is already in an acceptable form for the estimate in Proposition \ref{KS1.2D.2h}.   Further
\begin{equation}\label{KS1.II}
K_{S,1}^{\mathcal{II}} \eqdef
\int_0^t \int_{|y-x|\leq t-s}\int_{\mathbb R^3} \frac{(|K| f)(s,y,\pel)}{\sqrt{(t-s)^2-|y-x|^2}}\frac{\langle \pel_3 \rangle^2}{\pel_0^3(1+\vh\cdot\xi)}~ d\pel\, dy\, ds.
\end{equation}
Applying the Cauchy-Schwarz inequality to the $dp$ integration in \eqref{KS1.II}, we have
\begin{multline}\label{KS1.2D.2h.1}
K_{S,1}^{\mathcal{II}}
\ls \int_0^t \int_{|y-x|\leq t-s} \frac{|K|(s,y) \, dy\, ds}{\sqrt{(t-s)^2-|y-x|^2}}
\\
\times(\int_{\mathbb R^3}\frac{ f(s,y,\pel)\langle \pel_3 \rangle^4}{\pel_0^{2+\ep'}(1+\vh\cdot\xi)^{\frac 12-\ep'}}\, d\pel)^{\frac 12}\left(\int_{\mathbb R^3} \frac{ f(s,y,\pel)}{\pel_0^{4-\ep'}(1+\vh\cdot\xi)^{\frac 32+\ep'}}\,d\pel \right)^{\frac 12}.
\end{multline}
Notice that  
\begin{multline}\label{KS1.2D.2h.2}
\int_{\mathbb R^3}\frac{ f(s,y,\pel)\langle \pel_3 \rangle^4}{\pel_0^{2+\ep'}(1+\vh\cdot\xi)^{\frac 12-\ep'}}\, d\pel
\\
\ls  \left\|\int_{-\infty}^{\infty} \langle \pel_3 \rangle^4 f(s,y,\pel) d\pel_3\right\|_{L^\infty_{(t,x,\pel_1,\pel_2)}}
\sup_{\pel_3}\iint \frac{d\pel_1 d\pel_2}{\langle (\pel_1,\pel_2)\rangle^{2+\ep'}(1+\vh\cdot\xi)^{\frac 12-\ep'}}
\\
\ls 
\sup_{\pel_3}\iint \frac{d\pel_1 d\pel_2}{\langle (\pel_1,\pel_2)\rangle^{2+\ep'}(1+\vh\cdot\xi)^{\frac 12-\ep'}},
\end{multline}
where we have used Proposition \ref{add.cons.law.2}. We claim that this last integral is finite. To see this, we first define (as in the previous proposition) $\theta\in (-\pi,\pi]$ and $r>0$ as in \eqref{angle.imp}.  Recall that by \eqref{sing.2D.2h.est}, we have
$
(1+\vh\cdot\xi)^{-1}\ls \frac{1}{\th^{2}}.
$
Therefore, 
\begin{equation}\label{KS1.2D.2h.3}
\iint \frac{d\pel_1 d\pel_2}{\langle (\pel_1,\pel_2)\rangle^{2+\ep'}(1+\vh\cdot\xi)^{\frac 12-\ep'}}
\ls \int_{-\pi}^{\pi} \int_0^{\infty} \frac{r}{(1+r^2)^{\frac{2+\ep'}{2}} }\frac{1}{\th^{1-2\ep'}} dr d\theta
\ls 1.
\end{equation}
Putting together equations \eqref{KS1.2D.2h.1}, \eqref{KS1.2D.2h.2} and \eqref{KS1.2D.2h.3}, we have
\begin{equation*}
\begin{split}
K_{S,1}^{\mathcal{II}}
\ls &\int_0^t \int_{|y-x|\leq t-s} \frac{|K|(s,y)}{\sqrt{(t-s)^2-|y-x|^2}}(\int_{\mathbb R^3} \frac{ f(s,y,\pel)}{\pel_0^{4-\ep'}(1+\vh\cdot\xi)^{\frac 32+\ep'}}\,d\pel)^{\frac 12}\, dy\, ds\\
\ls &\int_0^t \int_{|y-x|\leq t-s} \frac{|K|(s,y)}{\sqrt{(t-s)^2-|y-x|^2}}(\int_{\mathbb R^3} \frac{ f(s,y,\pel) 
\langle (\pel_1,\pel_2)\rangle^{3+2\ep'}}{\pel_0^{4-\ep'}}\,d\pel)^{\frac 12}\, dy\, ds,
\end{split}
\end{equation*}
where we also used \eqref{sing.2D.2h.est}.  This is better than the desired estimate.
\end{proof}

\subsection{Conclusion of the proof}
We are now ready to conclude the proof of Theorem \ref{main.theorem.2D.2h}. First, notice that except for one term in $K_{S,1}$, i.e. 
\begin{equation}\label{tKS1.def}
\int_0^t \int_{|y-x|\leq t-s} \frac{|K|(s,y)}{\sqrt{(t-s)^2-|y-x|^2}}\left(\int_{\mathbb R^3} \frac{ f(s,y,\pel)}{\pel_0^{1-3\ep'}}\,d\pel \right)^{\frac 12}\, dy\, ds
\eqdef
\tilde K_{S,1},
\end{equation}
all the bounds we have derived for $K_T$, $K_{S,1}$ and $K_{S,2}$ in Propositions \ref{KT.2D.2h}, \ref{KS2.2D.2h} and \ref{KS1.2D.2h} are very similar to those in the $2$-dimensional case in Propositions \ref{KT.2D} and \ref{KS2.2D}. The only differences for the terms in Propositions \ref{KT.2D.2h}, \ref{KS2.2D.2h} and \ref{KS1.2D.2h} as compared to those in Propositions \ref{KT.2D} and \ref{KS2.2D} are that firstly, the integrals are in $\pel$ are over $\mathbb R^3$ instead of $\mathbb R^2$; and secondly, the weights in $\pel_0$ depend on $\pel_1$, $\pel_2$ and $\pel_3$ rather than only $\pel_1$ and $\pel_2$.

We first control the term $\tilde K_{S,1}$ in \eqref{tKS1.def}. We will show that its $L^1_t([0,T);L^{N+2}_x)$ norm can be bounded a priori using the conservation laws:
\begin{proposition}\label{tKS1.est}
For $N>13$ as in the assumption \eqref{ini.bd.2D.2h.2} in Theorem \ref{main.theorem.2D.2h},  we have the estimate
\begin{equation*}
\|\tilde K_{S,1}\|_{L^1_t([0,T);L^{N+2}_x)}^{N+2}
\ls 1,
\end{equation*}
where the implicit constant depends at most polynomially on $T$.
\end{proposition}

\begin{proof}
We use the Strichartz estimates in Theorem \ref{Strichartz} with the exponents from \eqref{good.exp2}.  This yields
\begin{multline*}
\|\tilde K_{S,1} \|_{L^{\frac{3k(N+2)}{(k-3)N+(2k-24)}}_t([0,T);L^{N+2}_x)}^{N+2}\\
\ls 
 \left\|K \left(\int_{\mathbb R^3} \frac{ f(s,y,\pel)}{\pel_0^{1-3\ep'}}\,d\pel \right)^{\frac 12} \right\|_{L^{\frac{k(N+2)}{(k-1)N+(2k-8)}}_t([0,T);L_x^{\frac{3(N+2)}{2N+7}})}^{N+2}
\end{multline*}
for $k>6$.
Now we use H\"older's inequality for the upper bound,  similar to \eqref{KS2.2.2}, to get
\begin{multline*}
 \left\|K \left(\int_{\mathbb R^3} \frac{ f(s,y,\pel)}{\pel_0^{1-3\ep'}}\,d\pel \right)^{\frac 12} \right\|_{L^{\frac{k(N+2)}{(k-1)N+(2k-8)}}_t([0,T);L_x^{\frac{3(N+2)}{2N+7}})}^{N+2}
 \\
\ls 
\| K\|_{L^\infty_t([0,T); L^2_x)}^{N+2}
\left\|\left(\int_{\mathbb R^3} \frac{ f(s,y,\pel)}{\pel_0^{1-3\ep'}}\,d\pel \right)^{\frac 12} \right\|_{L^\infty_t([0,T);L_x^{\frac{6(N+2)}{N+8}})}^{N+2}
\\
\ls 
\left\|\frac{ f(s,y,\pel)}{\pel_0^{1-3\ep'}} \right\|_{L^\infty_t([0,T);L_x^{\frac{3(N+2)}{N+8}}L^1_\pel)}^{\frac{N+2}{2}},
\end{multline*}
where to get the last inequality we used the conservation law in Proposition \ref{cons.law.1} to observe that $\| K\|_{L^\infty_t([0,T); L^2_x)} \ls 1$. Further applying the interpolation inequality (Proposition \ref{imp.moment}) to the upper bound with 
$$S=-1+3\ep',
\quad \& \quad 
M=\frac{3(N+2)(1+3\ep')}{N+8}-2=\frac{N-10+9\ep'(N+2)}{N+8}<1,
$$ 
where $M<1$ holds for $\ep'>0$ sufficiently small.  We conclude that
$$
\left\|f \pel_0^S \right\|_{L^\infty_t([0,T);L_x^{\frac{3(N+2)}{N+8}}L^1_\pel)}^{\frac{N+2}{2}}
\ls 
\|f \pel_0^{M}\|_{L^\infty_t([0,T); L^1_x L^1_\pel)}^{\frac{N+8}{6}}.
$$
Since $M<1$, the last term is controlled by the conserved energy in Proposition \ref{cons.law.1}. This concludes the proof of the proposition.
\end{proof}

We now turn to the remaining terms in Propositions \ref{KT.2D.2h}, \ref{KS2.2D.2h} and \ref{KS1.2D.2h}, which as we remarked before, are the same as those in Propositions \ref{KT.2D} and \ref{KS2.2D} except for the extra $\pel_3$ weights and integrals. A priori, the fact that we have an extra integration in $\pel_3$ and an extra weight in $\pel_3$ can create extra difficulties. Nevertheless, notice that in the proofs of Propositions \ref{KT.2}, \ref{KS1.2} and \ref{KS2.2}, the only estimate we have used to control the weighted integral in $\pel$ is the interpolation inequality Proposition \ref{prop.interpolation}. On the other hand, in the $2\frac 12$-dimensional case, while the weights depend also on $\pel_3$ and there is an extra integration in $\pel_3$, we have the interpolation inequality inequality Proposition \ref{imp.moment} which has exactly the numerology as in Proposition \ref{prop.interpolation}. Therefore, repeating the proofs of Propositions \ref{KT.2}, \ref{KS1.2} and \ref{KS2.2} and replacing the application of Proposition \ref{prop.interpolation} with that of Proposition \ref{imp.moment}, we achieve the same estimates for $K$ except for the $\tilde K_{S,1}$ term defined in \eqref{tKS1.def} above. More precisely, we have
\begin{proposition}\label{K.est.2h}
For $N>13$ as in the assumption \eqref{ini.bd.2D.2h.2} in Theorem \ref{main.theorem.2D.2h}, there exists $q_2'<\infty$ sufficiently large such that we have the estimates
$$\|K_T\|_{L_t^1([0,T);L^{N+2}_x)}^{N+2}\ls \|f\pel_0^N\|_{L^\infty_t([0,T);L^1_xL^1_{\pel})}^{\frac{N+2}{5(N-1)}}\|f\pel_0^N\|_{L^{q_2'}_t([0,T);L^1_xL^1_{\pel})}^{\frac{4N-7}{5(N-1)}},$$
$$\|K_{S,2}\|_{L_t^1([0,T);L^{N+2}_x)}^{N+2}\ls \|f\pel_0^N\|_{L^\infty_t([0,T);L^1_xL^1_{\pel})}^{\frac{N+2}{5(N-1)}}\|f\pel_0^N\|_{L^{q_2'}_t([0,T);L^1_xL^1_{\pel})}^{\frac{4N-7}{5(N-1)}},$$
and
$$\|K_{S,1}\|_{L_t^1([0,T);L^{N+2}_x)}^{N+2}\ls \|\tilde K_{S,1}\|_{L_t^1([0,T);L^{N+2}_x)}^{N+2}+\|f\pel_0^N\|_{L^\infty_t([0,T);L^1_xL^1_{\pel})}^{\alpha}, 
$$
where $\alpha=\frac{N-4}{2(N-1)}<1$ and the implicit constant depends at most polynomially on $T$.
\end{proposition}

Combining Propositions \ref{tKS1.est} and \ref{K.est.2h}, we have
\begin{proposition}
For $N>13$ as in the initial data bound \eqref{ini.bd.2D.2h.2} in the assumptions of Theorem \ref{main.theorem.2D.2h}, we have the estimate
$$\| f\pel_0^N\|_{L^\infty_t([0,T); L^1_x L^1_\pel)}\leq C_1 e^{C_1T^{k_1}}$$
for some constants $C_1>0$ and $k_1>0$.
\end{proposition}
\begin{proof}
As before, we will allow the implicit constants in $\ls$ to depend at most polynomially on $T$. By Proposition \ref{prop.moment.2h}, we have
$$\| f\pel_0^N\|_{L^\infty_t([0,T); L^1_x L^1_\pel)}\ls\| f_0 \pel_0^N\|_{L^1_x L^1_\pel}
+\|K\|_{L^1_t([0,T); L^{N+2}_x)}^{N+2}.
$$
Then using Propositions \ref{tKS1.est} and \ref{K.est.2h}, we obtain
\begin{equation*}
\begin{split}
&\| f\pel_0^N\|_{L^\infty_t([0,T); L^1_x L^1_\pel)}\\
\ls &1+ \| f\pel_0^N\|_{L^\infty_t([0,T); L^1_x L^1_\pel)}^{\frac{N-4}{2(N-1)}}+\| f\pel_0^N\|_{L^\infty_t([0,T);L^1_xL^1_{\pel})}^{\frac{N+2}{5(N-1)}}\|f \pel_0^N\|_{L^{q_2'}_t([0,T);L^1_x L^1_\pel)}^{\frac{4N-7}{5(N-1)}}.
\end{split}
\end{equation*}
As in the proof of Proposition \ref{propagation.moment} we can divide the equation through by $$\| f\pel_0^N\|_{L^\infty_t([0,T); L^1_x L^1_\pel)}^{\frac{N-4}{2(N-1)}}$$ to get
\begin{equation*}
\begin{split}
\| f\pel_0^N\|_{L^\infty_t([0,T); L^1_x L^1_\pel)}^{\frac{N+2}{2(N-1)}}
\ls &1+\|f \pel_0^N\|_{L^{q_2'}_t([0,T);L^1_x L^1_\pel)}^{\frac{N+2}{2(N-1)}}.
\end{split}
\end{equation*}
We can assume without loss of generality (after using H\"older's inequality and losing a constant depending polynomially on $T$) that $q_2'\geq \frac{N+2}{2(N-1)}$. The conclusion of the proposition thus follows from Lemma \ref{Gronwall}.
\end{proof}
Recall that in the $2$-dimensional case, Proposition \ref{KLinftybd} shows that once the moment bounds are obtained, we can use the bound for the $2$D wave kernel and the interpolation inequality (Proposition \ref{prop.interpolation}) to obtain $L^\infty$ estimates for $K$. Replacing the application of Proposition \ref{prop.interpolation} with that of Proposition \ref{imp.moment}, the exact same proof as Proposition \ref{KLinftybd} gives the following $L^\infty$ bound for $K$, with an estimate of its temporal growth:
\begin{proposition}\label{KLinftybd.2h}
The following $L^\infty$ estimate for $K$ holds:
\begin{equation}\notag
\|K\|_{L^\infty_t([0,T]; L^\infty_x)}\leq C e^{CT^k},
\end{equation}
for some constants $C>0$ and $k>0$.
\end{proposition}

To proceed, we need an analogue of Proposition \ref{dKdecomposition} in the $2\frac 12$-dimensional setting. To this end, we recall\footnote{While all the ideas of the proof are in \cite{GS86}, the precise version that we cite is in part II of the present paper \cite{LS}.}  the following estimates for the \emph{full three dimensional} relativistic Vlasov-Maxwell system:

\begin{proposition}[Glassey-Strauss \cite{GS86}]\label{dKdecomposition.3}
Let $(f,E,B)$ be a solution to the $3$D relativistic Vlasov-Maxwell system.  Then $\nab_x K$ can be decomposed\footnote{We group the terms in a slightly different way from \cite{GS86} and \cite{LS}. In particular, in the present decomposition, only $\nab_x K_0$ depends explicitly on the initial data.} into the following five terms
$$\nab_x K=\nab_x K_0+\nab_xK_{SS}+\nab_xK_{ST}+\nab_xK_{TS}+\nab_xK_{TT}$$
such that\footnote{Here, we use the notation that $dS$ is the standard measure on the round sphere with radius $|y-x|$. Moreover, $d\sigma$ denotes the measure on the cone $C_{t,x}^3=\{(s,y)\in\mathbb R\times\mathbb R^3 |0\leq s\leq t,|y-x|=(t-s)\}$ given by $dS\, ds$.} $\nab_xK_0$ is bounded by the following expressions which depend only on the initial data
$$\frac 1{t}\int_{|y-x|=t}\int_{\mathbb R^3}\pel_0 |\nab_xf_0| d\pel dS,\quad\frac 1{t}\int_{|y-x|=t} |\nab_x^2 K_0| dS,\quad \frac 1{t^2}\int_{|y-x|=t} |\nab_xK_0| dS,$$
$$\frac 1{t^2}\int_{|y-x|=t}\int_{\mathbb R^3}\pel_0 f_0 d\pel dS,\quad \frac 1{t^2}\int_{|y-x|=t}\int_{\mathbb R^3}\pel_0 |K_0| f_0 d\pel dS$$
and the remaining terms obey the following estimates:
$$|\nab_xK_{SS}(t,x)|\ls \int_{C_{t,x}^3} \int_{\mathbb R^3} \frac{\pel_0^3(|K|^2f)(s,y,\pel)}{t-s}d\pel\, d\sigma,$$
\begin{equation*}
\begin{split}
|\nab_xK_{ST}(t,x)|\ls &\int_{C_{t,x}^3} \int_{\mathbb R^3} \frac{\pel_0^3\,\big(|K| f\big)(s,y,\pel)}{(t-s)^2}d\pel\, d\sigma\\
&+\int_{C_{t,x}^3} \int_{\mathbb R^3} \frac{\pel_0^3\,\big((|\nabla_y K|+\rho)f\big)(s,y,\pel)}{(t-s)}d\pel\, d\sigma,
\end{split}
\end{equation*}
$$|\nab_xK_{TS}(t,x)|\ls \int_{C_{t,x}^3} \int_{\mathbb R^3} \frac{\pel_0^3(|K|f)(s,y,\pel)}{(t-s)^2}d\pel\, d\sigma,$$
and for any $\delta\in(0,t)$, we have
\begin{equation*}
\begin{split}
|\nab_xK_{TT}(t,x)|\ls &\int_{C_{t,x}^3\cap\{0\leq s\leq t-\delta\}} \int_{\mathbb R^3} \frac{\pel_0^3\,f(s,y,\pel)}{(t-s)^3}d\pel\, d\sigma\\
&+\int_{C_{t,x}^3\cap\{t-\delta\leq s\leq t\}} \int_{\mathbb R^3} \frac{\pel_0\,|\nabla_{y} f|(s,y,\pel)}{(t-s)^2}d\pel\, d\sigma\\
&+\int_{|y-x|=\delta}\int_{\mathbb R^3} \frac{\pel_0^3\, f(s=t-\delta,y,\pel)}{\delta^2}d\pel\, dS.
\end{split}
\end{equation*}
\end{proposition}

We now note the easy fact that the solution $$(f(t,x_1,x_2,\pel_1,\pel_2,\pel_3),E(t,x_1,x_2), B(t,x_1,x_2))$$ to the $2\frac 12$ dimensional relativistic Vlasov-Maxwell system can be considered as solutions $$(f(t,x_1,x_2,x_3,\pel_1,\pel_2,\pel_3),E(t,x_1,x_2,x_3), B(t,x_1,x_2,x_3))$$ to the $3$ dimensional relativistic Vlasov-Maxwell system which are independent of the $x_3$ variable. Therefore, Proposition \ref{dKdecomposition.3} implies the following proposition:

\begin{proposition}\label{dKdecomposition.2h}
Let $(f,E,B)$ be a solution to the \emph{$2\frac 12$-dimensional} relativistic Vlasov-Maxwell system. Denote $K$ to be $E$ or $B$. Then $\nab_x K$ can be decomposed into the following five terms
$$\nab_x K=\nab_x K_0+\nab_xK_{SS}+\nab_xK_{ST}+\nab_xK_{TS}+\nab_xK_{TT}$$
such that $\nab_xK_0$ is bounded by a constant depending only on the initial data norms \eqref{ini.bd.2D.2h.2} - \eqref{ini.be.2D.2h.7} for $f_0$, $E_0$ and $B_0$
and the remaining terms obey the following estimates:
\begin{equation}\label{dKdecomposition.2h.1}
|\nab_xK_{SS}|\ls \int_0^t\int_{|y-x|\le t-s} \int_{\mathbb R^3} \frac{\pel_0^3 (|K|^2f)(s,y,\pel)}{\sqrt{(t-s)^2-|y-x|^2}}d\pel\, dy\, ds,
\end{equation}
and
\begin{equation}\label{dKdecomposition.2h.2}
\begin{split}
|\nab_xK_{ST}|\ls &\int_0^t\int_{|y-x|\le t-s} \int_{\mathbb R^3} \frac{\pel_0^3\,\big(|K| f\big)(s,y,\pel)}{(t-s)\sqrt{(t-s)^2-|y-x|^2}}d\pel\, dy\, ds\\
&+\int_0^t\int_{|y-x|\le t-s} \int_{\mathbb R^2} \frac{\pel_0^3\,\big((|\nabla_y K|+\rho)f\big)(s,y,\pel)}{\sqrt{(t-s)^2-|y-x|^2}}d\pel\, dy\, ds,
\end{split}
\end{equation}
and further
\begin{equation}\label{dKdecomposition.2h.3}
|\nab_xK_{TS}|\ls \int_0^t\int_{|y-x|\le t-s} \int_{\mathbb R^3} \frac{\pel_0^3(|K|f)(s,y,\pel)}{(t-s)\sqrt{(t-s)^2-|y-x|^2}}d\pel\, dy\, ds,
\end{equation}
and lastly
\begin{equation}\label{dKdecomposition.2h.4}
\begin{split}
|\nab_xK_{TT}|\ls &\int_0^{t-\delta}\int_{|y-x|\le t-s} \int_{\mathbb R^3} \frac{\pel_0^3\,f(s,y,\pel)}{(t-s)^2\sqrt{(t-s)^2-|y-x|^2}}d\pel\, dy\, ds\\
&+\int_{t-\delta}^t\int_{|y-x|\le t-s} \int_{\mathbb R^3} \frac{\pel_0^3\,|\nabla_{y} f|(s,y,\pel)}{(t-s)\sqrt{(t-s)^2-|y-x|^2}}d\pel\, dy\, ds\\
&+\frac{1}{\delta}\int_{|y-x|\le \delta} \int_{\mathbb R^3} \frac{\pel_0^3\,f(s=t-\delta,y,\pel)}{\sqrt{\delta^2-|y-x|^2}}d\pel\, dy.
\end{split}
\end{equation}
\end{proposition}
\begin{proof}
We first note that for every $$F(x_1,x_2,x_3)=G(x_1,x_2),$$ we have
\begin{equation}\label{dKdecomposition.2h.5}
\frac 1{a}\int_{|y-x|=a} F(y) \,d\sigma \ls \int_{|y-x|\leq a} \frac{G(y)}{\sqrt{a^2-|y-x|^2}}\, dy, 
\end{equation}
which in turn follows from an easy calculation. By the symmetry of $f$ and $K$, this immediately implies \eqref{dKdecomposition.2h.1}-\eqref{dKdecomposition.2h.4}. Finally, by \eqref{dKdecomposition.2h.5} and Proposition \ref{dKdecomposition.3}, to show that $\nab_x K_0$ obeys the required bounds, it suffices to show that the terms 
$$\int_{|y-x|\leq t}\int_{\mathbb R^3}\frac{\pel_0 |\nab_xf_0| d\pel\, dy}{\sqrt{t^2-|y-x|^2}},\quad \int_{|y-x|\leq t} \frac{|\nab_x^2 K_0| \, dy}{\sqrt{t^2-|y-x|^2}},\quad \int_{|y-x|\leq t} \frac{|\nab_xK_0| \, dy}{t\sqrt{t^2-|y-x|^2}},$$
$$\int_{|y-x|\leq t}\int_{\mathbb R^3}\frac{\pel_0 f_0 d\pel\, dy\, ds}{t\sqrt{t^2-|y-x|^2}},\quad \int_{|y-x|\leq t}\int_{\mathbb R^3}\frac{\pel_0 |K_0| f_0 d\pel\, dy\, ds}{t\sqrt{t^2-|y-x|^2}}$$
are bounded by a constant depending only on the initial data norms \eqref{ini.bd.2D.2h.2} - \eqref{ini.be.2D.2h.7} for $f_0$, $E_0$ and $B_0$. This can be easily checked by directly controlling the wave kernel and using the Sobolev embedding theorem.
\end{proof}

Using the estimates for $\nab_x K$, we proceed as in the 2-dimensional case to obtain $L^\infty_x$ bounds for $\nab_x K$ and the bounds for the first derivatives of the characteristics. To achieve this, we recall that Lemma \ref{lemm.forw.back} gives the following bound on the derivatives of the backward characteristics in terms of that of the forward characteristics\footnote{which are slightly weaker than in the 2D case, but will not affect the argument.}:
\begin{lemma}\label{lemm.forw.back.2h}
The following estimate holds:
$$\bakC(t)\ls \fowC(t)^4.$$
Here, $\bakC$ and $\fowC$ are defined analogous to \eqref{forw.def} and \eqref{back.def} in the 2-dimensional case.
\end{lemma}

Proposition \ref{dKdecomposition.2h} and Lemma \ref{lemm.forw.back.2h} imply that Proposition \ref{dKLinftybd.2h} below can be proven in an \emph{identical} way as in Proposition \ref{prop.higher.reg} in Section \ref{sec.2D.hr}. More precisely,  the following is the $2\frac 12$D analogue of Proposition \ref{prop.higher.reg}. Notice that from this point onwards, we will allow the implicit constants in $\ls$ to depend \emph{arbitrarily on $T$}.

\begin{proposition}\label{dKLinftybd.2h}
$\nab_x K$ obeys the bound
\bea\label{Kbd.2h}
\|\nab_xK\|_{L^\infty([0,T);L^\infty_x)}\ls 1,
\eea
Moreover, the derivatives of the characteristics $X$ and $V$ are bounded:
\bea\label{dchar.bd.2h}
\sup_{t\in [0,T)}(\bakC(t)+\fowC(t))\ls 1.
\eea
\end{proposition}
By Theorem \ref{theorem.2h.local.existence}, to conclude the proof of global existence and uniqueness in the $2\frac 12$ dimensional case, we need to show that 
$$\|K\|_{L^\infty([0,T);L^\infty_x)}+\|\nab_x K\|_{L^\infty([0,T);L^\infty_x)}+\| w_3 \nab_{x,\pel} f\|_{L^1_t([0,T);L^\infty_x L^2_{\pel})}\ls 1.$$
The first two terms are controlled in Propositions \ref{KLinftybd.2h} and \ref{dKLinftybd.2h} respectively. It thus remains to show that 
$$\| w_3 \nab_{x,\pel} f\|_{L^1_t([0,T);L^\infty_x L^2_{\pel})}\ls 1.$$
This follows immediately from \eqref{Kbd.2h}, \eqref{dchar.bd.2h}, formula \eqref{along.char} in $2\frac 12$D and the fact that $f_0$ obeys the initial bound \eqref{ini.bd.2D.2h.5.5}. This concludes the proof of Theorem \ref{main.theorem.2D.2h}.
\hfill {\bf Q.E.D.}

\end{document}